\documentclass[11pt]{amsart}
\usepackage{hyperref,pb-diagram,amssymb,epic,eepic,verbatim,graphicx,graphics,epsfig,psfrag}
\hypersetup{colorlinks}
\usepackage{color}
\definecolor{darkred}{rgb}{0.5,0,0}
\definecolor{darkgreen}{rgb}{0,0.5,0}
\definecolor{darkblue}{rgb}{0,0,0.5}
\hypersetup{ colorlinks,
linkcolor=darkblue,
filecolor=darkgreen,
urlcolor=darkred,
citecolor=darkblue }
\usepackage{pb-diagram,amssymb,epic,eepic,verbatim,graphicx}
\usepackage{graphics,epsfig,psfrag}
 
\newtheorem{theorem}{Theorem}[section]
\newtheorem{corollary}[theorem]{Corollary}

\newtheorem{proposition}[theorem]{Proposition}
\newtheorem{lemma}[theorem]{Lemma}
\newtheorem{lem}[theorem]{}
\theoremstyle{definition}
\newtheorem{definition}[theorem]{Definition}
\theoremstyle{remark}
\newtheorem{remark}[theorem]{Remark}
\newtheorem{example}[theorem]{Example}

\newcommand{\blem}{\begin{lem} \rm}
\newcommand{\elem}{\end{lem}}

\newcommand\A{\mathcal{A}}
\newcommand\B{\mathcal{B}}
\newcommand\E{\mathcal{E}}
\newcommand\M{\mathcal{M}}

\newcommand\cH{\mathcal{H}}

\renewcommand\M{\mathcal{M}}
\renewcommand\S{\mathcal{S}}

\newcommand{\V}{\mathcal{V}}
\renewcommand{\L}{\mathcal{L}}

\newcommand{\J}{\mathcal{J}}

\newcommand{\U}{\mathcal{U}}

\newcommand{\R}{\mathbb{R}}

\newcommand{\C}{\mathbb{C}}
\newcommand{\cC}{\mathcal{C}}

\newcommand{\Z}{\mathbb{Z}}

\newcommand{\ddt}{\frac{d}{dt}}

\newcommand{\ppth}{\frac{\partial}{\partial \theta}}

\newcommand\lie[1]{\mathfrak{#1}}

\newcommand{\h}{\lie{h}}
\newcommand{\g}{\lie{g}}

\renewcommand{\t}{\lie{t}}
\newcommand{\Alc}{\lie{A}}

\newcommand{\su}{\lie{su}}
\newcommand{\so}{\lie{so}}

\newcommand{\on}{\operatorname}
\newcommand{\ainfty}{{$A_\infty$\ }}

\newcommand{\mycap}{{\on{cap}}}
\newcommand{\cut}{{\on{cut}}}

\newcommand{\Fact}{\on{Fact}}

\newcommand{\dual}{\vee}
\newcommand{\Ab}{\on{Ab}}

\newcommand{\Fun}{\on{Func}}
\newcommand{\Obj}{\on{Obj}}

\newcommand{\graph}{\on{graph}}

\newcommand{\Lag}{\on{Lag}}

\newcommand{\Ad}{ \on{Ad} }

\newcommand{\Hol}{ \on{Hol} } 

\newcommand{\Hom}{ \on{Hom}}
\newcommand{\Ind}{ \on{Ind}}
\renewcommand{\ker}{ \on{ker}}
\newcommand{\coker}{ \on{coker}}

\newcommand{\Vol}{  \on{Vol}}
\newcommand{\diag}{  \on{diag}}

\newcommand{\codim}{\on{codim}}

\newcommand{\ssm}{-}

\newcommand\dirac{/\kern-1.2ex\partial} % Dirac operator
\newcommand\qu{/\kern-.7ex/} % Categorical quotients
\newcommand\lqu{\backslash \kern-.7ex \backslash} % Categorical
\newcommand\dr{r_+ \kern-.7ex - \kern-.7ex r_-}
 
\renewcommand{\comment}[1]   {{}}
\newcommand{\labell}\label

\renewcommand{\d}{{\mbox{d}}}
\newcommand{\ol}{\overline}
\newcommand{\olp}{\ol{\partial}}

\newcommand\Phinv{\Phi^{-1}}

\newcommand\eps{\epsilon}

\newcommand\om{\omega}

\newcommand{\f}{\frac}
\newcommand{\lan}{\langle}
\newcommand{\ran}{\rangle}
\newcommand{\hh}{{\f{1}{2}}}

\newcommand{\ti}{\tilde}

\newcommand\pt{\on{pt}}
\newcommand\cE{\mathcal{E}}

\newcommand\Tr{\on{Tr}}
\newcommand\mE{\mathcal{E}}
\newcommand\mB{\mathcal{B}}

\newcommand\curv{\on{curv}}

\newcommand\Map{\on{Map}}
\newcommand\rank{\on{rank}}
\newcommand\ev{\on{ev}}

\newcommand\Vect{\on{Vect}}
\newcommand\ul{\underline}

\renewcommand\Im{\on{Im}}

\newcommand\reg{{\on{reg}}}
\newcommand\bra[1]{ \lan {#1} \ran} 
\newcommand\bdefn{\begin{definition}}
\newcommand\edefn{\end{definition}}
\newcommand\bea{\begin{eqnarray*}}
\newcommand\eea{\end{eqnarray*}}
\newcommand\bcv{\left[ \begin{array}{r} }
\newcommand\ecv{\end{array} \right] }

\newcommand\bma{\left[ \begin{array} }
\newcommand\ema{\end{array} \right]}
\newcommand\ben{\begin{enumerate}}
\newcommand\een{\end{enumerate}}
\newcommand\beq{\begin{equation}}
\newcommand\eeq{\end{equation}}
\newcommand\bex{\begin{example}}
\newcommand\bsj{\left\{ \begin{array}{rrr} }
\newcommand\esj{\end{array} \right\}}

\newcommand\Cone{\on{Cone}}
\newcommand\Id{\on{Id}}
\newcommand\cI{\mathcal{I}}

\newcommand\eex{\end{example}}
\newcommand\crit{{\on{crit}}}
\newcommand\critval{{\on{critval}}}
\newcommand\sx{*\kern-.5ex_X}

\def\mathunderaccent#1{\let\theaccent#1\mathpalette\putaccentunder}
\def\putaccentunder#1#2{\oalign{$#1#2$\crcr\hidewidth \vbox
to.2ex{\hbox{$#1\theaccent{}$}\vss}\hidewidth}}

\newcommand{\tri}[3]{
\setlength{\unitlength}{0.00037489in}
\begin{picture}(2643,2220)(0,-10)
\path(375,2130)(2175,2130)
\path(2055.000,2100.000)(2175.000,2130.000)(2055.000,2160.000)
\path(1050,330)(150,1889)
\path(235.977,1800.073)(150.000,1889.000)(184.014,1770.075)
\path(2400,1844)(1500,285)
\path(1534.014,403.925)(1500.000,285.000)(1585.977,373.927)
\put(-515,2040){\makebox(0,0)[b]\tiny {#1}}
\put(485,-115){\makebox(0,0)[b]\tiny {#2}}
\put(2490,2040){\makebox(0,0)[b]\tiny {#3}}
\end{picture}
}

\newcommand\one{\setlength{\unitlength}{0.00017489in}
{
\begin{picture}(1824,1839)(-300,1500)
\path(12,1812)(14,1810)(18,1807)
	(25,1800)(36,1790)(52,1775)
	(72,1757)(98,1734)(127,1707)
	(161,1678)(198,1645)(237,1611)
	(278,1576)(320,1541)(362,1505)
	(404,1471)(445,1439)(485,1408)
	(524,1380)(561,1353)(597,1329)
	(631,1307)(663,1288)(694,1271)
	(724,1257)(753,1244)(781,1234)
	(808,1226)(834,1220)(860,1215)
	(886,1213)(912,1212)(938,1213)
	(964,1215)(990,1220)(1016,1226)
	(1043,1234)(1071,1244)(1100,1257)
	(1130,1271)(1161,1288)(1193,1307)
	(1227,1329)(1263,1353)(1300,1380)
	(1339,1408)(1379,1439)(1420,1471)
	(1462,1505)(1504,1541)(1546,1576)
	(1587,1611)(1626,1645)(1663,1678)
	(1697,1707)(1726,1734)(1752,1757)
	(1772,1775)(1788,1790)(1799,1800)
	(1806,1807)(1810,1810)(1812,1812)
\path(12,12)(14,14)(18,17)
	(25,24)(36,34)(52,49)
	(72,67)(98,90)(127,117)
	(161,146)(198,179)(237,213)
	(278,248)(320,283)(362,319)
	(404,353)(445,385)(485,416)
	(524,444)(561,471)(597,495)
	(631,517)(663,536)(694,553)
	(724,567)(753,580)(781,590)
	(808,598)(834,604)(860,609)
	(886,611)(912,612)(938,611)
	(964,609)(990,604)(1016,598)
	(1043,590)(1071,580)(1100,567)
	(1130,553)(1161,536)(1193,517)
	(1227,495)(1263,471)(1300,444)
	(1339,416)(1379,385)(1420,353)
	(1462,319)(1504,283)(1546,248)
	(1587,213)(1626,179)(1663,146)
	(1697,117)(1726,90)(1752,67)
	(1772,49)(1788,34)(1799,24)
	(1806,17)(1810,14)(1812,12)
\end{picture}}}

\newcommand\cupcap{\setlength{\unitlength}{0.00002489in}
{ 
\begin{picture}(1824,839)(-300,500)
\path(12,1812)(14,1810)(18,1807)
	(25,1800)(36,1790)(52,1775)
	(72,1757)(98,1734)(127,1707)
	(161,1678)(198,1645)(237,1611)
	(278,1576)(320,1541)(362,1505)
	(404,1471)(445,1439)(485,1408)
	(524,1380)(561,1353)(597,1329)
	(631,1307)(663,1288)(694,1271)
	(724,1257)(753,1244)(781,1234)
	(808,1226)(834,1220)(860,1215)
	(886,1213)(912,1212)(938,1213)
	(964,1215)(990,1220)(1016,1226)
	(1043,1234)(1071,1244)(1100,1257)
	(1130,1271)(1161,1288)(1193,1307)
	(1227,1329)(1263,1353)(1300,1380)
	(1339,1408)(1379,1439)(1420,1471)
	(1462,1505)(1504,1541)(1546,1576)
	(1587,1611)(1626,1645)(1663,1678)
	(1697,1707)(1726,1734)(1752,1757)
	(1772,1775)(1788,1790)(1799,1800)
	(1806,1807)(1810,1810)(1812,1812)
\path(12,12)(14,14)(18,17)
	(25,24)(36,34)(52,49)
	(72,67)(98,90)(127,117)
	(161,146)(198,179)(237,213)
	(278,248)(320,283)(362,319)
	(404,353)(445,385)(485,416)
	(524,444)(561,471)(597,495)
	(631,517)(663,536)(694,553)
	(724,567)(753,580)(781,590)
	(808,598)(834,604)(860,609)
	(886,611)(912,612)(938,611)
	(964,609)(990,604)(1016,598)
	(1043,590)(1071,580)(1100,567)
	(1130,553)(1161,536)(1193,517)
	(1227,495)(1263,471)(1300,444)
	(1339,416)(1379,385)(1420,353)
	(1462,319)(1504,283)(1546,248)
	(1587,213)(1626,179)(1663,146)
	(1697,117)(1726,90)(1752,67)
	(1772,49)(1788,34)(1799,24)
	(1806,17)(1810,14)(1812,12)
\end{picture}}}

\newcommand\two{\setlength{\unitlength}{0.00017489in}
\begin{picture}(1824,1839)(0,-10)
\path(12,1812)(1812,12)
\path(1812,1812)(1137,1137)
\path(687,687)(12,12)
\end{picture}
}
\renewcommand\sharp{\setlength{\unitlength}{0.00013333in}
\begin{picture}(688,703)(0,-10)
\path(244,644)(244,44)
\path(444,644)(444,44)
\path(44,444)(644,444)
\path(644,244)(44,244)
\end{picture}
}
\newcommand\twist{\setlength{\unitlength}{0.00002489in}
\begin{picture}(1824,1839)(0,-10)
\path(12,1812)(1812,12)
\path(1812,1812)(1137,1137)
\path(687,687)(12,12)
\end{picture}
}
\newcommand\three{\setlength{\unitlength}{0.00017489in}
{
\begin{picture}(1824,1839)(0,-10)
\path(12,1812)(14,1810)(17,1806)
	(24,1799)(34,1788)(49,1772)
	(67,1752)(90,1726)(117,1697)
	(146,1663)(179,1626)(213,1587)
	(248,1546)(283,1504)(319,1462)
	(353,1420)(385,1379)(416,1339)
	(444,1300)(471,1263)(495,1227)
	(517,1193)(536,1161)(553,1130)
	(567,1100)(580,1071)(590,1043)
	(598,1016)(604,990)(609,964)
	(611,938)(612,912)(611,886)
	(609,860)(604,834)(598,808)
	(590,781)(580,753)(567,724)
	(553,694)(536,663)(517,631)
	(495,597)(471,561)(444,524)
	(416,485)(385,445)(353,404)
	(319,362)(283,320)(248,278)
	(213,237)(179,198)(146,161)
	(117,127)(90,98)(67,72)
	(49,52)(34,36)(24,25)
	(17,18)(14,14)(12,12)
\path(1812,1812)(1810,1810)(1807,1806)
	(1800,1799)(1790,1788)(1775,1772)
	(1757,1752)(1734,1726)(1707,1697)
	(1678,1663)(1645,1626)(1611,1587)
	(1576,1546)(1541,1504)(1505,1462)
	(1471,1420)(1439,1379)(1408,1339)
	(1380,1300)(1353,1263)(1329,1227)
	(1307,1193)(1288,1161)(1271,1130)
	(1257,1100)(1244,1071)(1234,1043)
	(1226,1016)(1220,990)(1215,964)
	(1213,938)(1212,912)(1213,886)
	(1215,860)(1220,834)(1226,808)
	(1234,781)(1244,753)(1257,724)
	(1271,694)(1288,663)(1307,631)
	(1329,597)(1353,561)(1380,524)
	(1408,485)(1439,445)(1471,404)
	(1505,362)(1541,320)(1576,278)
	(1611,237)(1645,198)(1678,161)
	(1707,127)(1734,98)(1757,72)
	(1775,52)(1790,36)(1800,25)
	(1807,18)(1810,14)(1812,12)
\end{picture}
}}

\newcommand{\kr}{
\setlength{\unitlength}{0.00087489in}
{
\begin{picture}(82,147)(0,30)
\path(12,100)(41,71)(70,100)
\path(41,71)(41,41)(12,12)
\path(41,41)(70,12)
\end{picture}
}
}

\newcommand{\Diff}{\on{Diff}}
\newcommand\GFuk{{\on{Fuk}^{\sharp}}}
\newcommand\Fuk{{\on{Fuk}}}
\newcommand{\Ob}{\on{Ob}}

\begin{document}

\title[Exact triangle for fibered Dehn twists] {Exact triangle for
fibered Dehn twists}

\author{Katrin Wehrheim} 
\address{
Department of Mathematics,
University of California, Berkeley, 
Berkeley, CA 94720.
{\em E-mail address: katrin@math.berkeley.edu}}

\author{Chris T. Woodward}
\address{Department of Mathematics, 
Rutgers University,
Piscataway, NJ 08854.
{\em E-mail address: ctw@math.rutgers.edu}}

\thanks{Partially supported by NSF grants CAREER 0844188 and DMS
  0904358}

\begin{abstract}  
We use quilted Floer theory to generalize Seidel's long exact sequence
in symplectic Floer theory to fibered Dehn twists.  We then apply the
sequence to construct versions of the Floer and Khovanov-Rozansky
exact triangles in Lagrangian Floer theory of moduli spaces of
bundles.
\end{abstract} 

\maketitle

\tableofcontents

\section{Introduction} 
\label{tri}

Seidel's long exact sequence \cite[Theorem 1]{se:lo} describes the
effect of a symplectic Dehn twist on Lagrangian Floer cohomology.  In
many examples (moduli spaces of bundles, nilpotent slices, etc.) the
relevant fibrations have Morse-Bott rather than Morse singularities
and the associated monodromy maps are {\em fibered} Dehn twists.  Many
years ago Seidel suggested that this sequence should generalize to the
fibered case.  In this paper we show how to carry out this suggestion 
using quilted Floer theory developed in Wehrheim-Woodward
\cite{ww:isom}, \cite{ww:quilts}, \cite{ww:quiltfloer}.  Quilted Floer
theory gives an expression of the third term in the exact triangle as
a push-pull functor, similar to the situation in the analogous
triangle in algebraic geometry developed by Horja \cite{ho:de}.

To state the main result suppose that $M$ is an exact or monotone
symplectic manifold.  If $M$ has boundary $\partial M$ then we assume
that the boundary $\partial M$ is convex so that our spaces of
pseudoholomorphic curves satisfy good compactness properties.  A {\em
  Lagrangian brane} in $M$ is a
\label{addedcompactoriented}
compact, oriented Lagrangian submanifold $L$ equipped with a grading
in the sense of \cite{ww:quiltfloer} and relative spin structure in
the sense of \cite{orient}.  We say that a Lagrangian brane $L$ is
{\em admissible} if $L$ is monotone in the sense of
\cite{ww:quiltfloer}, the image of the fundamental group $\pi_1(L)$ of
$L$ in $\pi_1(M)$ is torsion for any choice of base point and $L$ has
minimal Maslov number at least three.  This notion of admissibility is
chosen so that any pair $L^0,L^1$ of admissible Lagrangian branes in
$M$ is monotone as a pair, which implies an energy-index relation for
pseudoholomorphic strips.  This relation in turn implies that disk
bubbles cannot obstruct the proof of $\partial^2=0$ in the
construction of the {\em Lagrangian Floer homology group}
$HF(L^0,L^1)$.  The underlying complex for this group is generated by
perturbed intersection points of $L^0,L^1$.
The differential counts finite energy holomorphic strips with boundary
in $L^0,L^1$.  Taking the Floer groups as morphism spaces, one obtains
the cohomology of the Fukaya category of Lagrangian branes in $M$.
For the moment, we work with $\Z_2$ coefficients, although the main
result will be stated with $\Z$ coefficients. 

To recall Seidel's exact triangle in Lagrangian Floer cohomology, let
$C \subset M$ be a Lagrangian sphere equipped with an identification
with a unit sphere, and let $\iota_C: C \to C, \ v \mapsto -v$ denote
the antipodal map.  Associated to $C$ is a {\em symplectic Dehn twist}
$$\tau_C: M \to M, \quad \tau_C | C = \iota_C, \quad \on{supp}(\tau_C)
\subset U_C $$
equal to the antipodal map on $C$ and supported in a neighborhood
$U_C$ of $C$.  In the case of a Lefschetz fibration, if $C$ is a
vanishing cycle then $\tau_C$ is the monodromy around the
corresponding critical value.  The assumption $\dim M\ge 4$ in the
following is equivalent to its Lagrangian submanifolds, in particular
the sphere $C$, being of codimension (equal to dimension) at least
$2$.

\begin{theorem}  {\rm (Seidel exact triangle, \cite{se:lo})}  
Let $M$ be a compact monotone or exact symplectic manifold with convex
boundary and dimension at least four, and let $C,L^0,L^1$ be
admissible Lagrangian branes such that $C$ is equipped with a
diffeomorphism to a sphere.  Then there exists a long exact sequence
\vskip .1in
$$ \tri{\hskip -.7in$HF(L^0,L^1)$}{\hskip -.5in $HF(L^0,C) \otimes
  HF(C,L^1)$.}{$HF(L^0,\tau_C L^1)$} $$
\end{theorem}

Seidel's result is often referred to as a categorification of
Picard-Lefschetz since by taking the Euler characteristics one
essentially recovers the Picard-Lefschetz formula as in Arnold
\cite[Chapter I]{arn:sing}.  The Fukaya-categorical version
(conjectured by Kontsevich) is developed in Seidel's book
\cite{se:bo}.

Many interesting fibrations that arise in representation theory or
gauge theory (such as nilpotent slices or moduli spaces of bundles
over a family of curves) have not just Morse singularities but rather
Morse-Bott singularities.  Here the words {\em Morse} and {\em
  Morse-Bott} are used in the sense of non-degeneracy of the Hessian,
which in this setting is a {\em complex} matrix.  Extensions of the
Picard-Lefschetz formula to fibrations $\pi: E \to S$ with more
general singularities are considered by Clemens \cite{cl:pl}, Landman
\cite{la:pl}, and many subsequent authors; the Morse-Bott situation is
a particularly easy case.  Let $s_0 \in S$ be a critical value, and
$$B := \crit(\pi) \cap E_{s_0}$$ 
the critical locus in its fiber, which need no longer consist of
isolated points.  Let $s \in S$ be a generic nearby point.  The analog
of the vanishing cycle in this case is a manifold $C$ fibering over
$B$, 
$$ E_s \supset C \to B $$ 
consisting of points that converge to $B$ under parallel transport,
called by Clemens \cite{cl:pl} the {\em vanishing bundle}.  In the
symplectic setting, $C$ is a coisotropic submanifold and the map to
$B$ is a smooth submersion with maximally isotropic fibers.  The
monodromy $\tau_C$, which arises from parallel transport around $s_0$,
is a {\em fibered Dehn twist} as introduced in Section~\ref{twists}.
Roughly speaking, a fibered Dehn twist is a Dehn twist in each fiber
of a fibered neighbourhood $U \to B$ of $C\subset M$.  Let $M = E_s$
and $c$ the codimension of $C$.  A special case of Clemens
\cite[Theorem 4.4]{cl:pl} gives that the monodromy $\tau_C$ acts on a
homology class $\alpha \in H(M)$ by the formula
\begin{equation} \label{fpl} 
(\tau_C)_*\alpha = \alpha + (-1)^{(c+1)(c+2)/2} [C]\cdot [C^t] \cdot \alpha ,
\end{equation} 
where the action of
$$[C] \cdot : H(B) \to H(M), \quad [C^t] \cdot : H(M) \to H(B) $$
is given by slant products, that is, the images of $[C]$ under
$$H(M \times B) \to \Hom(H(M),H(B)) \ \ \text{resp.} \ \Hom(H(B),H(M)) .$$  
The main result of this paper is a categorification of the fibered
Picard-Lefschetz formula \eqref{fpl} to the setting of Floer-Fukaya
theory.  In other words, we generalize Seidel's triangle to the
fibered case.  As before, need to assume will use suitable
monotonicity conditions to ensure well-defined Floer cohomology and
relative invariants arising from pseudoholomorphic quilts in
\cite{ww:quilts}.  However, we expect our proofs to directly
generalize to a version of the exact triangle in any setting in which
algebraic and analytic refinements provide well-defined Floer
cohomologies.

\begin{definition} \label{sphere} {\rm (Spherically fibered coisotropics)}  
A {\em spherically fibered coisotropic submanifold} of a symplectic
manifold $M$ is a coisotropic submanifold $C \subset M$ of codimension
$c\geq 1$ such that
\begin{enumerate} 
\item {\rm (Fibrating)} the null-foliation of $C$ is fibrating over a
  symplectic base $B$ with fiber $S^c$ a sphere of dimension $c$ and
\item {\rm (Orthogonal structure group)} the structure group of $p: C
  \to B$ is equipped with a reduction to $SO(c+1)$, that is, a
  principal $SO(c+1)$-bundle $P\to B$ and a bundle isomorphism
  $P\times_{SO(c+1)} S^c \cong C$.
\end{enumerate} 
\end{definition}  
\noindent Any spherically fibered coisotropic gives rise to a fibered
Dehn twist $\tau_C \in \Diff(M,\omega)$, see Section \ref{fd}.  We
identify $C$ with its Lagrangian image in $B^- \times M$, where $B^-$
denotes $B$ with symplectic structure reversed.  Let $C^t$ denote the
transpose of $C$ in $M^- \times B$.  Thus $C$ defines {\em Lagrangian
  correspondences} from $B$ to $M$ and vice-versa.  These
correspondences fit into the framework of {\em quilted Floer theory}
developed in \cite{we:co}, \cite{ww:isom}, \cite{ww:quilts},
\cite{ww:quiltfloer}.  Assuming monotonicity as in in \cite{ww:quilts}
(i.e.\ an energy-index relation for pseudoholomorphic quilts), the
correspondence $C$ defines functors from the Fukaya category of $B$ to
that of $M$ and vice versa.  On the level of homology, the latter
functor gives rise to a homomorphism between quilted Floer cohomology
groups
\begin{equation} \label{pantsmap}
HF(\ldots, L^0 ,C^t,C, L^1, \ldots )[\dim(B)] \to HF(\ldots, L^0 ,
L^1, \ldots) ,
\end{equation}
where the Lagrangian correspondences 
$$L^0\subset N_0^-\times M, \quad L^1\subset M^-\times N_1$$ 
are parts of generalized Lagrangian correspondences 
$$\ul{L}^0=(\ldots, L^0), \quad \ul{L}^1=(\ldots, (L^1)^t)$$ 
from a point to $M$ in the sense of \cite{ww:quiltfloer}.  In the
special case of simple Lagrangian submanifolds $\ul{L}^i=L^i\subset
M$, the homomorphism \eqref{pantsmap} is equivalent with $\Z_2$
coefficients to a homomorphism
$$ HF(L^0 \times C,C^t \times L^1)[\dim(B)] \to HF(L^0,L^1) .
$$
The map \eqref{pantsmap} is more precisely defined in \cite{ww:quilts}
and provides the ``quilted chaps'' map which we use to generalize the
``chaps'' map in Seidel's proof of the exact triangle.  In the above
special case, we obtain an exact triangle
\vskip .05in
$$ \tri{\hskip -.7in$HF(L^0,L^1)$}{\hskip -.5in $HF(L^0 \times C,C^t
  \times L^1)[\dim(B)].$}{$HF(L^0,\tau_C L^1)$}
$$
The precise statement including monotonicity conditions, and allowing for generalized Lagrangian correspondences, is the following, which we prove in Section~\ref{exact}.
Here $\on{\graph}(\tau_C) \subset M^- \times M$ denotes the graph of the
fibered Dehn twist $\tau_C$.

\begin{theorem} \label{main} {\rm (Exact triangle for fibered Dehn twists)}   
Let $M$ be a compact monotone or exact symplectic manifold with convex
boundary, let $\ul{L}^0=(\ldots, L^0)$ and $\ul{L}^1=(\ldots,
(L^1)^t)$ be admissible generalized Lagrangian branes in $M$, and let
$C \subset M$ be a spherically fibered coisotropic submanifold of
codimension $c \ge 2$ with base $B$, that is equipped with an
admissible brane structure as a Lagrangian submanifold of $M^-\times
B$.  Then there exists an exact triangle
\vskip .05in
$$ \tri{\hskip -1in$HF(\ldots, L^0 , L^1, \ldots)$}{\hskip -.5in
  $HF(\ldots, L^0 ,C^t,C, L^1, \ldots)[\dim(B)]$.}{$HF(\ldots, L^0  ,\on{graph}(\tau_C), L^1, \ldots )$} $$
\end{theorem} 

A similar triangle was developed by T. Perutz, as part of the program
described in \cite{per:lag}.  A different approach to exact triangles
via Lagrangian cobordism is given in the recent work of Mak-Wu
\cite{mak:dehn} who also treated the codimension one case with $\Z_2$
coefficients for the first time. We also treat the codimension one
case below (see Remark \ref{codimonemain}) although the monotonicity
assumptions required in this case are more complicated.

As in Seidel's work, there is a connection with the mapping cone
construction in the derived Fukaya category, which we establish in
Section \ref{fukaya} as follows.

\begin{theorem}  \label{fukversion} 
{\rm (Derived Fukaya-categorical version of the exact triangle for
  fibered Dehn twists)} Let $C \subset M$ be a spherically fibered
coisotropic submanifold with admissible brane structure as above.
Then for every admissible generalized Lagrangian brane $\ul{L}$ in $M$
there exists an exact triangle in the derived Fukaya category
$$ \tri{\hskip .1in $\ul{L}$}{\hskip -.3in $ \ul{L}\sharp C^t{\sharp}
  C [\dim(B)]$.}{$ \ul{L} \sharp \on{graph}(\tau_C)$}
$$
\end{theorem} 

Here the notation $ \ul{L} \sharp \on{graph}(\tau_C)$ and $
\ul{L}\sharp C^t{\sharp} C$ indicates the generalized Lagrangian
submanifolds in the sense of \cite{ww:quiltfloer} formed by
concatenation $\sharp$.  For $\ul{L}=(L_0,\ldots, L_k)$ with
$L_i\subset M_{i-1}^-\times M_{i}$ and $M_k=M$ the object $ \ul{L}
\sharp \on{graph}(\tau_C)$ is equivalent to $\bigl(L_0,\ldots, ({\rm
  id}_{M_{k-1}}\times \tau_C)(L_k)\bigr)$ for suitable choices of
relative spin structures.  In particular $ L \sharp \on{graph}(\tau_C)
\sim \tau_C(L)$ for $k=0$.  One can also write the bottom object in
the exact triangle
$$L \sharp C^t {\sharp} C [\dim(B)] = \Phi(C)\Phi(C^t)L $$
where $\Phi(C^t), \Phi(C)$ are the \ainfty functors associated to
Lagrangian correspondences constructed in \cite{Ainfty}.  The
formulation of the third term as a push-pull functor makes clear that
the exact triangle is the mirror partner of Horja's exact triangle in
\cite{ho:de}.  We remark that Perutz \cite{per:gys} proves a related
exact triangle describing a symplectic version of the Gysin sequence;
roughly speaking Perutz' result describes the composition of the
functors for $C^t,C$ in the opposite order as a mapping cone for the
map given by multiplication of the Euler class.

We briefly outline the contents of the paper.  Section \ref{fd}
contains background results on fibered Dehn twists and Lefschetz-Bott
fibrations.  Section \ref{flat} describes various situations in which
surface Dehn twists induce generalized Dehn twists on moduli spaces of
flat bundles; these are mostly minor improvements of results of Seidel
and Callahan.  Sections \ref{sections} and \ref{floerversion} contain
the proof of the exact triangle.  Section \ref{field} applies the
triangle to moduli spaces of flat bundles to obtain generalizations of
Floer's exact triangle for surgery along a knot, as well as surgery
exact triangles for crossing changes in knots which have the same form
as the surgery exact triangles as Khovanov \cite{kh:ca} and
Khovanov-Rozansky \cite{kr:ma}.  Finally Section \ref{fukayaversion}
describes generalizations to the \ainfty setting.  These are limited
to the case of minimal Maslov number greater than two.

We thank Mohammed Abouzaid, Tim Perutz, and Reza Rezazadegan for
helpful comments, and especially Paul Seidel for initial discussions
about the project.

The present paper is an updated and more detailed version of a paper
the authors have circulated since 2007. The authors have unreconciled
differences over the exposition in the paper, and explain their points
of view at \cite{katrinpoint} and \cite{chrispoint}.  The publication
in the current form is the result of a mediation.

\section{Lefschetz-Bott fibrations and fibered Dehn twists}
\label{fd} 

This section covers the generalization of the theory of symplectic
Lefschetz fibrations to the Lefschetz-Bott case, that is, to the case
that the singularities of the fibration are not isolated but still
non-degenerate in the normal directions. Most of this material is
covered in an unpublished manuscript of Seidel \cite{se:phd} and in
the works of Perutz \cite{per:lag}, \cite{per:lag2}.  For more recent
appearance of fibered Dehn twists, see Chiang et al
\cite{chiang:open}.

\subsection{Symplectic Lefschetz-Bott fibrations} 

Lefschetz-Bott fibrations have a natural definition in the setting of
holomorphic geometry: one requires the projection to be proper and
Morse-Bott.  In the setting of symplectic geometry, there are several
analogous definitions which we discuss below.  We begin with the
holomorphic setting: Let $S$ be a complex curve.  A {\em Lefschetz
  fibration} over $S$ is a complex manifold $E$ equipped with a proper
holomorphic map $\pi: E \to S$ such that $\pi$ only has critical
points of Morse type.  That is, in local coordinates $z_1,\ldots,
z_{n}$ on $E$, the map $\pi$ is given by
$$ \pi(z_1,\ldots,z_{n}) = \sum_{i=1}^{n} z_i^2 . $$
A {\em Lefschetz-Bott fibration} over $S$ is a complex manifold
equipped with a proper holomorphic map $\pi:E\to S$ that has only
Morse-Bott singularities.  That is, the critical set
$$E^{\crit}=\{e\in E \, | \, D_e\pi = 0 \}$$
is a smooth (necessarily holomorphic) submanifold and the Hessian of
$\pi$ is non-degenerate along the normal bundle of $E^{\crit}$.  By
the parametric Morse lemma \cite[p.12]{ar:si1} for any critical point
$ e \in E^{\crit}$ there exist a neighborhood $U$ of $e$ and
coordinates $(z_1,\ldots,z_n) : U \to \C^n$ such that
\begin{equation} \label{normalform} \pi(z_1,\ldots,z_n) = \sum_{i=1}^{c+1} z_i^2 , \end{equation}
where $n$ is the dimension of $E$ and $c+1$ is the codimension of
$E^{\crit}$ at $e$.

In our examples we will not have global complex structures on $E$ and
$S$ (at least no canonical ones).  Instead, we work with symplectic
versions of Lefschetz-Bott fibrations.  The definition of the
symplectic version uses the following condition introduced in Perutz
\cite{per:lag}.  Let $(M,\omega)$ be a symplectic manifold equipped
with an almost complex structure $J$ and $M' \subset M$ an almost
complex submanifold.  The submanifold $M'$ is said to be {\em normally
  K\"ahler} if a tubular neighborhood $N$ of $M'$ in $M$ is foliated
by $J$-complex normal slices $\{ N_e \subset N\}, e \in M'$, such that
$J | N_e$ is integrable and $\omega | N_e$ is K\"ahler for each $e$.

\begin{definition} \label{fibdef} 
\begin{enumerate} 
\item {\rm (Symplectic fibrations)} A {\em symplectic fibration} is a
  manifold $E$ equipped with a closed two-form
$\omega_E \in \Omega^2(E)$ and a fibration $\pi: E \to S$ over a
  smooth surface $S$, such that the restriction of $\omega_E$ to any
  fiber of $\pi$ is symplectic:
$$ (\omega_E(e) |_{D_e \pi^{-1}(0)} )^{\dim(E)-1} \neq 0, \forall e \in E .$$
\item {\rm (Symplectic Lefschetz-Bott fibrations)} A {\em symplectic
  Lefschetz-Bott fibration} consists of
\begin{enumerate}
\item a smooth manifold $E$ equipped with a closed two-form
  $\omega_E$;
\item a smooth, oriented surface $S$;
\item a smooth proper map $\pi: E \to S$ with critical set  and values
$$E^{\crit} := \{e \in E \ | \ \rank(D_e \pi) < 2 \}, \quad
  S^{\crit}=\pi(E^{\crit})\subset S ;$$
\item a positively oriented complex structure $j_0\in{\rm
  End}(TS|_\U)$ defined in a neighborhood $\U\subset S$ of the
  critical values $S^{\crit}$; and
\item an almost complex structure $J_0\in{\rm End}(TE|_\V)$ defined in a neighborhood $\V\subset E$ of the critical set $E^{\crit}$
\end{enumerate}
satisfying the following conditions:
\begin{enumerate} 

\item $E^{\crit}\subset E$ is a smooth submanifold with finitely many
  components;
\item $E^{\crit}$ is normally K\"ahler;
\item $\pi|_\V: \V \to \U$ is $(J_0,j_0)$ holomorphic;
\item the normal Hessian $D^2\pi_e | T^{\otimes 2}N_e$ at any critical
  point is non-degenerate;
\item $\omega_E$ is non-degenerate on $\ker(D\pi)\subset TE$;
\item $\omega_E|_\V$ is non-degenerate and compatible with $J_0$.
\end{enumerate}
\end{enumerate}  
\end{definition} 

\begin{remark}  One natural reason to consider 
symplectic Lefschetz-Bott fibrations rather than symplectic Lefschetz
fibrations is that the category of Lefschetz-Bott fibrations is
somewhat better behaved than Lefschetz fibrations with respect to
products.  Suppose that $\pi_k: E_k \to S, k = 1,2$ are (symplectic)
Lefschetz-Bott fibrations such that $\critval(\pi_1) \cap
\critval(\pi_2)$ is empty, where $\critval$ denotes the set of
critical values.  Then
$$\pi_1
\times_S \pi_2 : E_1 \times_S E_2 \to S$$ 
is a (symplectic) Lefschetz-Bott fibration with 
$$ \critval(\pi_1 \times_S \pi_2) = \critval(\pi_1) \cup \critval(\pi_2) .$$
In particular, the fiber product of a fibration with a Lefschetz
fibration is a Lefschetz-Bott fibration.
\end{remark} 

Associated to any Lefschetz-Bott fibration there is a natural parallel
transport between the fibers.  The usual notion of parallel transport
in Lefschetz fibrations extends to the symplectic Lefschetz-Bott case.
First, suppose that $\pi: E \to S, \omega_E \in \Omega^2(E)$ is a
symplectic fibration with connected base $S$, for simplicity.  The
canonical symplectic connection on $E$ is the connection defined by
\begin{equation} \label{symconn}
T_e^hE := \bigl( \ker(D_e \pi) \bigr)^{\omega_{E}}  \subset T_e E .\end{equation}
Here the superscript denotes the symplectic complement with respect to
$\omega_{E}$.  The symplectic complement has dimension $2$ due to the
nondegeneracy of $\omega_E|_{\ker(D_e\pi)}$.  Moreover, for any
horizontal vector field $v\in\Gamma(T^hE)$ and fiber $E_s$ we have
$$
\L_v\omega_E|_{E_s} = (\d \iota_v\omega_E + \iota_v \d \omega_E)|_{E_s}
= \d ( \iota_v\omega_E|_{E_s} ) =0.
$$ 
Hence, given any smooth path $\gamma: \ [0,1] \to S \ssm S^{\crit}$
the parallel transport 
$$ \rho_{t,\tau} : E_{\gamma(t)} \to E_{\gamma(\tau)} $$ 
for any $t,\tau \in [0,1]$ is a symplectomorphism. (Also see
\cite[Section 1.2]{gu:sf}.)  This parallel transport gives a reduction
of structure group of the fibration to the symplectomorphism group of
any fiber.

The notion of parallel transport above extends to parallel transport
to critical fibers.  Suppose that $\pi: E \to S$ is a Lefschetz-Bott
fibration.  The smooth locus $E \ssm \pi^{-1}(S^{\on{crit}})$ is a
fibration over $S \ssm S^{\crit}$ with vertical tangent spaces
$T_e^vE=\ker (D_e\pi)$ and a canonical symplectic connection $T^h E
\subset T(E \ssm E^{\on{crit}}) $ defined by \eqref{symconn}.  Given
an embedded path $\gamma:[0,1] \to E$ ending on the critical locus
$\gamma(1) \in S^{\crit}$ parallel transport extends to a continuous
map
\begin{equation}
 \label{limrho} 
\rho_{t,1}: E_{\gamma(t)} \to E_{\gamma(1)}, \quad x \mapsto
\lim_{\tau \to 1} \rho_{t,\tau}(x) .\end{equation}
Indeed, choose a tubular neighborhood of $\gamma$.  After rescaling,
parallel transport becomes the gradient flow of the function $f \circ
\pi$ for some coordinate function $f: S \to \R$ with respect to the
metric $\omega_E(\cdot, J_0 \cdot)$, see \cite[Lemma 1.13]{se:lo}.
Since the critical points of $\pi$ are Lefschetz-Bott, the gradient
flow is hyperbolic and the limit is well-defined \cite{sh:gl}.

Vanishing thimbles and cycles for Lefschetz-Bott fibrations are
defined as follows.  Let $\gamma: \ [0,1] \to S$ be a smooth embedded
path with $\gamma(1) \in S^{\crit}$ such that $\gamma([0,1))\subset S
  \ssm S^{\crit}$.  Fix a connected component $B\subset E^{\crit} \cap
  E_{\gamma(1)}$ of the critical set in the endpoint fiber.  The {\em
    vanishing thimble} for the path $\gamma$ and component $B$ is
$$ 
T_{\gamma,B} = \Bigl\{ x \in \cup_{t\in[0,1)} E_{\gamma(t)} \,\Big| \,    \rho_{t,1}(x) \in B \Big\} \cup B .
$$
The vanishing thimble $T_{\gamma,B}\subset E$ is a smooth submanifold
with boundary since it is the stable manifold of $B$; see
\cite{sh:gl}.  The intersections
\begin{equation} \label{Ct} C_t := T_{\gamma,B}\cap E_{\gamma(t)} \end{equation}
with the smooth fibers of $\pi$ for $t\in[0,1)$ are called the {\em
    vanishing cycles} for the path $\gamma$.

\begin{proposition}  Each vanishing cycle $C_t$ from \eqref{Ct} is a coisotropic
submanifold of the fiber $E_{\gamma(t)}$.  The map $\rho_{t,1}:C_t \to
B$ is smooth and gives $C_t$ the structure of a spherically fibered
coisotropic submanifold in the sense of Definition~\ref{sphere}.
\end{proposition} 

\begin{proof}  
The parallel transport map $\rho_{t,1}$ of \eqref{limrho} is a smooth
fibration with fibers $c$-dimensional spheres and structure group
$SO(c+1)$ by the stable manifold theorem, see for example
\cite{sh:gl}.  The dimension $c$ is the dimension of the fiber as well
as the codimension of $C_t\subset E_{\gamma(t)}$, by the normal form
\eqref{normalform} of $\pi$.  The parallel transport can also be
written as a rescaled Hamiltonian flow of $g \circ \pi$ for some
coordinate function $g:S\to\R$, as in the unfibered case described in
\cite{se:lo}.  Since the Hamiltonian flow preserves the symplectic
form, the symplectic form vanishes on the fibers of $C_t$:
$$ \omega_E |_{\rho_{t,1}^{-1}(b)} = 0, \quad \forall b \in B .$$
Since $\pi$ is $J_0$-holomorphic on $B$, the tangent space $TB$ (which
is the null space of the Hessian of $\pi$) is $J_0$-invariant.  Hence
nondegeneracy of $\omega_E$ in a neighborhood of $B$ implies that the
restriction of $\omega_E$ to $B$ is symplectic.  Thus the total space
$C_t$ is coisotropic and the projection $\rho_{t,1}: C_t \to B$ is the
null foliation as claimed.
\end{proof}  

\subsection{Fibered Dehn twists}
\label{twists}

The symplectic Dehn twist along a Lagrangian sphere in \cite{se:lo}
can be generalized to spherically fibered coisotropics using the
associated symplectic fiber bundle construction.  This construction
associates to any principal bundle and Hamiltonian action with small
moment map a symplectic fiber bundle.

We first set up notation for connections on principal bundles.  Let
$G$ be a compact Lie group with Lie algebra $\g$.  Recall that a {\em
  connection one-form} on a principal $G$-bundle $P$ is a one-form
$\alpha \in \Omega^1(P,\g)$ with values in the Lie algebra $\g$
satisfying the following two conditions:
\begin{enumerate} 
\item the evaluation $\alpha(\xi_P)$ is the constant function equal to
  $\xi$ for any $\xi \in \g$, where $\xi_P \in \Vect(P)$ is the vector
  field generating the action of $\xi$, and
\item the pull-back $ g^* \alpha$ is equal to $ \Ad(g)^{-1} \alpha$
  for any $g \in G$, where the adjoint action is on the values of
  $\alpha$.
\end{enumerate} 
Let $\ker(\alpha) \subset TP$ denote the {\em horizontal subbundle} of
vectors whose fiber at a point point $p \in P$ is the space of vectors
annihilated by $\alpha_p$.  The tangent bundle $TP$ admits a splitting
into horizontal and vertical parts
$$ TP \cong \ker(\alpha) \oplus \ker (D \pi) $$
invariant under the group action.  Denote by $\A(P)$ the space of
connection one-forms on $P$; it is an affine space modelled on
$\Omega^1(C,P(\g))$.

Given a principal bundle with connection and a Hamiltonian action, the
associated fiber bundle has a natural two-form defined as follows.  To
set up notation, suppose that the following are given:
\begin{itemize}
\item a {\em base symplectic manifold} $(B,\omega_B)$;
\item a {\em principal bundle} $\pi: P \to B$ with structure group
  $G$;
\item a {\em fiber symplectic manifold} $(F,\omega_F)$ equipped with a
  Hamiltonian $G$-action with moment map $\Phi_F: F \to \g^\dual :=
  \Hom(\g,\R)$; and
\item a {\em connection one-form} $\alpha \in \Omega^1(P,\g)^{G}$.
\end{itemize} 
Denote by $\pi_1,\pi_2$ the projections to the factors of $P \times
F$.  The {\em minimally coupled form} on $P \times F$ is
\begin{equation} \label{mincoup}
\omega_{P \times F,\alpha} = \pi_1^* \pi^*\omega_B + \pi_2^* \omega_F
+ \d \lan \pi_1^* \alpha, \pi_2^* \Phi_F \ran \in \Omega^2(P \times F)
.\end{equation}

\begin{theorem} 
\label{assoc}
{\rm (Symplectic associated fiber bundles)} (see e.g.  \cite{gu:sy})
Let $P$ be a principal $G$-bundle, $F$ a Hamiltonian $G$-manifold and $\alpha$ a connection
one-form as above. 
\begin{enumerate} 
\item The minimally coupled form $\omega_{P \times F,\alpha}$ descends
  to a closed two-form 
$$\omega_{P(F),\alpha} \in \Omega^2(P(F)), \quad P(F)  := P \times_G
  F.$$
The form $\omega_{P(F),\alpha}$ is non-degenerate on the fibers in a
neighborhood
$$ P(F)_\eps:= P(
  \{ \vert \Phi_F \vert < \eps \})$$ 
of the associated bundle to the zero level set $P(\Phinv_F(0)) \subset
P(F)$.  Here $\vert \Phi_F \vert$ denotes the norm of $\Phi_F$ with
respect to an invariant metric on $\g \cong \g^\dual$.  Hence
$P(F)_\eps$ is a symplectic fiber bundle over $B$ for sufficiently
small $\eps > 0 $.
\item If $\Phinv_F(0)$ is smooth, then $P(\Phinv_F(0))$ is a smooth
  submanifold of $P(F)_\eps$ with coisotropic fibers.
\item Given two choices of connection $\alpha_j, j = 0,1$, there exists
  an isomorphism of symplectic fiber bundles from
  $(P(F)_\eps,\omega_{P(F),\alpha_0})$ to $(P(F)_\eps,
  \omega_{P(F),\alpha_1})$ for sufficiently small $\eps > 0 $.
\item The association $(F,\omega_F,\Phi_F) \to
  (P(F)_\eps,\omega_{P(F)})$ is functorial in the sense that any
  isomorphism of Hamiltonian $G$-manifolds
  $(F_0,\omega_{F_0},\Phi_{F_0})$ to $(F_1,\omega_{F_1}, \Phi_{F_1})$
  induces an isomorphism of symplectic fiber bundles $P(F_0)_\eps \to
  P(F_1)_\eps$.
\end{enumerate} 
\end{theorem}  

\begin{example} \label{functor}
{\rm (Associated bundles with cotangent-sphere fibers)} We are mainly
interested in the following special case of the general construction.
For any integer $c \ge 1$, let $S^c$ denote a sphere of dimension $c$
and $T^\dual S^c$ its cotangent bundle.  Consider $T^\dual S^c$ with
canonical symplectic form $\omega_{T^\dual S^c}$ and the canonical
$SO(c+1)$-action.  The action is Hamiltonian with a moment map $
\Phi_{T^\dual S^c}$ whose zero level set is $S^c$. Thus for any
principal $SO(c+1)$-bundle $\pi: P \to B$ the associated fiber bundle
construction yields a symplectic fiber bundle $P(T^\dual S^c)_\eps$
over $B$.  By functoriality, any automorphism $\tau$ of $(T^\dual
S^c,\omega_{T^\dual S^c},\Phi_{T^\dual S^c})$ induces a bundle
isomorphism $\tau_{P(T^\dual S^c)}: P(T^\dual S^c) \to P(T^\dual
S^c)$. The latter is an isomorphism of symplectic fiber bundles on
$P(T^\dual S^c)_\eps$.
\end{example} 

The notion of Dehn twist is most familiar from Riemann surface theory,
where a Dehn twist denotes a diffeomorphism obtained by twisting
around a circle on a handle.  In \cite{se:lo} Seidel introduces a
generalized notion of Dehn twist which is a symplectomorphism around a
Lagrangian sphere, called a {\em generalized Dehn twist}.  The
symplectomorphisms we consider here are further generalized by
allowing the twists to be fibered, so that the vanishing cycles are
fibered coistropics.  To save space, we call these simply fibered Dehn
twists.  We begin with the local model introduced by Seidel
\cite{se:lo}.

\begin{definition} {\rm (Model Dehn twist, Seidel \cite[Lemma 1.8]{se:lo})} 
For any sphere $S^c$, a {\em model Dehn twist} $\tau_{S^c}$ along the
zero section $S^c$ in the cotangent bundle $T^\dual S^c$ is a
compactly supported symplectomorphism equal to the antipodal map on
$S^c$, given as follows.  Let $\zeta \in C^\infty(\R)$ be a function
satisfying
\begin{equation} \label{rt}
 \zeta(t) = 0 \quad\text{for}\;t \geq \eps
 \qquad\text{and}\qquad 
\zeta(-t) = \zeta(t) - t  \quad\text{for all } t \in \R.\end{equation}
In particular, $\zeta$ is compactly supported and $\zeta'(0) = 1/2$.
Fix the standard Riemannian structure on $S^c$ and let
$$  T^\dual  S^c \to \R, \quad v \mapsto \vert v \vert $$
denote the Riemannian norm.  The norm is smooth on the complement of
the zero section and similarly for composition with $\zeta$.  The time
$2 \pi$ flow of $v \mapsto \zeta(|v|)$ extends to a smooth
symplectomorphism $\tau_{S^c}$ of $T^\dual S^c$.  The flow acts on the
zero section by the antipodal map given by
$$\tau_{S^c} |_{S^c} : S^c \to S^c, \quad v \mapsto - v .$$
Furthermore $\tau_{S^c}$ is $SO(c+1)$-equivariant and preserves the
moment map for the $SO(c+1)$-action.  Any two model Dehn twists given
by different choices of $\zeta$ are equivalent up to symplectomorphism
generated by a compactly supported Hamiltonian.
\end{definition} 

We construct fibered Dehn twists along spherically fibered
coisotropics as follows.  

\begin{definition}\label{model}  
Let $(M,\omega)$ be a symplectic
  manifold and $C\subset M$ a spherically fibered coisotropic
  submanifold of codimension $c\geq 1$, fibering $C\to B$ over a
  symplectic manifold $(B,\omega_B)$, as in Definition \ref{sphere}.
\begin{enumerate}
\item {\rm (Coisotropic embedding)} 
  Recall that $C$ is diffeomorphic to an associated fiber bundle
  $P(S^c) := P \times_{SO(c+1)} S^c ,$ for some principal
  $SO(c+1)$-bundle $\pi: P \to B$.  By the coisotropic embedding
  theorem \cite[p.315]{gu:sy}, a neighborhood of $C$ in $M$ is
  symplectomorphic to a neighborhood of the zero section in $P(T^\dual
  S^c)_\eps$ as in Theorem \ref{assoc}.
\item {\rm (Model fibered Dehn twists)} Any $SO(c+1)$-equivariant
  model Dehn twist $\tau_{T^\dual S^c}: T^\dual S^c \to T^\dual S^c$
  induces a symplectomorphism
$$\tau_{P(T^\dual S^c)} : P(T^\dual S^c)_\eps \to P(T^\dual
  S^c)_\eps$$ 
by functoriality of the associated symplectic fiber bundle
construction as in Example \ref{functor}.  Given a
symplectomorphism $\phi$ of a neighborhood $U$ of $C\subset M$ with
$P(T^\dual S^c)_\eps$ we define a symplectomorphism $ \tau_C : M \to M
$ by $\tau_{ P(T^\dual S^c)}$ on the neighborhood of $C$ and the
identity outside:
$$ \tau_C|_U = \phi^{-1} \tau_{T^\dual P(S^c)}|_{\phi(U)} \phi, \quad
  \tau_C |_{M - U}= \on{Id}_M .$$
We call $\tau_C$ a {\em model fibered Dehn twist along $C$}.
\item {\rm (Fibered Dehn twists)} A symplectomorphism $\tau_C$ of $M$
  is called a {\em fibered Dehn twist along $C$} if $\tau_C$ is
  Hamiltonian isotopic to a model Dehn twist.
\end{enumerate} 
\end{definition} 

\begin{remark} 
The Hamiltonian isotopy class of a fibered Dehn twist is independent
of the choice of local model and fibered Dehn twist used in its
construction: Any two local models for a fibered coisotropic are
isotopic, by a family version of Moser's construction.  This fact
implies that any two fibered Dehn twists $\tau_{C,0}, \tau_{C,1}$
defined using different local models and model twists may be connected
by a family $\tau_{C,t}$.  The vector field $ v_t : = (\tau_{C,t}^{-1}
)_* \ddt \tau_{C,t} \in \Vect(M)$ vanishes on $C$ and is necessarily
Hamiltonian in a neighborhood of $C$.
\end{remark}

\subsection{Equivariant fibered Dehn twists}
\label{sec:twistred} 

In this section we discuss the interaction of equivariant fibered Dehn
twists with symplectic reduction.  Let $G$ be a compact connected Lie
group and $(M,\om)$ a symplectic $G$-manifold.  A {\em spherically
  fibered coisotropic $G$-submanifold} is an invariant coisotropic
submanifold $C \subset M$ of codimension $c\geq 1$ such that there
exists
\begin{itemize} 
\item a principal $SO(c+1)$-bundle $\pi: P\to B$ equipped with an
  action of $G$ by bundle automorphisms (i.e. $SO(c+1)$-equivariant
  diffeomorphisms) and
\item a $G$-equivariant bundle isomorphism $P\times_{SO(c+1)} S^c
  \cong C$, where the $G$-action is induced by the action on the first
  factor.
\end{itemize} 
Given such a coisotropic $C$, one obtains a $G$-equivariant Dehn twist
on $P(T^\dual S^c)$ by the $G$-equivariant version of the associated
symplectic fiber bundle construction.  One obtains a {\em
  $G$-equivariant model fibered Dehn twist} on $M$ via the
$G$-equivariant coisotropic embedding theorem.  A $G$-equivariant
symplectomorphism $\tau \in \Diff(M,\omega)$ is a {\em $G$-equivariant
  fibered Dehn twist} if $\tau$ is equivariantly Hamiltonian isotopic
to a model Dehn twist, that is, Hamiltonian isotopic via a
symplectomorphism generated by a family of {\em $G$-invariant}
Hamiltonians.

We now show that equivariant fibered Dehn twists give rise to fibered
Dehn twists in symplectic quotients. Let $({M},{\om})$ be a
Hamiltonian $G$-manifold with moment map ${\Phi}: M \to \g^\dual$.
The symplectic quotient of $M$ by $G$ is
$$M\qu G:={\Phi}^{-1}(0)/G .$$  
Assuming that $G$ acts freely on $\Phi^{-1}(0)$, $M \qu G$ is a
symplectic manifold with unique symplectic form that lifts to the
restriction of the symplectic form to $\Phi^{-1}(0)$.

\begin{lemma} \label{via} 
 Let $({M},{\om})$ be a Hamiltonian $G$-manifold with moment map
 ${\Phi}$.  Let ${C} \subset {M}$ be a $G$-invariant coisotropic
 submanifold.  Suppose that $0$ is a regular value of $\Phi$, the
 action of $G$ on $\Phi^{-1}(0)$ is free, and ${C}$ intersects
 ${\Phi}^{-1}(0)$ transversally. The quotient
$$ C\qu G:=({C} \cap {\Phi}^{-1}(0))/G $$ 
is a coisotropic submanifold of $M \qu G$.
\end{lemma} 

\begin{proof} 
 The transversality $TC\pitchfork \ker D \Phi$ implies $(\ker D
 \Phi)^\omega \cap TC^\omega=\{0\}$.  It follows that 
$$(\ker D \Phi \cap TC )^\omega = (\ker D \Phi)^\omega + TC^\omega =
 (\ker D \Phi)^\omega \oplus TC^\omega .$$
Now $(\ker D \Phi)^\omega\cong\g$ is the tangent space to the
$G$-orbits and contained in $TC$.  On the other hand, $TC^\omega
\subseteq TC$ since $C$ is coisotropic.  Hence
$$ (\ker D \Phi \cap TC)^\omega \cap (\ker D
\Phi) \subseteq \ker D \Phi \cap TC .$$  
This implies $T(C \qu G)^\omega \subseteq T(C \qu G)$.
\end{proof} 

\begin{lemma} \label{induced} 
Suppose that $C \subset M$ is a spherically fibered $G$-coisotropic
over a base $B$ where $M$ is a Hamiltonian $G$-manifold with moment
map $\Phi$.  Then $\Phi$ is constant on the fibers of $C$ and the
induced action of $G$ on $B$ is Hamiltonian with moment map $\Phi_B:B
\to \g^\dual$ the unique map satisfying $p^* \Phi_B = \Phi | C$.
\end{lemma} 

\begin{proof} By assumption, the action of $G$ on $P$ is 
$SO(c+1)$-equivariant and so induces an action on $B$.  For any
  $\xi\in \g$ the infinitesimal action $\xi_M\in \Vect(M)$ is tangent
  to $C$.  Hence $L_v \lan {\Phi},\xi \ran = {\omega}(\xi_M,v) = 0$
  for all fiber tangent vectors $v \in T^{\on{vert}} {C}=TC^{\omega}$.
  It follows that $\Phi$ is constant on the fibers of $p: C \to B$.
  So $\Phi$ induces a map $\Phi_B: B \to \g^\dual$, satisfying $\d
  \lan \Phi_B, \xi \ran = \iota(\xi_B) \omega_B$ for all Lie algebra
  vectors $\xi \in \g$ as claimed.
\end{proof}  

\begin{remark}
 {\rm (Quotients of spherically fibered coisotropics are spherically
   fibered)} Suppose that in the setting of Lemma \ref{induced}, $G$
 also acts freely on $\Phi_B^{-1}(0)$.  It follows that $B\qu
 G=\Phi_B^{-1}(0)/G$ is a smooth symplectic quotient. Then the null
 foliation on $C\qu G$ fibers over $B\qu G$: If $p: C \to B$ is the
 projection then
$$ C\qu G :=({C} \cap {\Phi}^{-1}(0))/G = p^{-1}( {\Phi_B}^{-1}(0)) / G
\;\overset{p}{\longrightarrow} \; {\Phi_B}^{-1}(0) / G =: B \qu G .
$$ 
Define $\ti{P} = ({P} | {\Phi_B}^{-1}(0) )/G$; this quotient is a
principal $SO(c+1)$-bundle over $\ti{B}=B \qu G$.  We have a bundle
isomorphism $\ti{P}\times_{SO(c+1)} S^c \cong C\qu G$.  It follows
that $C\qu G$ is spherically fibered in the sense of
Definition~\ref{sphere}.
\end{remark} 

In this setting every $G$-equivariant fibered Dehn twist along $C$
descends to a fibered Dehn twist of $M\qu G$ along $C\qu G$:

\begin{theorem} \label{thm fibered quotients}
Let $({M},{\om})$ be a Hamiltonian $G$-manifold with moment map
${\Phi}$ such that $0$ is a regular value of $\Phi$ and the action of
$G$ on $\Phi^{-1}(0)$ is free.  Let ${C} \subset {M}$ be a spherically
fibered $G$-coisotropic over a base $B$. Suppose that $C$ intersects
${\Phi}^{-1}(0)$ transversally, and that the induced action of $G$ on
the base $\Phi_B^{-1}(0)\subset B$ is free.  Let $\tau_C \in
\Diff(M,\omega)$ be a $G$-equivariant fibered Dehn twist along $C$.
Then the symplectomorphism
$$[\tau_C]: M\qu G \to M\qu G, \quad [m] \mapsto [\tau_C(m)]$$ 
is a fibered Dehn twist $[\tau_C]=:\tau_{C\qu G}$ along $ C\qu G$.
\end{theorem}

\begin{proof}  By definition $\tau_C$ is Hamiltonian isotopic to 
a model Dehn twist on $P(T^\dual S^c)$ given by $\tau^0_C:[{p},v]
\mapsto [{p},\tau_{S^c}(v)]$.  The latter is $G$-equivariant since the
$G$-action commutes with the $SO(c+1)$ action.  Any $G$-equivariant
local model $P(T^\dual S^c) \to M$ induces a local model given by a
symplectomorphism of a neighborhood of the zero section in
$\ti{P}(T^\dual S^c)$ to $M \qu G$.  One obtains from the local model
a Dehn twist
$$\tau^0_{C\qu G}:[\ti{p},v] \mapsto [\ti{p},\tau_{S^c}(v)]$$ 
on $\ti{P}(T^\dual S^c)$ along $C\qu G$.  The equivariant Hamiltonian
isotopy of $\tau_C$ to $\tau_C^0$ induces a Hamiltonian isotopy of
$\tau_{C \qu G}$ to $\tau_{C \qu G}^0$.  This completes the proof.
\end{proof}

\subsection{Lefschetz-Bott fibrations associated to fibered Dehn twists} 
\label{standard}

In this section we explain that any fibered Dehn twist appears as the
monodromy of a symplectic Lefschetz-Bott fibration.  Conversely, the
monodromy of a symplectic Lefschetz-Bott fibration is given by a
fibered Dehn twist up to isotopy by Theorem \ref{isafdt} of Perutz
\cite{per:lag} recalled below.  (Theorem \ref{isafdt} is not used in
this paper; we mention it only for its conceptual importance linking
Lefschetz-Bott fibrations and fibered Dehn twists.)

\begin{proposition} \label{ECexist} 
 Let $M$ be a symplectic manifold, $C \subset M$ a spherically fibered
 coisotropic, and $\tau_C: M \to M $ a fibered Dehn twist around $C$.
 There is a {\em standard Lefschetz-Bott fibration} $E_C$ with generic
 fiber $M$ and symplectic monodromy $\tau_C$.
\end{proposition} 

\begin{proof}  
Let $p: C \to B$ denote the fibration, and $P \to B$ the associated
$SO(c+1)$-bundle.  Recall from \cite{se:lo} the symplectic Lefschetz
fibration associated to a model Dehn twist.  Given the standard
representation of $SO(c+1)$ on $V = \C^{c+1}$ we have a vector bundle
$$P(V) := (P \times V)/SO(c+1) \to B .$$  
Let 
$$ \omega_{V} \in \Omega^2(V), \quad \Phi_{V}: V \to \so(c+1)^\dual $$
denote the symplectic form and moment map for the $SO(c+1)$-action
induced from the identification $V \to T^\dual \R^{c+1}$.  The
associated symplectic fiber bundle construction above produces a
closed form $\omega_{P(V)}$ on $P(V)$, equal to $\omega_B$ on $B$ and
equal to $ \omega_{V}$ on the fiber $ V$.  The map
\begin{equation} \label{piV}
 \pi_V: \C^{c+1} \to \C,  \quad (z_0,\ldots,z_c) \mapsto \sum_{i=0}^c
 z_i^2 \end{equation}
is $SO(c+1)$-invariant and has a single Morse singularity.  By the
associated bundle construction $\pi_V$ induces the structure of a
Lefschetz-Bott fibration on a neighborhood of the zero section of
$P(V)$ over $\C$.  Let $S^c \subset \R^{c+1} \subset V$ denote the
unit sphere and
\begin{equation} \label{Sig} T_z := \sqrt{z}S^c \subset V, \ \ \ T := \bigcup_{z \in \C}
T_z .\end{equation}
By \cite[(1.17)]{se:lo} the symplectic form on $V$ can be changed
slightly so that the symplectic monodromy around $0$ is a Dehn twist
along $S^c$.  By \cite[1.13]{se:lo} there exists an isomorphism of
Hamiltonian $SO(c+1)$-manifolds
$$ V \ssm T \to \C \times ( T^\dual  S^{c} \ssm S^{c} ).$$
By the coisotropic embedding theorem, a neighborhood of $C$ in $M$ is
symplectomorphic to the fiber bundle $P(U)$, where $U$ is a
neighborhood of the zero section in $T^\dual  S^{c}$.  It follows that $
P(V) \ssm P(T)$ is symplectomorphic to $P(\C \times T^\dual  S^c \ssm S^c)$
in a neighborhood of $P(T)$ resp.\ $P(\C \times S^c)$.  By replacing a
neighborhood of $\C \times C$ in $\C \times M$ with a neighborhood of
$P(T)$ in $P(V)$, one obtains a Lefschetz-Bott fibration $E_C \to \C$
with monodromy $\tau_C$.
\end{proof}

\begin{theorem} \label{isafdt}
\cite[Theorem 2.19]{per:lag} Suppose that $\pi: E \to S$ is a
symplectic Lefschetz-Bott fibration, and $s_0 \in S^{\crit}$ is such
that $\pi^{-1}(s_0) \cap E^{\crit}$ has a unique connected component.
Then the symplectic monodromy around $s_0$ is a fibered Dehn twist.
\end{theorem} 

\subsection{Further examples of fibered Dehn twists} 
\label{collective}

Fibered Dehn twists are often induced by flows of components of moment
maps. First let $U(1) = \{ z \in \C \ | \ |z| = 1 \}$ denote the
circle group.  We identify the Lie algebra of $U(1)$ with $\R$ via
division by $2 \pi i$.  The integers $\Z = \exp^{-1}(1)$ are then the
coweight lattice.  Let $(M,\omega)$ be a symplectic manifold, and $M_0
\subset M$ an open subset equipped with a free Hamiltonian action of
$U(1)$ with moment map $\Phi: M_0 \to (c_-,c_+)$.

\begin{proposition} \label{circle} {\rm (Fibered Dehn twists via Hamiltonian circle actions)}  
Let $\psi \in C^\infty[c_-,c_+]$ be a function such that $\psi' = 1$
on a neighborhood of $c_+$ and $\psi' = 0 $ on a neighborhood of
$c_-$.  Then the time one flow of $ \psi \circ \Phi $ on $M_0$ extends
to a smooth flow on $M$ equal to the identity on the complement of
$M_0$ and the extension is a fibered Dehn twist along $\Phinv(c)$ for
any $c \in (c_-,c_+)$.  If $M_0$ is a Hamiltonian $G$-manifold for a
compact Lie group $G$ and $\psi$ is $G$-invariant then this fibered
Dehn twist is equivariant.
\end{proposition}  

\begin{proof}  Let $\psi$ be as in the statement of the Proposition.  
We have
$$ \d (\psi \circ \Phi) = \psi' \d \Phi = \iota \left(\psi' \ppth
\right) \omega \in \Omega^1(M_0),$$
where $\ppth \in \Vect(M_0)$ is the generating vector field for the
action.  Hence the flow of $\psi \circ \Phi$ is rotation by the angle
%\psi'
$2\pi \psi'$.  The rest is immediate from the definitions.
\end{proof}  

\begin{remark} {\rm (Standard Dehn twists as symplectic Dehn twists)}  
Standard Dehn twists of complex curves arise from the construction in
\ref{circle} as follows.  Suppose that $M$ is a complex curve and $C
\to M$ an embedded circle equipped with an identification $C \cong
S^1$.  Choose an area form $\omega_M$ on $M$.  Since $C \subset M$ is
Lagrangian, by the Lagrangian embedding theorem there exists a tubular
neighborhood $M_0 = C \times (c_-,c_+)$ on which the symplectic form
is standard.  Then the $U(1)$ action by rotation on the left factor of
$M_0$ is free, and the projection $\Phi$ on the second factor is a
moment map.  For any $\psi$ with the properties in Proposition
\ref{circle}, the flow of $\psi \circ \Phi$ is a standard Dehn twist.
\end{remark} 

Next we consider Dehn twists induced by flows of moment maps of
$SU(2)$-actions. We fix a metric on the Lie algebra $\su(2)$ so that
non-zero elements $\xi$ with $\exp(\xi) = 1$ have minimal length $1$.

\begin{proposition} \label{su2twist}
{\rm (Fibered Dehn twists via Hamiltonian $SU(2)$-actions)} Suppose
that $(M,\omega,\Phi)$ is a Hamiltonian $SU(2)$-manifold with moment
map $\Phi: M \to \su(2)^\dual \cong \su(2)$ and the stabilizer $H$ of
the action of $SU(2)$ on any point in $\Phinv(0)$ is trivial resp.
isomorphic to $U(1)$.  Let $\psi \in C^\infty_c[0,\infty)$ be a
  compactly-supported function such that $\psi' = 1/2$ in a
  neighborhood of $0$.  Then $\Phinv(0)$ is a spherically fibered
  coisotropic of codimension $3$ resp.\ $2$ and the flow of $\psi
  \circ \vert \Phi \vert$ is a Dehn twist along $\Phinv(0)$.
\end{proposition}

\begin{proof} 
The zero level set $P := \Phinv(0)$ is a $G$-equivariant $G/H$-bundle
over the symplectic quotient $M \qu G$, by the assumption on
stabilizers and existence of local slices.  We identify $G/H \cong
S^c$ with $c = 3$ resp. $c = 2$, in the trivial stabilizer
resp. $U(1)$-stabilizer case.  We show that $P$ is induced from an
$SO(c+1)$ principal bundle and that the flow of $\psi \circ \vert \Phi
\vert$ is obtained by a Dehn twist by the symplectic fiber bundle
construction.

Consider the case of trivial stabilizers.  By the equivariant
coisotropic embedding theorem, there exists an equivariant
symplectomorphism of a neighborhood of $\Phinv(0)$ in $M$ with a
neighborhood of the zero section in
$$ (T^\dual G) \times_G P \cong P \times \g^\dual $$
where the quotient is by the diagonal action $g(v,p) = (R_{g^{-1},*}
v, gp)$ and $R_{g^{-1},*}: T^\dual G \to T^\dual G$ is induced by the
right action of $G$.  The moment map for the action in the local model
is
$$ \Phi: (T^\dual G) \times_G P \to \g^\dual, \quad [v,p] \mapsto
\Phi_{T^\dual G}(v) $$
where $\Phi$ is the moment map for the left action of $G$ on $T^\dual
G$.  The norm
$$\vert \Phi \vert : (T^\dual G) \times_G P \to \R_{\ge 0}, \quad
([v,p]) \mapsto \vert v \vert $$
is identified with the norm on the fibers $T^\dual S^c$.  Thus the
function $\psi \circ \vert \Phi \vert$ is that in the definition of
fibered Dehn twist.  The claim follows.

In the case of circle stabilizers, denote the fixed point set of the
action of the circle subgroup $H$ of diagonal matrices by
$$ P^H \subset P, \quad \{ hp = p \ | \ \forall h \in H \} .$$
The set $P^H$ is a double cover of $M \qu G$, since there are two
$H$-fixed points (the poles) in each fiber $G/H \cong S^2$.  The local
model in this case is
$$  T^\dual(G /H) \times_{N(H)}  P^H \cong G \times_{N(H)}  (P^H \times (\g/\h)^\dual)  .$$
Since the center $Z \subset N(H)$ acts trivially on $P$, we may write
$$T^\dual (G/H) \times_{N(H)} P^H \cong T^\dual(G/H) \times_{G/Z} (
(G/Z) \times_{N(H)} P^H ) $$
where now $G/Z \times_{N(H)} P^H$ is a $G/Z \cong SO(3)$ bundle.  The
moment map is given by
$$ \Phi: (T^\dual (G/H)) \times_G P \to \g^\dual, \quad [v,p] \mapsto
\Phi_{T^\dual (G/H)}(v) $$
and the norm is again the Riemannian norm of the cotangent vector in
the local model.
\end{proof} 

\section{Fibered Dehn twists on moduli spaces of flat bundles}
\label{flat}

This section describes a natural collection of fibered Dehn twists on
moduli spaces of flat bundles, which are our motivating examples.

\subsection{Moduli spaces of flat bundles} 

We first recall the construction of symplectic structures on moduli
spaces of flat bundles on surfaces with markings, that is, flat
bundles on the complement of the markings with holonomies around them
in fixed conjugacy classes.

\begin{definition} \label{labels} 
\begin{enumerate} 
\item {\rm (Conjugacy classes in compact $1$-connected Lie groups)}
  Let $G$ be a simple compact, 1-connected Lie group, with maximal
  torus $T$ and Weyl group $W = N(T)/T$.  Let $\g,\t$ denote the Lie
  algebras of $G$ and $T$.  We choose a {\em highest root} $\alpha_0
  \in \t^\dual$ and {\em positive closed Weyl chamber} $\t_+ \subset
  \t$.  Conjugacy classes in $G$ are parametrized by the Weyl alcove
$$\Alc := \{ \xi \in \t_+,  \ \alpha_0(\xi) \leq 1 \} ,$$ 
see \cite{ps:lg}.  For any $\mu \in \Alc$, we denote by
$$\mathcal{C}_\mu = \{ g \exp(\mu) g^{-1}, \ g \in G \} $$
the corresponding conjugacy class.  Inverting each conjugacy
class defines an involution
\begin{equation} \label{star} 
*: \Alc \to \Alc, \quad \mathcal{C}_{*\mu} = \mathcal{C}_\mu^{-1} .\end{equation}
In the case $G = SU(2)$, we identify $ \t \cong \R$ and $ \Alc \cong
[0,1/2] $ so that
\begin{equation} \label{interval} 
\cC_\mu = \{ \Ad(g) \diag( \exp( 2\pi i \mu), -\exp( 2\pi i \mu)) \} . \end{equation}
In particular 
$$\cC_{1/4} = \{ A \in SU(2) \ | \ \Tr(A)  = 0 \} $$ 
is the conjugacy class of {\em traceless} elements of $SU(2)$.
\item {\rm (Marked surfaces)} By a {\em marked surface} we mean a
  compact oriented connected surface $X$ equipped with a collection of
  distinct points 
$$\ul{x} = (x_1,\ldots, x_n) \in X^n$$ 
and a collection of 
{\em labels} 
$$\ul{\mu}  = (\mu_1,\ldots, \mu_n) \in \Alc^n .$$  
For simplicity, we denote such a surface $(X,\ul{\mu})$.
\item {\rm (Holonomies)} Let $P \to X$ be a $G$-bundle equipped with a
  flat connection $A \in \A(P)$.  Parallel transport around loops in
  $X$ gives rise to a {\em holonomy representation}
$$ \pi_1(X) \to G .$$
In particular, for any point $x \in X$ a small loop around $x$ defines
a conjugacy class in $\pi_1(X)$, obtained by joining the loop to a
base point, and the holonomy around $x$ is well-defined up to
conjugacy.
\item {\rm (Moduli spaces of bundles on marked surfaces)} Let
  $(X,\ul{\mu})$ be a marked surface.  Let $M(X,\ul{\mu})$ denote the
  moduli space of isomorphism classes of flat $G$-bundles on $X - \{
  x_1,\ldots, x_n \}$ with holonomy around $x_i$ conjugate to
  $\exp(\mu_i)$, for each $i = 1,\ldots,n$, see
  e.g. Meinrenken-Woodward \cite{me:lo}.  If $M(X,\ul{\mu})$ contains
  no reducible bundles (bundles with non-central automorphisms) then
  $M(X,\ul{\mu})$ is a compact symplectic orbifold.
\end{enumerate} 
\end{definition} 

\begin{remark}
\begin{enumerate} 
\item {\rm (Action of central bundles)} Let $Z$ denote the center of
  $G$.  Let $M_Z(X)$ denote the moduli space of $Z$-bundles on $X$
  with trivial holonomy around the markings.  The group multiplication
  on $Z$ induces a group structure on $M_Z(X)$, isomorphic to $Z^{2g}$
  where $g$ is the genus of $X$.  The action of $Z$ on $G$ induces a
  symplectic action of $M_Z(X)$ on $M(X,\ul{\mu})$
$$ M_Z(X) \times M(X,\ul{\mu}) \to M(X,\ul{\mu}) $$
corresponding to twisting the holonomies by elements of $Z$.
\item {\rm (Combining central markings)} A label $\mu \in \Alc$ is
  {\em central} if $\exp(\mu)$ lies in the center $Z$ of $G$.  In the
  case $G = SU(r)$, the central labels are the vertices of the
  alcove $\Alc$.  Several central labels may be combined into a
  single central label as follows: Suppose that $\lambda_1 \in \Alc$
  resp.\ $\lambda_2 \in \Alc$ are labels corresponding to $z_1,z_2 \in
  Z$, and $\lambda_{12} \in \Alc$ is the label that corresponds to
  $z_1z_2 \in Z$.  Then there is a symplectomorphism
$$M(X,\lambda_1,\lambda_2,\lambda_3, \ldots,\lambda_n)
\to M(X,\lambda_{12},\lambda_3,\ldots,\lambda_n) .$$ 
This follows immediately from the description of the moduli space as
representations of the fundamental group as in \eqref{fundrep}.
\end{enumerate} 
\end{remark} 

\begin{remark} \label{cutmod}
 {\rm (Moduli of flat bundles as a symplectic quotient
by the loop group)} The moduli space of flat bundles may be realized
  as the symplectic quotient of the moduli space of framed bundles on
  a cut surface described in \cite{me:lo}.  

First we introduce notation for the cut surface. Let $Y \subset X$ be
an embedded {\em circle}, that is, a compact, oriented, connected
one-manifold, disjoint from the markings $\ul{x}$.  Let $X_\cut$
denote the surface obtained from $X$ by cutting along $Y$ as in Figure
\ref{capoff}, with boundary components $(\partial X_\cut)_j \cong S^1,
j =1,2$.  The cut surface $X_\cut$ may be disconnected or connected
depending on whether $Y$ is separating.

The moduli space of {\em framed} flat bundles on the cut surface
naturally has an action of two copies of the loop group, acting by
changing the framings.  Let $M(X_\cut,\ul{\mu})$ be the moduli space
of flat bundles with framings (trivializations) on the boundary of
$X_\cut$:
$$ M(X_\cut,\ul{\mu}) = \{ P \to X_\cut, \ A \in \A(P), \ \phi: P|
\partial X_{\cut} \to \partial X_\cut \times G \ | \curv(A) = 0 \} / \sim .$$
Equivalently, in the case of simply-connected structure group,
$M(X_\cut,\ul{\mu})$ is the quotient of flat connections by gauge
transformations which vanish on the boundary.  Working in suitable
spaces of completions the moduli space $M(X_\cut,\ul{\mu})$ is a
symplectic Banach manifold, with symplectic form given by the usual
pairing of one forms and integration.  Let $LG = \Map(S^1,G)$ denote
the loop group of $G$, with multiplication given by pointwise
multiplication.  Any element $g \in LG^2$ acts on $M(X_\cut,\ul{\mu})$
by changing the framing:
$$ g (P,A,\phi) = (P,A, (1 \times g) \phi) .$$ 
The moment map for the action of $LG^2$ is restriction to the boundary
\begin{equation} \label{phis}
 (\Phi_1,\Phi_2) : M(X_\cut,\ul{\mu}) \to \Omega^1(S^1,\g)^2, \quad
A \mapsto \tau_{\phi} ( A | \partial X_{\cut}) \end{equation}
where $\tau_\phi: \A(P | \partial X_{\cut}) \to \Omega^1(S^1,\g)^2$ is
the parametrization of the space of connections induced by $\phi$.
Since the orientations on the two extra boundary circles are opposite,
the diagonal action of $LG$ has moment map $\Phi$ given by the
difference of the moment maps for each boundary component:
$$ \Phi = \Phi_1 - \Phi_2 : M(X_\cut,\ul{\mu}) \to \Omega^1(S^1,\g) .$$

Taking the quotient by the diagonal loop group action recovers the
moduli space for the original surface: By \cite{me:lo} \label{diag}
$M(X,\ul{\mu})$ is naturally symplectomorphic (on the smooth locus) to
the symplectic quotient of $M(X_\cut,\ul{\mu})$ by the diagonal action
of $LG$, that is,
$$ M(X,\ul{\mu}) \cong M(X_\cut,\ul{\mu}) \qu LG .$$
In particular, the symplectic structure on $M(X,\ul{\mu})$ descends
from a symplectic structure on $M(X_\cut,\ul{\mu})$.  

Locally the loop group actions admit finite dimensional slices and so
the infinite-dimensional quotients above are equivalent to
finite-dimensional quotients \cite{me:lo}.  In particular, let
$\Alc^\circ \subset \Alc$ denote the interior of the alcove.  For
example, for $SU(2)$, we have as in \eqref{interval}
$$\Alc \cong [0,1/2], \quad \Alc^\circ \cong (0,1/2) .$$
Let
$$M(X,\ul{\mu})^\circ = \{ [A] \in M(X,\ul{\mu}) \ | \ A|_U = \xi \d
\theta, \xi \in \Alc^\circ \} $$
be the subset consisting of connections of ``standard form'' $\xi \d
\theta, \xi \in \Alc^\circ$ in a neighborhood $U$ of the circle $Y$.
Similar let
\begin{equation} \label{circ}
 M(X_\cut,\ul{\mu})^\circ = \Phi^{-1}((\Alc^\circ )^2) = \left\{ [A]
 \in M(X_\cut,\ul{\mu}) \ \left| \ \begin{array}{l} \forall k \in \{
 0,1 \}, \ \exists \xi_k \in \Alc^\circ \\ A | (\partial X_\cut)_k =
 \xi_k \d \theta_k \end{array} \right. \right\} \end{equation}
The locus $M(X_\cut,\ul{\mu})^\circ$ is an open subsets of
$M(X_\cut,\ul{\mu})$, dense if non-empty, which is the quotient by the
diagonal action of the maximal torus $T$:
\begin{equation} \label{qu} M(X,\ul{\mu})^\circ = M(X_\cut,\ul{\mu})^\circ \qu T; \end{equation} 
see \cite{me:lo}.  More generally, for any face $\sigma$ of the alcove
let $LG_\sigma$ denote the stabilizer of any point in $\sigma$.  The
stabilizer $LG_\sigma$ is isomorphic to a finite-dimensional subgroup
of $G$ via evaluation at a base point.  Denote by
$$ \Alc_\sigma := \bigcup_{\ol{\tau} \supset \sigma} LG_\sigma \tau $$
the slice for the action of $LG$ on $L\g^\dual$ at $\sigma$.  Then 
$$  M(X_\cut,\ul{\mu})^\sigma 
:= (\Phi_1 \times \Phi_2)^{-1}(\Alc_\sigma \times \Alc_\sigma) $$%
is a Hamiltonian $LG_\sigma^2$ space whose quotient 
\begin{equation} \label{cover} M(X,\ul{\mu})^\sigma
 := M(X_\cut,\ul{\mu})^\sigma \qu LG_\sigma \end{equation}
is a dense (if non-empty) open subset of $M(X_\cut,\ul{\mu})$.  As the
face $\sigma$ varies the collection $ M(X,\ul{\mu})^\sigma$ covers
$M(X_\cut,\ul{\mu})$:
$$ M(X_\cut,\ul{\mu}) = \bigcup_{\sigma \subset \Alc}
M(X,\ul{\mu})^\sigma .$$
\end{remark}  

\subsection{Symplectomorphisms induced by Dehn twists}

Any Dehn twist on a marked surface induces a symplectomorphism of the
moduli space of flat bundles.  In this section we explicitly describe
this symplectomorphism as the Hamiltonian flow of a {\em non-smooth}
function.

We begin by setting up notation for symplectomorphisms of moduli
spaces induced by diffeomorphisms of the underlying surface.  Let
$(X,\ul{\mu})$ be a marked surface.  Let $\Diff^+ \subset \Diff(X)$
denote the group of orientation-preserving diffeomorphisms of $X$ and
$\Map^+(X,\ul{\mu})$ the group of isotopy classes of orientation
preserving diffeomorphisms $\phi$ of $X$ preserving the labels
$\ul{\mu}$:
$$ \Map^+(X,\ul{\mu}) = \{ \phi \in \Diff^+(X) \, | \, \phi( \{ x_i
\}) = \{ x_i \}, \ \phi(x_i) = x_j \implies \mu_i = \mu_j \ \forall
i,j \}/ \sim .$$
The following is elementary and left to the reader:

\begin{proposition} {\rm (Symplectomorphisms associated to mapping 
class group elements)} Pullback defines a homomorphism from
  $\Map^+(X,\ul{\mu})$ to the group of symplectomorphisms
  $\Diff(M(X,\ul{\mu}),\omega)$ of $M(X,\ul{\mu})$,
$$ \Map^+(X,\ul{\mu}) \to \Diff(M(X,\ul{\mu}),\omega), \ \ [\phi]
  \mapsto ([A] \mapsto [(\phi^{-1})^* A]).  $$
\end{proposition}  

In particular, a Dehn twist on the surface induces a symplectomorphism
of the moduli space of bundles.  Our aim is to describe this
symplectomorphism as a Hamiltonian flow of a function relating to the
holonomy.  The space of connections on the trivial bundle $S^1 \times
G \to S^1$ may be canonically identified with the space of $\g$-valued
one-forms
$$ L\g^\dual := \Omega^1(S^1,\g) .$$
Denote $\pi:\R \to S^1$ and let 
$$\gamma_{t_0,t_1}(t) = \pi(t_0 + t(t_1
- t_0)), \quad t \in [0,1] $$ 
denote the standard path from $t_0$ to $t_1$. Parallel transport along
$\gamma_{t_0,t_1}$ defines a map
$$ \rho_{t_0,t_1}: L\g^\dual \to G .$$
Given an embedded oriented circle $Y$ in $X$ disjoint from the
markings, we suppose that $Y$ is the image of a path
$$ Y = \iota([0,1]), \quad \iota: [0,1] \to X, \quad \iota(0) =
\iota(1) $$
such that $\iota|(0,1)$ is an embedding.  Let $\rho_{0,1}$ denote
parallel transport once around $Y$ and define
$$ \rho_Y: M(X,\ul{\mu}) \to \Alc, \quad [A] \mapsto [\rho_{0,1}(A |
  Y)] $$
by mapping an equivalence class of flat connections $[A]$ to the
conjugacy class of the holonomy of any representative $A$ around $Y$.
The map $\rho_Y$ is independent of the choice of base point on
$Y$. Let
$$h_Y: M(X,\ul{\mu}) \to \R_{ \ge 0},  \quad [A] \mapsto (\rho_Y([A]), \rho_Y([A]))/2 .$$
The function $h_Y$ is smooth on the inverse image
$\rho_Y^{-1}(\Alc^\circ)$ of the interior $\Alc^\circ$ of the Weyl
alcove.

\begin{proposition}
 \label{flow}
 {\rm (The action of a Dehn twist is a Hamiltonian flow,
   c.f. \cite[Theorem 4.5]{al:mom})} Let $(X,\ul{\mu})$ be a marked
 surface such that $M(X,\ul{\mu})$ contains no reducibles, $Y \subset
 X$ an embedded circle and $\tau_Y: X \to X$ a Dehn twist around $Y$.
 Then $\tau_Y$ acts on $M(X,\ul{\mu})$ by the time-one Hamiltonian
 flow of $h_Y$ on $\rho_Y^{-1}(\Alc^\circ)$.  In particular, the
 time-one Hamiltonian flow of $h_Y$ extends smoothly to all of
 $M(X,\ul{\mu})$.
\end{proposition}

The idea of the proof is to describe the twist upstairs on the moduli
space of the cut surface, and then descend to the glued surface.
Recall from Remark \ref{cutmod} that the moduli space
$M(X_\cut,\ul{\mu})$ has an $LG^2$-action.  In particular the action
of $LG^2$ restricts to an action of the subgroup of constant loops
$G^2$.  For $[A] \in M(X_\cut,\ul{\mu})$ the notation
$$(1,\rho_{0,1}(A | (\partial X_\cut)_2) )[ A ] \in
M(X_\cut,\ul{\mu}) $$
indicates the connection obtained by acting by the holonomy
$\rho_{0,1}(A | (\partial X_\cut)_2)$ of the connection on the second
component of the boundary $(\partial X_\cut)_2$.

\begin{lemma} \label{cutact}
The twist $\tau_Y$ acts on $M(X_\cut,\ul{\mu})$ by changing the
framing of $A$ by the holonomy along the second boundary component:
$$ (\tau_Y^{-1})^* : M(X_\cut,\ul{\mu}) \to M(X_\cut,\ul{\mu}),
\ \ [A] \mapsto (1, \rho_{0,1}(A | (\partial X_\cut)_2) )[ A ] .$$
\end{lemma} 

\begin{proof} 
Since a Dehn twist along a circle in a Riemann surface is only defined
up to isotopy, we may assume that $\tau_Y$ is a Dehn twist along a
small translation of the boundary component $(\partial X_\cut)_2$.
This twist induces a Dehn twist on $X_\cut$, also denoted $\tau_Y$.

\vskip .05in
\noindent {\em Step 1: We show that on the cut surface $X_\cut$ the
Dehn twist $\tau_Y$ maps any connection to a gauge equivalent
connection.}  Suppose first that $X_\cut$ is connected, and choose a
base point $x_0 \in X_\cut$.  Given a connection $A$, the connections
$A, (\tau_Y^{-1})^* A$ determine representations of the fundamental
group
\begin{equation} \label{fundrep}
\Hol_A: \pi_1(X_\cut,x_0) \to G, \quad \Hol_{ (\tau_Y^{-1})^* A}:
\pi_1(X_\cut,x_0) \to G \end{equation}
given by mapping any homotopy class of loops to the holonomy of the
connection.  The two representations are equal, since the generators
of $\pi_1(X_\cut,x_0)$ have representatives that are disjoint from the
support of $\tau_Y$ (or alternatively, since $\tau_Y$ is homotopic to
the identity).  Therefore, $(\tau_Y^{-1})^*A$ is gauge equivalent to
$A$.

\vskip .05in
\noindent {\em Step 2: We compute the gauge transformation relating
  the two connections above on the boundary.}  The necessary gauge
transformation $g : X \to G$ given by the difference in parallel
transports of $A$ and $(\tau_Y^{-1})^* A$.  For any $x \in X_\cut$,
choose a path $\gamma_x$ from $x_0$ to $x$ and let $\rho^A(\gamma_x)
\in G$ denote parallel transport along $\gamma_x$ using the connection
$A$, and similarly for the pull-back conection $(\tau_Y^{-1})^* A$.
Define a gauge transformation
\begin{equation} \label{diff}
g : X_\cut \to G, \quad x \mapsto \rho^{(\tau_Y^{-1})^*A}(\gamma_x)
(\rho^A(\gamma_x))^{-1} .\end{equation}
The gauge transformed connection $g A$ has parallel transport along
$\gamma_x$ given by
$$ \rho^{(\tau_Y^{-1})^*A)}(\gamma_x) (\rho^A(\gamma_x))^{-1}
\rho^A(\gamma_x) = \rho^{(\tau_Y^{-1})^*A)}(\gamma_x).
$$
Therefore, 
$$g A = (\tau_Y^{-1})^* A  .$$  
To compute the gauge transformation $g$, denote by
$$ \rho_{0,t} := \rho_{0,t} ( A | (\partial X_\cut)_2 ) $$
 parallel transport from $0$ to $t$ along the boundary $(\partial
 X_\cut)_2 \cong S^1$.  Define a path from $x_0$ to $x_2$ on the
 boundary, identified with $t \in S^1$ by concatenating a path from
 $x_0$ to $x_1$ and a path from $0$ to $t$ in $S^1$.  The parallel
 transports of $A$ resp. $(\tau_Y^{-1})^* A$ are
$$ \rho^A( \gamma_{x_2} ) = \rho_{0,t} \rho^A(\gamma_{x_1}), \quad
 \rho_{x_2}^{(\tau_Y^{-1})^*A} = \rho_{0,t}
 \rho^{(\tau_Y^{-1})^*A}(\gamma_{x_1}) = \rho_{0,t} \rho_{0,1}
 \rho^A(\gamma_{x_1}) .$$
It follows that \eqref{diff} is given at points $t$ on the second
boundary component $(\partial X)_2 \cong S^1 $ by
\begin{equation}   \label{tosimplify} g(t) = \rho_{0,t} \rho_{0,1} g_0 (\rho_{0,t} g_0)^{-1} = \rho_{0,t}
\rho_{0,1} \rho_{0,t} ^{-1} .\end{equation}

\noindent {\em Step 3: We identify the gauge transformation on the
  boundary with the loop group element in the statement of the Lemma}.
We simplify the formula \eqref{tosimplify} as follows.  After gauge
transformation we may assume that the restriction of $A$ to $(\partial
X_\cut)_2 )$ is of the form $\xi_2 \d \theta$, for some $\xi_2 \in
\g$.  The set of group elements $\rho_{0,t} = \exp( t \xi_2)$ forms a
one-parameter subgroup and $\rho_{0,t}$ and $\rho_{0,1}$ commute.
Hence
$$ g | (\partial X_\cut)_2 = \rho_{0,1} = \exp(\xi_2) , \quad g |
(\partial X_\cut)_1 = 1 $$
as claimed.  The case that $X_\cut$ is disconnected is similar, using
that $\tau_Y$ is trivial on one of the components.
\end{proof} 

\begin{proof}[Proof of  Proposition \ref{flow}]  By Lemma \ref{cutact},
the restriction of $(\tau_Y^{-1})^*$ to $M(X_\cut,\ul{\mu})^\circ$ of
\eqref{circ} has the form
$$ M(X_\cut,\ul{\mu})^\circ \to M(X_\cut,\ul{\mu})^\circ, \ \ [A] \mapsto
(1,\exp(\xi_2)) [A] .$$
Thus the action of $\tau_Y$ is given by the time-one Hamiltonian flow
of the function
$$ \ti{h}_Y: M(X_\cut,\ul{\mu})^\circ \to \R, \quad [A] \mapsto (\xi_2,\xi_2)/2
.$$
The Proposition follows since the restriction of $\ti{h}_Y$ to the
zero-level set in $M(X_\cut,\ul{\mu})$ is a lift of $h_Y$.
\end{proof} 

For later use, we recall the following facts about level sets of
holonomy maps which follow from the {\em gluing equals reduction}
description in Meinrenken-Woodward \cite[Section 3.4]{me:lo}.  We
recall the involution $* :\Alc \to \Alc$ from \eqref{star}.  Let
$X_\mycap$ denote the surface obtained from $X_\cut$ capping off with
a pair of disks, with an additional marked point on each disk, as in
Figure \ref{capoff}.

\begin{figure}[ht]
\begin{picture}(0,0)%
\includegraphics{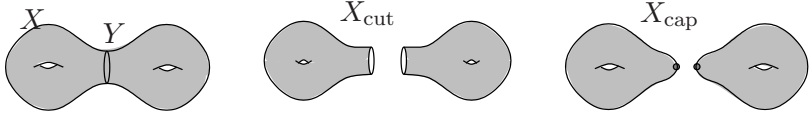}%
\end{picture}%
\setlength{\unitlength}{4144sp}%
\begingroup\makeatletter\ifx\SetFigFont\undefined%
\gdef\SetFigFont#1#2#3#4#5{%
  \reset@font\fontsize{#1}{#2pt}%
  \fontfamily{#3}\fontseries{#4}\fontshape{#5}%
  \selectfont}%
\fi\endgroup%
\begin{picture}(4823,646)(4947,-5019)
\put(8755,-4527){\makebox(0,0)[lb]{{{{$X_{\mycap}$}
}}}}
\put(5540,-4603){\makebox(0,0)[lb]{{{{$Y$}%
}}}}
\put(5040,-4523){\makebox(0,0)[lb]{{{{$X$}%
}}}}
\put(6957,-4496){\makebox(0,0)[lb]{{{{$X_\cut$}%
}}}}
\end{picture}%
\caption{Capping off a surface}
\label{capoff}
\end{figure}
\noindent 

\begin{lemma}  {\rm (Quotients of holonomy level sets)}   Let $(X,\ul{\mu})$ be a marked surface, and $Y \subset X$
an embedded circle.
\label{facts}
 For any $\lambda \in \Alc$, let $G_{\exp(\lambda)}$ denote the
 centralizer of $\exp(\lambda)$.  The product $G_{\exp(\lambda)}
 \times G_{\exp(\lambda)}$ acts on $ \Phi_1^{-1}(\lambda) \times
 \Phi_2^{-1}(\lambda) $ with diagonal quotient resp.\ full quotient %
\begin{eqnarray*}
% punctuation
 \rho^{-1}_Y(\lambda) &\cong & ( \Phi_1^{-1}(\lambda) \times
 \Phi_2^{-1}(\lambda)) /G_{\exp(\lambda)}
 \\ M(X_\mycap,\ul{\mu},\lambda,*\lambda) &\cong&
 (\Phi_1^{-1}(\lambda) \times \Phi_2^{-1}(\lambda))/
 G_{\exp(\lambda)}^2 .\end{eqnarray*}
If all points in the level set of $(\lambda,\lambda)$ have the same
stabilizer $H \subset G_{\exp(\lambda)}^2$ up to conjugacy, then
$$\rho^{-1}_Y(\lambda) \to M(X_\mycap,\ul{\mu},\lambda,*\lambda) $$ 
is a $G_{\exp(\lambda)} \backslash G^2_{\exp(\lambda)}/H$-fiber bundle
and equips $\rho^{-1}_Y(\lambda)$ with the structure of a fibered
coisotropic.
\end{lemma} 

\subsection{Full twists for rank two bundles} 

In this section we show that Dehn twists on a surface induce fibered
Dehn twists of the moduli space of flat rank two bundles with trivial
determinant.  The following is a slight generalization of a result of
M. Callahan (unpublished) resp.\ P. Seidel \cite[Section 1.7]{se:phd}
in the case of a separating resp.\ non-separating curve on a surface.

\begin{theorem} \label{nomarkings} 
{\rm (For $G = SU(2)$, Dehn twists on the surface act by fibered Dehn
  twists on the moduli space)} Suppose $G = SU(2)$, and $(X,\ul{\mu})$
is a marked surface such that $M(X,\ul{\mu})$ contains no
reducibles.  Let $Y \subset X$ be an embedded circle.
\begin{enumerate}
\item {\rm (Separating case gives a codimension one Dehn twist)} If
  $Y$ is separating then $\tau_Y$ acts on $M(X,\ul{\mu})$ by a fibered
  Dehn twist along the fibered coisotropic
$$C_\lambda := \rho_Y^{-1}(\lambda)$$ 
for any $\lambda \in (0,\hh)$ such that
$M(X,\ul{\mu},\lambda,*\lambda)$ contains no reducibles.
\item {\rm (Non-separating case gives a codimension three Dehn twist)}
  If $Y$ is non-separating and $M(X_{\mycap},\ul{\mu})$ contains no
  reducibles, then $\tau_Y$ acts by a fibered Dehn twist along the
  fibered coisotropic
$$C_{1/2} := \rho_Y^{-1}(1/2)$$ 
of bundles with holonomy along $Y$ equal to $\exp(1/2)= -I \in SU(2)$.
\end{enumerate}
\end{theorem}

\begin{proof}  
We have already expressed the action of the Dehn twist as the action
of holonomy in Proposition \ref{flow}.  It remains to identify these
flows as Dehn twists in the rank two case.  For this, it suffices as
in Theorem \ref{thm fibered quotients} to show that the corresponding
Dehn twists induce equivariant Dehn twists on the moduli space of the
cut surface.

\vskip .05in
\noindent {\em Step 1: we compute the generic stabilizers, in order to
  establish that the level sets in the statement are fibrations, and
  examine the action of the Dehn twists on the level sets.}  By the
irreducible-free assumption, the generic stabilizer of the $LG^2$
action on $M(X_\cut,\ul{\mu})$ is canonically identified with $Z^{\#
  \pi_0(X_\cut)}$, where $Z$ is the center of $G$ and $\#
\pi_0(X_\cut)$ is the number of components; the identification is via
evaluation at any point on the boundary.  By Lemma \ref{facts} the
level set $\rho_Y^{-1}(\lambda) $ with $ \lambda \in (0,1/2)$
resp. $\lambda= 1/2$ is a fibered coisotropic with fiber
$$ G_{\exp(\lambda)} \backslash G^2_{\exp(\lambda)}/H = U(1)
\backslash U(1)^2/ (\Z_2 \times \Z_2) \cong S^1 $$
in the non-separating case and
$$ G_{\exp(\lambda)} \backslash G^2_{\exp(\lambda)}/H = 
SU(2)
\backslash SU(2)^2 / \Z_2 \cong S^3$$
in the separating case. Explicitly, the latter isomorphism is given by
$$SU(2) \backslash SU(2)^2 / \Z_2 \to SU(2) \cong S^3,
\quad [ a,b] \mapsto a^{-1} b .$$

\vskip .05in
\noindent {\em Step 2: We establish the Theorem in the separating
  case}.  We write the moduli space as a finite dimensional quotient
in a neighborhood of the fibered coisotropic in the statement.  
Denote
the subset with generic holonomy
$$M(X,\ul{\mu})^\circ = M(X_\cut,\ul{\mu})^\circ \qu T = \rho_Y^{-1}(
  \Alc^\circ) .$$
Since $M(X_\cut,\ul{\mu})^\circ$ has a $T \times T$-action, there is a
residual $T$-action on $M(X,\ul{\mu})^\circ$ with generic stabilizer
$Z^{\# \pi_0(X_\cut) - 1}$.  The action of $\tau_Y$ on
$M(X,\ul{\mu})^\circ$ is Hamiltonian isotopic to a symplectomorphism
given by Proposition \ref{flow} as the Hamiltonian flow of
$\rho_Y^2/2$ on the complement of $\rho_Y^{-1}(0), \rho_Y^{-1}(1/2)$.

We use a small Hamiltonian deformation to construct a
symplectomorphism that is supported on a neighborhood of a fibered
coisotropic.  Let $\sigma = \{ 1/2 \} \subset \Alc$ denote the
endpoint of the alcove.  By \eqref{cover}, a neighborhood of
$\rho_Y^{-1}(1/2)$ is the symplectic reduction of a neighborhood of a
cross-section for $ \Phi_1^{-1}(1/2) \times \Phi_1^{-1}(1/2) \subset
M(X_{\cut},\ul{\mu})$ by a diagonal $G$ action.  More precisely,
evaluation a base points defines an isomorphism $ LG^2_{(1/2,1/2)}
\cong G^2$.  The maximal slice at $(1/2,1/2)$ is
$$\Alc^2_\sigma = LG_{(1/2,1/2)}^2 (0,1/2]^2 \subset L
\g^\dual_{(1/2,1/2)}.$$
Then the quotient of $ M(X_\cut,\ul{\mu})_\sigma = (\Phi_1 \times
\Phi_2)^{-1}(\Alc^2_\sigma)$ by $LG_\sigma \cong G$ is a neighborhood
of $\rho_Y^{-1}(1/2)$.  The function $ \Phi_2^2/2 - ( \Phi_2/2 - 1/8)
$ on $M(X_\cut,\ul{\mu})^\circ$ is equal to the restriction of the
function $\Vert \Phi_2 - 1/2 \Vert^2/2$ on
$M(X_{\cut},\ul{\mu})_{1/2}$ which is smooth.  Let
$$\beta \in C^\infty(M(X,\ul{\mu}))
, \quad \beta |_{\rho_Y^{-1}(1/2)}
= 1 $$
be an cutoff function equal to $1$ near $\rho_Y^{-1}(1/2)$ and
supported in a small neighborhood of $\rho_Y^{-1}(1/2)$.  The function
\begin{equation} \label{cutoff} - \beta (\rho_Y^2/2 - (\rho_Y/2 - 1/8) )
\in C^\infty(M(X,\ul{\mu})^\circ)
\end{equation}
has as smooth extension to $M(X,\ul{\mu})$, which agrees with the
action of $(\tau_Y^{-1})^*$ near $\rho_Y^{-1}(1/2)$ since the time-one
flow of $\rho_Y/2 - 1/8$ is the identity.  Composing the action
of $(\tau_Y^{-1})^*$ with this flow yields an equivariant
symplectomorphism that is the identity near $\rho_Y^{-1}(1/2)$, and is
Hamiltonian isotopic to the original symplectomorphism.  A similar
modification produces a equivariant Hamiltonian perturbation of
$(\tau_Y^{-1})^*$ that is the identity near $\Phi_2^{-1}(0)$.

It now follows from the results on equivariant Dehn twists that the
Dehn twist on the surface acts by a fibered Dehn twist on the moduli
space.  That is, $(\tau_Y^{-1})^*$ is a Dehn twist along any
$\Phi_2^{-1}(\lambda) \subset M(X_\cut,\ul{\mu})^\circ$ for $\lambda$
generic. By Proposition \ref{circle}, $(\tau_Y^{-1})^*$ acts as an
equivariant Dehn twist on $M(X_\cut,\ul{\mu})^\circ$.  By Theorem
\ref{thm fibered quotients} the action of $(\tau_Y^{-1})^*$ descends
to a fibered Dehn twist on $M(X,\ul{\mu})$.

\vskip .05in
\noindent {\em Step 3: We establish the Theorem in the non-separating
  case}.  As before, a neighborhood of the coisotropic is given as a
symplectic reduction of a finite dimensional manifold.  The level set
$$\widetilde{C}_{1/2} = \Phi_2^{-1}(1/2) \subset U \subset
M(X_{\cut},\ul{\mu})$$
is an equivariant fibered coisotropic in the finite dimensional
manifold $U$ since it is the zero level set of a moment map for a free
action.  Lemma \ref{cutact} shows that $(\tau_Y^{-1})^*$ is an
equivariant fibered Dehn twist around $\widetilde{C}_{1/2}$.  Again a
Hamiltonian perturbation gives a symplectomorphism that acts by the
identity on a neighborhood of $\rho_Y^{-1}(0)$.  It follows from
Theorem \ref{thm fibered quotients} that $(\tau_Y^{-1})^*$ acts on
$M(X,\ul{\mu})$ by a fibered Dehn twist around $C_{1/2} =
\widetilde{C}_{1/2} \qu G$.
\end{proof}

\subsection{Half twists for rank two bundles with fixed holonomies} 

In this section we show that a half-twist on a marked surface
$(X,\ul{\mu})$ induces a fibered Dehn twist on the moduli space
$M(X,\ul{\mu})$ of flat $G = SU(2)$ bundles with fixed holonomy.

\begin{definition} \label{CY} 
{\rm (Half-twist and corresponding coisotropic)} Given a pair of
points $x_i,x_j \in X$ with the same label $\mu_i = \mu_j$, let $Y$
denote a circle around $x_i,x_j$ and
$${\tau_Y}:X \to X, \quad x_i \mapsto x_j, \ \ x_j \mapsto x_i $$
a half-twist along $Y$ that interchanges $x_i,x_j$.  Denote the
coisotropic of bundles with trivial holonomy along $Y$
$$ C_Y := \{ [A] \in M(X,\ul{\mu}), \ \rho_Y([A]) = 0 \} .$$
\end{definition} 

\begin{example}  {\rm (Moduli spaces of flat bundles on the sphere punctured
five times)} Let $X = S^2$ with five markings $x_1,\ldots, x_5$ all
  with labels $1/4 \in \Alc$.  The moduli space of flat bundles
$$ M(X,\ul{\mu}) = \left\{ (g_1,\ldots, g_5) \in \cC_{1/4}^5 \, | \, 
g_1 \ldots g_5 = 1 \right\}/G $$
is real dimension four.  The real manifold $M(X,\ul{\mu})$ admits the
structure of a del Pezzo surface obtained by blowing up the projective
plane at four points.  This fact follows from the existence of a
K\"ahler structure by the Mehta-Seshadri theorem \cite{ms:pb}, and
computing its homology by any number of standard techniques and noting
the rationality of the moduli space or by the more detailed discussion
in \cite{loray:lag}.  The submanifold $C_Y$ from \eqref{CY} given by a
loop $Y$ around the $i$-th and $j$-th marking is a Lagrangian sphere
described as bundles whose holonomy along a loop containing the $i$-th
and $j$-th markings is the identity:
$$ C_Y = \{ g_i g_j = 1 \} \subset M(X,\ul{\mu}) .$$
The intersection diagram of these Lagrangians reproduces the root
system $A_4$ corresponding to the fifth del Pezzo, discussed in
general in Manin \cite{manin:cubic}.  After choosing suitable
generators for the fundamental group we may assume that $i,j$ are
adjacent.  The moduli space $M(X,\ul{\mu})$ can be described as the
moduli space of closed spherical polygons in $S^3 \cong SU(2)$ with
side lengths $\pi$ and vertices
$$1,\ g_1,\ g_1g_2,\ldots, \ g_1g_2g_3g_4g_5 \ =\  1 .$$  
The submanifold $C_Y$ consists of closed spherical polygons of side
lengths $\pi$ such that the $i$ and $j$-th edges are opposite.  That
is, the polygon consists of a bigon and a triangle as in Figure
\ref{bigon}.
\begin{figure}[ht]
\includegraphics[height=1in]{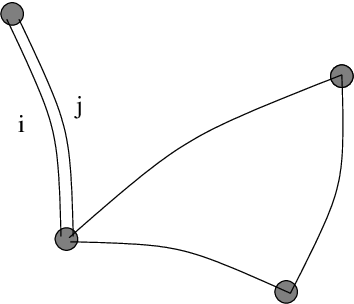}
\caption{Coisotropic of degenerate configurations}
\label{bigon}
\end{figure} 
Let $Y_{i(i+1)}, Y_{(i-1)i} \subset X $ be small circles around the
pairs consisting of $i,i+1$ resp. $i-1,i$-th markings.  The
intersection $C_{Y_{i(i+1)}}, C_{Y_{(i-1)i}}$ is the configuration
where the bigon is conincident with the $i+1$-st edge of the triangle
as in Figure \ref{degen}.
\begin{figure}[ht]
\includegraphics[height=1in]{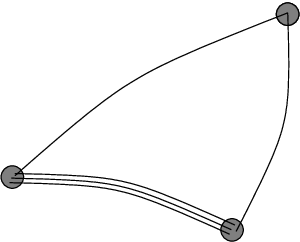}
\caption{Configuration in the intersection of two coisotropics}
\label{degen}
\end{figure} 
See \cite[Section 5]{fieldb} for more discussion.  More generally, in
the case of $2n+1$ markings, the coisotropics $C_{ij}$ consist of
configurations composed of a bigon and an $2n-1$-gon; each is
$S^2$-fibered over the moduli space with $2n-1$ markings by forgetting
the bigon.  Applied to the case $X = S^2$ with five markings, this
construction gives a lift of generators of the action of the Weyl
group $S_4$ (the symmetric group on five letters) of type $A_4$ on
homology to Dehn twists on the fifth del Pezzo.  A lift of the action
to the braid group is discussed in Seidel \cite{se:le}.
\end{example}

The following is slight generalization of a result of Seidel
\cite{se:phd}.

\begin{theorem} {\rm (For $G = SU(2)$ half-twists around pairs of markings act by 
codimension two fibered Dehn twists on the moduli
space)} \label{halftwist} Let $G = SU(2)$.  Let $(X,\ul{\mu})$ be a
  marked surface such that $M(X,\ul{\mu})$ contains no reducibles and
  $\mu_i = \mu_j = 1/4$. Let $Y \subset X$ be an embedded circle
  disjoint from the markings that is the boundary of a disk containing
  $x_i,x_j$.  Then the set $C_Y$ from \eqref{CY} is a spherically
  fibered coisotropic submanifold of codimension $2$, and the action
  of ${\tau_Y}$ on $M(X,\ul{\mu})$ is a fibered Dehn twist along
  $C_Y$.
\end{theorem}

\begin{proof}   
First show that $M(X, \ul{\mu} - \{ \mu_i, \mu_j \})$ is smooth, or
equivalently, contains no reducible representations.  Indeed, note
that the stabilizers of 
$$ (g_i,g_j), \quad g_i g_j = 1, \quad g_i,g_j \in \cC_{1/4}$$
range over all one-parameter subgroups of $SU(2)$.  Hence if some
point in $M(X, \ul{\mu} - \{ \mu_i, \mu_j \})$ has a one-parameter
group of automorphisms then there is also a point in $M(X,\ul{\mu})$
with a one-parameter group of automorphisms.  Since $M(X,\ul{\mu})$
contains no reducibles, $M(X, \ul{\mu} - \{ \mu_i, \mu_j \})$ also
contains no reducibles.

By Lemma \ref{facts}, the constancy of generic stabilizer implies the
existence of a fibered coisotropic.  Let
$$X_\cut = X_{\cut,1} \sqcup X_{\cut,2}, \quad \{ x_i, x_j \} \subset X_{\cut,2} $$ 
denote the surface obtained by cutting $X$ into two components
$X_{\cut,1}, X_{\cut,2}$ along $Y$, so that $X_{\cut,2}$ contains the
markings $x_i,x_j$.  The generic stabilizer $H$ from \eqref{CY} is the
product of stabilizers for the two factors, since $Y$ is
disconnecting: The center $Z = \Z_2$ for $M(X_{\cut,1}, \ul{\mu} - \{
\mu_i, \mu_j \})$ by the previous paragraph, and the maximal torus
$U(1)$ for the right factor $M(X_{\cut,2}, \mu_i, \mu_j)$.  By
\ref{facts} $C_Y \subset M(X,\ul{\mu})$ from \eqref{CY} is a fibered
coisotropic of codimension $2$ with fiber
$$ G_{\exp(\lambda)} \backslash G^2_{\exp(\lambda)}/H = SU(2)/(\Z_2
\times U(1)) \cong S^2 .$$

To describe the action of the half-twist, we first describe the action
of the half-twist on the moduli space of the cut surface.  By
restriction an element of $M(X_\cut,\ul{\mu})$ gives a flat connection
on the twice-punctured disk $X_{\cut,2} - \{x_i, x_j \}$, with
holonomies $g_i,g_j$ around the punctures.  We choose convenient
representatives for $g_i,g_j$ as follows.  Let
$$n = \left[ \begin{array}{ll} 0 & i \\ i & 0 \end{array} \right] 
\in N(T)$$ 
be a representative of the non-trivial element $w$ of the Weyl group
$W \cong \Z_2$.  The inverse is
$$n^{-1} = -n \in \cC_{1/4} .$$
Furthermore, the stabilizer $G_n$ acts on $\cC_{1/4}$ by rotation
fixing $n$ and $n^{-1}$.  The set $nT$ is diffeomorphic to a circle
passing through $n$.  Thus after conjugating $g_i,g_j$ by some $g \in
G$, we may assume
$$ g_i =n, \quad g_j =
t n $$ 
for some $t \in T$.  We compute the action of the half-twist on the
holonomies: For the $i$-th holonomy
$$ g_i \mapsto g_ig_jg_i^{-1} = n t = t^{-1} n = \Ad( ( t^{-1})^{\hh})
  g_i = \Ad( (- g_i g_j)^{\hh} ) g_i $$
(the square root is unique up to an element of the center, which acts
  trivially).  On the other hand, for the $j$-th holonomy
$$ g_j \mapsto g_i = n = \Ad( ( t^{-1})^{\hh}) t n = \Ad( (- g_i
  g_j)^{\hh} ) g_j.$$
This shows that ${\tau_Y}$ acts on $g_i,g_j$ by conjugation by $\Ad(
(- g_i g_j)^{\hh})$.  As in the case of a full twist, this implies
that $(\tau_Y^{-1})^*A$ is gauge equivalent to the connection obtained
from $A$ by acting by $(1, (-g_i g_j)^{\hh})$. 

We write a neighborhood of the fibered coisotropic as a
finite-dimensional symplectic reduction.  Let $\sigma = \{ 0 \}$
denote the endpoint of the alcove.  By \eqref{qu}
$$ M(X,\ul{\mu})^\sigma = M(X_\cut,\ul{\mu})^\sigma \qu SU(2) $$
is a neighborhood of $C_Y$.  Restricted to
$M(X_\cut,\ul{\mu})^\sigma$, $|\Phi_2|/2$ gives a real-valued function
whose Hamiltonian flow is the action of $-I$.  Hence $|\Phi_2|/4$ has
Hamiltonian flow given by the action of $(-I)^{1/2} = \exp(1/4)$.
Similarly $|\Phi_2|^2/4$ has Hamiltonian flow given by the action of
$(g_ig_j)^{1/2}$.  Combining these remarks shows that the
%C
action of $\tau_Y$ on $M(X_\cut,\ul{\mu})^\sigma$ is the time-one
Hamiltonian flow of $|\Phi_2| (|\Phi_2| + 1)/4$, and the level set
$\Phi_2^{-1}(0)$ is an equivariant fibered coisotropic.  By using a
cutoff function as in \eqref{cutoff}, the action of $\tau_Y$ is
Hamiltonian isotopic to an equivariant Dehn twist along
$\Phi_2^{-1}(0)$.  The proof is completed by applying Theorem \ref{thm
  fibered quotients}.
\end{proof} 

\subsection{Half-twists for  higher rank bundles}

For special labels, a half twist on the marked surface $X$ induces a
fibered Dehn twist on the moduli space of flat $SU(r)$ bundles with
fixed holonomy.

\begin{definition} {\rm (Khovanov-Rozansky modification on moduli spaces)} 
\label{kr} 
Let $r \ge 2$ and let $\omega_k \in \Alc, \ k = 0,\ldots, r-1$ denote
the vertices of the alcove.  We identify $\g$ with $\g^\dual$ using
the basic inner product for which the roots have norm-square equal to
$2$.  Under this identification $\omega_0 = 0$ while
$$\omega_k = (\underbrace{(r-k)/r, \ldots, (r-k)/r}_k, \underbrace{
  -k/r,\ldots, -k/r}_{r-k}), \quad 1 \leq k \leq r$$
are the fundamental weights of $G = SU(r)$.  Let $\ul{\mu} = (\mu_1,
\ldots,\mu_n)$ be a collection of labels with
$$\mu_i = \mu_j = \nu_k^1 := (\omega_k + \omega_{k+1})/2$$
for some $i,j,k$.  (That is, $\mu_i,\mu_j$ equal the midpoint
$\nu_k^1$ of the edge of $\Alc$ connecting the $k$-th vertex with the
$k+1$-st vertex.)  Let $C_Y$ be the subset of $M(X,\ul{\mu})$ defined
by
\begin{equation} \label{CY2}
 C_Y = \{ [A] \in M(X,\ul{\mu}) \, | \, \rho_Y([A]) \sim
 \exp(\nu_{k}^2) \}
\end{equation} 
where $\nu_k^2$ is the midpoint between the vertices
$\omega_k,\omega_{k+2}$ of the alcove $\Alc$,
$$ \nu_{k}^2 := (\omega_{k+2} + \omega_{k})/2 .$$ 
For example, 
$$ ( r = 2,  \ \ k = 0) \quad \implies \quad (\nu_k^1 = 1/4, \ \ \nu_k^2 = 0) $$
while 
$$ (r = 3, \ \ k = 0) \quad \implies \quad (\nu_k^1 = \omega_1/2,
\ \ \nu_k^2 = \omega_2/2) .$$
Denote by
$$ M(X,\ul{\mu},(\mu_i,\mu_j) \mapsto \nu_k^2 ) :=
M(X,\mu_1, \ldots, \hat{\mu}_i, \ldots, \hat{\mu_i},
\ldots, \mu_n, \nu_k^2)$$
the moduli space of flat bundles on $X$ with labels obtained by
removing $\mu_i,\mu_j = \nu_k^1$ and adding $\nu_{k}^2$.
\end{definition} 

\begin{theorem} \label{halftwistr} {\rm (Codimension two 
fibered Dehn twists via half-twists)} Suppose that $G = SU(r)$ and
  $\mu_i,\mu_j, \nu_k^1,\nu_k^2$ are as in Definition \ref{kr} such
  that the moduli spaces $M(X,\ul{\mu})$ and
  $M(X,\ul{\mu},(\mu_i,\mu_j) \mapsto \nu_k^2 )$ contain no
  reducibles.  Let $Y \subset X$ denote an embedded circle enclosing
  only the $i$-th and $j$-th markings.  Then the subset $C_Y$ from
  \eqref{CY2} is a spherically fibered coisotropic submanifold of
  codimension $2$, fibered over the moduli space $M(X,\ul{\mu},
  (\mu_i,\mu_j) \mapsto \nu_k^2)$.  The action of ${\tau_Y}$ on
  $M(X,\ul{\mu})$ is a fibered Dehn twist along $C_Y$.
\end{theorem}

\begin{proof} 
First we show that the given level set is a fibered coisotropic.  Let
$g_i$ resp.\ $g_j$ denote the holonomies around the $i$-th
resp.\ $j$-th marking.  The half-twist produces a connection with
holonomies in which $g_i,g_j$ have been replaced with $g_i g_j
g_i^{-1}, g_i$.  Let $\cC_1$ denote the conjugacy class of
$\exp(\omega_1/2)$.  It suffices to consider the case that $k = 0$ so
that $g_i,g_j \in \cC_1$.  That is, let
$$\nu_k^1 = \mu_i = \mu_j = \omega_1/2, \quad \nu_k^2
= \omega_2/2 .$$ 
We begin by choosing convenient representatives of $g_i,g_j$.  We may
assume
\begin{eqnarray*} 
 g_i &=& \exp( \omega_1/2) = \exp( 2\pi i (\diag( (r-1)/2r , -1/2r,
 \ldots, -1/2r))) \\ &=& \diag(-\exp (\pi i / r), \exp(\pi i /r),
 \ldots, \exp(\pi i /r)) .\end{eqnarray*}
The centralizer of $g_i$ is therefore 
$$Z_{g_i} = S(U(1) \times U(r-1))
\cong SU(r-1) .$$  
Let
$$O 
 = \on{Im}( SO(2,\R) \to G, \ A \mapsto \diag(A,1,1,1
\ldots, 1)) $$
denote the subgroup of real orthogonal rotations in the first two
coordinates in $\C^r$.  Since $g_i$ is the product of
$\diag(-1,1\ldots,1)$ with a central element in $U(r)$, the adjoint
action of $g_i$ on $O$ is given by $g_i o g_i^{-1} = o^{-1}, o \in O$.  So 
$$ og_i = \Ad( o^{1/2}) g_i \in \cC_1, \forall o \in O .$$
The conjugacy class $\cC_1$ is a symmetric space of rank one.  In
particular $Z_{g_i}$ acts transitively on the unit sphere in $T_{g_i}
\cC_1$.  This implies that the map $O g_i \to \cC_1/Z_{g_i}$ is
surjective.  By conjugating $g_i,g_j$ by an element of $Z_{g_i}$ we
may choose the second element $g_j$ so that
$$g_j = o g_i = g_i o^{-1} , \quad \text{ for \ some\ } o \in O .$$  
The subgroup $O$ is conjugate to the one-parameter subgroup generated
by the simple root $\alpha_1$ (or rather its dual coweight).  It
follows that the square of $\cC_1$ in $G$ is given by the union of
conjugacy classes
\begin{equation} \label{union} \cC_1^2 = \Ad(G) \{ g_1^2 o  \ | \  o \in O \} = \bigcup_{\eps \in [0,-\hh]}
\cC_{ \omega_1 + \eps \alpha_1 } ,\end{equation}
In particular, the conjugacy class $\cC_2$ of $\exp(\omega_2/2)$
appears in $\cC^2_1$ as the orbit of the element 
$$ \exp(\omega_1/2) \exp(s_1 \omega_1/2) = \exp(\omega_2/2) $$
where $s_1 \in G$ is a representative first simple reflection. The
generic stabilizer for the action of $G$ on $\cC_1 \times \cC_1$ at
the inverse image of $\cC_2$ is the maximal torus
$$ S(U(2) \times U(r-2)) \cap \Ad(s_1) (SU(2) \times U(r-2)) = S(U(1)
\times U(1) \times U(r-2)) .$$
By Lemma \ref{facts}, $ C_Y $ is a fibered coisotropic with fiber
$$ S(U(2) \times U(r-2)) / S(U(1) \times U(1) \times U(r-2)) \cong S^2
.$$

Next we identify the action of the half-twist.  On the holonomies the
half-twist acts by
\begin{equation} \label{holact}  g_i \mapsto g_ig_jg_i^{-1} = g_i o = \Ad( o^{1/2} ) g_i, \quad g_j
\mapsto g_i = o g_j = \Ad( o^{1/2}) g_j .\end{equation} 
Let $\sigma$ denote the face of $\Alc$ containing $\omega_2/2$ and
$LG_\sigma$ the stabilizers so that
$$ M(X,\ul{\mu})^{\sigma} = M(X_\cut,\ul{\mu})^{\sigma} \qu
LG_\sigma $$
is an open neighborhood of $C_Y$ in $M(X,\ul{\mu})$.  Let $A$ be a
framed connection on $X_\cut$.  The pull-back connection
$(\tau_Y^{-1})^* A$ has the same holonomies as that of $(1,o^{1/2})A$,
where $(1,o^{1/2}) \in LG^2$ denotes the constant element of the loop
group with values $1,o^{1/2}$.  As in the previous cases, this
equality implies that the connections $A, (\tau_Y^{-1})^* A$ are gauge
equivalent.  

To identify the gauge transformation, suppose that $A | (\partial
X)_{\cut,2} = \xi \d \theta$ for some $\xi \in \g$.  The evaluation of
the moment map $\Phi_2([A]) = \xi$ is related to the holonomies around
$x_i,x_j$ by $ \exp(\xi) = g_i g_j $, since $\exp(\xi)$ is the
holonomy around the second boundary circle $(\partial X_{\cut})_2$.
Since $g_i^2 = \exp(\omega_1)$ we have
$$ o = g_i^{-1} g_j = (g_i^{-2}) g_i g_j = \exp(- \omega_1 + \xi) .$$
By \eqref{union}, the image of $M(X_\cut,\ul{\mu})^\circ$ under
$\Phi_2$ is the interval with endpoints $\omega_1,\omega_2/2$, see
Figure \ref{halftw} for the $SU(3)$ case.  It follows that the action
of the half-twist is the Hamiltonian flow of the function on
$M(X_{\cut},\ul{\mu})$ whose restriction to
$M(X_{\cut},\ul{\mu})^\circ$ is $ (\Phi_2,\alpha_1)^2/2 - (\Phi_2,
\omega_1/2)$.  An argument using cutoff functions as in \eqref{cutoff}
shows that the half-twist is an equivariant fibered Dehn twist along
$\Phi_2^{-1}(\omega_2/2)$. The claim follows from Theorem \ref{thm
  fibered quotients}.
\end{proof} 

\begin{figure}[h!]
\begin{picture}(0,0)%
\includegraphics{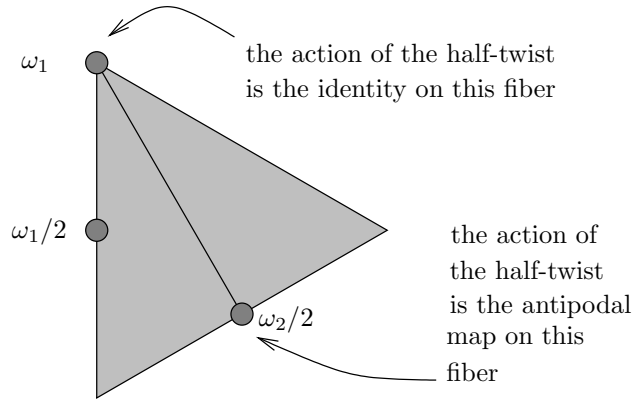}%
\end{picture}%
\setlength{\unitlength}{4144sp}%
\begingroup\makeatletter\ifx\SetFigFont\undefined%
\gdef\SetFigFont#1#2#3#4#5{%
  \reset@font\fontsize{#1}{#2pt}%
  \fontfamily{#3}\fontseries{#4}\fontshape{#5}%
  \selectfont}%
\fi\endgroup%
\begin{picture}(3941,2367)(2703,-3866)
\put(4117,-2051){\makebox(0,0)[lb]{\smash{{\SetFigFont{10}{14.4}{\rmdefault}{\mddefault}{\updefault}{\color[rgb]{0,0,0}is the identity on this fiber}%
}}}}
\put(4118,-1840){\makebox(0,0)[lb]{\smash{{\SetFigFont{10}{14.4}{\rmdefault}{\mddefault}{\updefault}{\color[rgb]{0,0,0}the action of the half-twist}%
}}}}
\put(5344,-2925){\makebox(0,0)[lb]{\smash{{\SetFigFont{10}{14.4}{\rmdefault}{\mddefault}{\updefault}{\color[rgb]{0,0,0}the action of}%
}}}}
\put(5337,-3138){\makebox(0,0)[lb]{\smash{{\SetFigFont{10}{14.4}{\rmdefault}{\mddefault}{\updefault}{\color[rgb]{0,0,0}the half-twist}%
}}}}
\put(5337,-3343){\makebox(0,0)[lb]{\smash{{\SetFigFont{10}{14.4}{\rmdefault}{\mddefault}{\updefault}{\color[rgb]{0,0,0}is the antipodal}%
}}}}
\put(5317,-3528){\makebox(0,0)[lb]{\smash{{\SetFigFont{10}{14.4}{\rmdefault}{\mddefault}{\updefault}{\color[rgb]{0,0,0}map on this }%
}}}}
\put(5317,-3753){\makebox(0,0)[lb]{\smash{{\SetFigFont{10}{14.4}{\rmdefault}{\mddefault}{\updefault}{\color[rgb]{0,0,0}fiber}%
}}}}
\put(4195,-3423){\makebox(0,0)[lb]{\smash{{\SetFigFont{10}{14.4}{\rmdefault}{\mddefault}{\updefault}{\color[rgb]{0,0,0}$\omega_2/2$}%
}}}}
\put(2718,-2893){\makebox(0,0)[lb]{\smash{{\SetFigFont{10}{14.4}{\rmdefault}{\mddefault}{\updefault}{\color[rgb]{0,0,0}$\omega_1/2$}%
}}}}
\put(2775,-1881){\makebox(0,0)[lb]{\smash{{\SetFigFont{10}{14.4}{\rmdefault}{\mddefault}{\updefault}{\color[rgb]{0,0,0}$\omega_1$}%
}}}}
\end{picture}%
\caption{Action of half-twist and and product of holonomies}
\label{halftw}\end{figure}

\section{Pseudoholomorphic sections of Lefschetz-Bott fibrations} 
\label{sections}

In this section we describe the relative invariants associated to
Lefschetz-Bott fibrations.  These are maps between Lagrangian Floer
cohomology groups obtained by counting pseudoholomorphic sections.

\subsection{Monotone Lagrangian Floer cohomology}

Novikov rings are not needed to define Lagrangian Floer cohomology for
a monotone {\em pair} of Lagrangian submanifolds.  However, we will
need our cochain complexes to admit ``action filtrations''.  For this
we find it convenient to use the version incorporating a formal
variable keeping track of area, as in the construction of the spectral
sequence in Fukaya-Oh-Ono-Ohta \cite{fooo}.

\begin{definition}  {\rm (Monotonicity conditions)}
\begin{enumerate} 
\item {\rm (Symplectic backgrounds)} Fix a monotonicity constant
  $\lambda\geq 0$ and an even integer $N > 0$.  A {\em symplectic
  background} is a tuple $(M,\omega,b,\Lag^N(M))$ as follows.
\begin{enumerate}
\item {\rm (Bounded geometry)} $(M,\omega)$ is a compact smooth
  symplectic manifold with either empty or convex boundary;
\item {\rm (Monotonicity)} $\omega$ is $\lambda$-monotone,
  i.e.\ $[\omega] = \lambda c_1(TM)$;
\item {\rm (Background class)} $b \in H^2(M,\Z_2)$ is a {\em
  background class}, which will be used for the construction of
  orientations; and
\item {\rm (Maslov cover)} $\Lag^N(M) \to \Lag(M)$ is an $N$-fold
  Maslov cover such that the induced $2$-fold Maslov covering
  $\Lag^2(M)$ is the oriented double cover.
\end{enumerate}
We often refer to a symplectic background $(M,\omega,b,\Lag^N(M))$ as
$M$.
\item {\rm (Monotone Lagrangians)} A Lagrangian $L \subset M
  \backslash \partial M$ is {\em monotone} if
%ZC
there is an area-index relation for disks with boundary in $L$, that
is,
$$ A(u) = \frac{\lambda}{2} I(u) , \quad \forall u: (D,\partial D) \to
(M,L) $$
where $A(u) = \lan [\omega], [u] \ran$ (the pairing of $[\omega] \in
H^2(M,L)$ with $[u] \in H_2(M,L)$) is again the symplectic area and
$I(u)$ is the Maslov index of $u$ as in \cite[Appendix]{ms:jh}.
\item {\rm (Monotone tuples of Lagrangians)} A tuple
  $(L_b)_{b\in\mathcal{B}}$ of Lagrangians in $M$ is {\em monotone} if
  the following holds: Let $S$ be any compact surface with boundary
  given as a disjoint union of one-manifolds $C_b$
$$\partial S=\sqcup_{b\in\mathcal{B}} C_b$$ 
(with $C_b$ possibly empty or disconnected).  Then for some constant
  $c(S,M, (L_b)_{b \in \mB})$ independent of $u$,
\begin{equation} \label{Au} A(u) = \frac{\lambda}{2} \cdot I(u) + c(S,M,(L_b)_{b \in
       \mB}), \quad \forall u:(S, ( C_b)_{b \in \mB}) \to (M,
  (L_b)_{b \in \mB})
\end{equation} 
where $I(u)$ is the sum of the Maslov indices of the totally real
subbundles $(u|_{C_b})^*TL_b$ in some fixed trivialization of $u^*TM$.
There is a similar definition of monotonicity for tuples of Lagrangian
correspondences \cite{ww:quiltfloer}.
\item {\rm (Admissible Lagrangians)} As mentioned in Section \ref{tri}
  we say that a compact monotone Lagrangian submanifold $L$ is {\em
    admissible} if the image of the fundamental group of $L$ in $M$ is
  torsion and $L$ has minimal Maslov number at least $3$.  One may
  also allow the case that $L$ has minimal Maslov number $2$ and disk
  invariant (number of disks passing through a generic point) $0$.
  However, the Maslov index $2$ case is discussed separately in
  \ref{maslovtwo}.  Any tuple of admissible, monotone Lagrangians is
  automatically monotone, by the argument in Oh \cite{oh:fl1}.  This
  argument involves completing each boundary component of the surface
  by adding a disk obtained by contracting a loop, possibly after
  passing to a finite cover.  Products of admissible Lagrangians are
  also automatically admissible.  Thus if $L^0,L^1$ resp $C$. are
  admissible Lagrangians in $M$ resp. $M^- \times B$ then the tuples
  $(L^0,L^1)$, $(L^0 \times C, C^t \times L^1)$ are monotone.
\item {\rm (Generalized Lagrangian correspondences)} Let $M,M'$ be
  symplectic manifolds.  A {\em generalized Lagrangian correspondence}
  $\ul{L}$ from $M$ to $M'$ consists of
\begin{enumerate}
\item a sequence $N_0,\ldots,N_r$ of any length $r+1\geq 2$ of
symplectic manifolds with $N_0 = M$ and $N_r = M'$ ,
\item a sequence $L_{01},\ldots, L_{(r-1)r}$ of compact Lagrangian
correspondences with $L_{(j-1)j} \subset N_{j-1}^-\times N_{j}$ for
$j=1,\ldots,r$.
\end{enumerate}
Let $\ul{L}$ from $M$ to $M'$ and $\ul{L}'$ from $M'$ to $M''$ be two
generalized Lagrangian correspondences. Then we define concatenation
$$ \ul{L} \sharp \ul{L}' :=
\bigl(L_{01},\ldots,L_{(r-1)r},L'_{01},\ldots,L'_{(r'-1)r'}\bigr)
$$
as a generalized Lagrangian correspondence from $M$ to $M''$.
Moreover, we define the dual
$$
\ul{L}^t := \bigl(L_{(r-1)r}^t,\ldots,L_{01}^t\bigr) .
$$ 
as a generalized Lagrangian correspondence from $M'$ to $M$.  A
generalized Lagrangian is called admissible if each component is.
\item 
  {\rm (Lagrangian branes)} Generally speaking a Lagrangian {\em
    brane} means a Lagrangian with extra structure sufficient for the
  definition of Floer cohomology.  In particular a {\em grading} of a
  Lagrangian submanifold $L \subset M$ is a lift of the canonical
  section $L \to \Lag(M)$ to $\Lag^N(M)$, as in Seidel \cite{se:gr}.
  A {\em brane structure} on a connected Lagrangian $L$ consists of a
  grading and relative spin structure for the embedding $L \to M$
  (equivalent to a trivialization of $w_2(M)$ in the relative chain
  complex for $(M,L)$, see for example \cite{orient}).  A {\em
    Lagrangian brane} is an oriented Lagrangian submanifold equipped
  with a brane structure.  A {\em generalized Lagrangian brane} is a
  generalized Lagrangian correspondence equipped with a brane
  structure.
\end{enumerate} 
\end{definition} 

We define Floer cochains as formal sums of perturbed intersection
points.  Let $(L^0,L^1)$ be a compact, monotone pair of Lagrangian
branes in $M$.  Choose a Hamiltonian $H \in C^\infty([0,1] \times M)$
and denote by $\phi_t \in \Diff(M,\omega)$ the time $t$ flow of the
Hamiltonian vector field $H^\# \in \Map([0,1],\Vect(M))$.  Choose $H$
satisfying the condition that $\phi_1(L^0)$ intersects $L^1$
transversally.  Define the set of {\em perturbed intersection points}
$$
\cI(L^0,L^1) := \bigl\{ x: [0,1] \to M \,\big|\, x(t) 
= \phi_t(x(0)), \ x(0) \in L^0, \  x(1) \in L^1 \bigr\}.
$$ 
The gradings on $L^0,L^1$ induce a {\em degree map}
$$ \cI(L^0,L^1) \to \Z_{N}, \quad x \mapsto |x| $$
given by the Maslov index of a path from the lifts in the Maslov cover
\cite{se:gr}.  By the assumption on the Maslov cover, the mod $2$
degree is determined purely by the orientations.  That is, the mod $2$
degree is $0$ resp. $1$ if the two Lagrangians meet positively
resp. negatively after perturbation.  The generalized intersection
points decompose into subsets of intersections points with fixed
index:
$$ \cI(L^0,L^1) = \bigcup_{d \in \Z_{N}} \cI_d(L^0,L^1), \quad
\cI_d(L^0,L^1) = \{ x \in \cI(L^0,L^1) \ | |x| = d \} .$$
Denote the space of time-dependent $\omega$-compatible almost complex
structures
$$ \J_t(M,\omega) := \Map([0,1], \J(M,\omega)).
$$ 
Given 
$$J \in \J_t(M,\omega), \quad H \in C^\infty([0,1] \times M)$$
as above, we say that a map $u: \R \times [0,1] \to M$ is {\em
  $(J,H)$-holomorphic} iff
$$ \overline{\partial}_{J,H} u(s,t) := \partial_s u({s},t) +
J_{t,u({s},t)}(\partial_tu({s},t) - H^\#_t(u({s},t))) = 0, \quad \forall
(s,t) \in \R \times [0,1] .$$
For any $x_\pm \in \cI(L^0,L^1)$ we denote by $ \M(x_-,x_+)$ the space
of finite energy $(J,H)$-holomorphic maps modulo translation in
$s\in\R$, and $\M(x_-,x_+)_0$ the subset of formal dimension $0$, that
is, index $1$.  The relative spin structures on $L^0,L^1$ induce a map
$$ o : \M(x_-,x_+)_0 \to \{ \pm 1 \} $$
measuring the difference between the orientation on each element $u$
and the canonical orientation of a point \cite{se:bo,orient}.  The Floer cochain complex is
the direct sum
$$ CF(L^0,L^1) = \bigoplus_{x \in \cI(L^0,L^1)} \Z \bra{x} .$$
The Floer coboundary operator $ \partial : \ CF(L^0,L^1) \to
CF(L^0,L^1) $ is
$$ \partial \bra{x_-} := \sum_{x_+\in\cI(L^0,L^1)}
\left( \sum_{[u] \in\M(x_-,x_+)_0} o(u) \right) \bra{x_+} .$$
For the following, see \cite{oh:fl1}, \cite{fooo}.

\begin{proposition} {\rm (Construction of Floer cohomology)}   
Let $M$ be a symplectic background, $L^0,L^1$ admissible Lagrangian
branes and $H$ a time-dependent Hamiltonian making the intersection
transverse.  There exists a comeager subset $\J^\reg_t(M)
\subset \J_t(M)$ such that every Floer trajectory is regular.
In this case, the Floer coboundary $\partial$ is well-defined and
satisfies $\partial^2 = 0$.  The {\em Floer cohomology}
$$ HF(L^0,L^1) := {\ker (\partial)}/{\on{im}}(\partial) $$
is independent of the choice of $H$ and $J$, up to isomorphism.
\end{proposition}  

We wish for our Floer cochain complexes to admit action filtrations.
To obtain these filtrations we pass to versions over Floer cochains
over polynomials in a formal variable.  Namely let 
$$ \Lambda = \left\{ \sum_{j= 1}^k n_j q^{\nu_j}, \ n_j \in \Z, \ \nu_j \in \R \right\} $$ 
denote the space of sums of real powers of a formal variable $q$.  In
this paper we do not need to complete $\Lambda$, that is, the Novikov
ring is not needed.  The Floer cochain complex over the coefficient
ring $\Lambda$ is the direct sum
$$ CF(L^0,L^1;\Lambda) = \bigoplus_{x \in \cI(L^0,L^1)}
\Lambda \bra{x} $$ 
with coboundary incorporating the energies of trajectories: Let
$H_t^\# \in \Vect(M), t \in [0,1]$ be the Hamiltonian vector field of
$H$ and
$$ {\mathcal{E}}_H(u) = \int_{\R \times [0,1]} \omega(\partial_s u(s,t),
\partial_t u(s,t) - H_t^\#(u(s,t))) \d s \d t $$
the perturbed energy.  Then
$$ \partial \bra{x_-} := \sum_{x_+\in\cI(L^0,L^1)} \Bigl( \sum_{[u]
  \in\M(x_-,x_+)_0} o(u) q^{{\mathcal{E}}_H(u)} \Bigr) \bra{x_+} $$ 
As before, the {\em Floer cohomology} with $\Lambda$ coefficients 
$$ HF(L^0,L^1;\Lambda) := {\ker (\partial)}/{\on{im}}(\partial) $$ 
is independent of the choice of $H$ and $J$, up to isomorphism of
graded $\Lambda$-modules.  More generally, for any admissible
generalized Lagrangian correspondence $\ul{L}$ from a point to a point
there is \cite{ww:quiltfloer} a {\em quilted Floer cohomology group}
$HF(\ul{L};\Lambda)$ defined by counting quilted strips.

\begin{remark} \label{spec} {\rm (Specialization)} 
In the monotone setting the natural map
$$ HF(L^0,L^1;\Lambda)/(q - \mu) \to HF(L^0,L^1)$$ 
is an isomorphism for any non-zero $\mu$.  Indeed, for any generators
$x, y \in \cI(L^0,L^1)$ let $c(x,y)$ be the constant in the
monotonicity relation for strips from $x$ to $y$ as in \eqref{Au}.
Choose a base point $x \in \cI(L^0,L^1)$ and consider the map of
$\Lambda$-modules 
$$\iota: CF(L^0,L^1) \to CF(L^0,L^1), \quad \bra{y} \mapsto q^{
  c(x,y)/2} \bra{y} .$$
Additivity of index and energy implies that 
$$c(x,y) + c(y,z ) =
c(x,z), \quad \forall x,y,z \in \cI(L^0,L^1) .$$  
Hence for any intersection point $y \in \cI(L^0,L^1)$,
\begin{eqnarray*}
\partial \iota \bra{y} &=& \sum_{z,[u]} q^{{\mathcal{E}}_H(u) + c(x,y)/2} \bra{z}
\\ &=& \sum_{z,[u]} q^{\lambda I(u)/2 + c(y,z)/2 + c(x,y)/2} \bra{z}
\\ &=& \sum_{z,[u]} q^{\lambda/2 + c(x,z)/2} \bra{z} = q^{\lambda/2}
\iota \partial_1 \bra{y} \end{eqnarray*}
where $\partial_1$ is obtained from $\partial$ by setting $q =1$.  It
follows that the coboundary for $HF(L^0,L^1)$ is, up to an
automorphism, equal to the coboundary for $HF(L^0,L^1;\Lambda)$ up to
a rescaling by a power of $q$.  The claim follows.
\end{remark} 

For later use we recall a basic fact about approximate intersection
points.

\begin{lemma}  \label{approxint} 
Suppose that $M$ is a compact Riemannian manifold and $L^0,L^1 \subset
M$ are compact submanifolds.  For any $\eps > 0$ there exists an
$\delta > 0$ such that if $m \in M$ is a point with $d(m,L^0) <
\delta$ and $d(m,L^1) < \delta$, then $d(m, L^0 \cap L^1) < \eps$.
\end{lemma} 

\begin{proof}   Suppose otherwise.  Then there exists an
$\eps >0$, a sequence $\delta_\nu \to 0$ and a sequence $m_\nu \in M$
  with 
$$d(m,L^0) < \delta_\nu, \quad d(m,L^1) < \delta_\nu, \quad d(m, L^0
  \cap L^1) > \eps .$$
By compactness of $M$, after passing to a subsequence we may assume
that $m_\nu$ converges to a point $m$, necessarily in $L^0$ and $L^1$.
Then 
$$d(m_\nu, L^0 \cap L^1) \to d(m,L^0 \cap L^1) = 0 .$$ 
This is a contradiction.
\end{proof}

It follows that any map that is sufficiently close to both $L^0$ and
$L^1$ at every point in the domain is in fact contained in a small
neighborhood of $L^0 \cap L^1$.

\subsection{Relative invariants for Lefschetz-Bott fibrations} 

We may now associate to Lefschetz-Bott fibrations over surfaces with
strip-like ends relative invariants that are morphisms of Floer
cohomology groups associated to the ends, given by counting
pseudoholomorphic sections.  The following material can also be found,
in a slightly different form, in Perutz \cite{per:lag2}.

\begin{definition}  {\rm (Monotonicity for Lefschetz-Bott fibrations)} 
Let $S$ be a surface without boundary or strip-like ends, and $\pi: \,
E \to S$ a symplectic Lefschetz-Bott fibration.  Let 
$$\Gamma(E): \{ u: S \to E \ | \ \pi \circ u = \on{Id}_S \} $$
denote the set of smooth sections of $E$.
\begin{enumerate} 
\item Given $u \in \Gamma(E)$ with image disjoint from the critical
  set define its {\em index} and {\em symplectic area}
$$ I(u) = 2 (c_1( u^* T^{\on{vert}} E),[S]), \ \ \ A(u) = \int_S u^*
  \omega_E .$$
Note that the form $\omega_E$ is only fiber-wise symplectic.  Thus the
area $A(u)$ may be negative, in the general case.
\item 
A symplectic Lefschetz-Bott fibration $E$ is {\em monotone} with
monotonicity constant $\lambda \ge 0$ if there exists a constant
$c(E)$ such that
$$ \lambda I(u) =  2 A(u) + c(E), \ \ \ \forall u \in \Gamma(E) .$$
\end{enumerate} 
\end{definition}

\begin{remark} {\rm (Behavior of the monotonicity constants under shifts)}  
Given a symplectic Lefschetz-Bott fibration $(\pi: E \to S, \omega_E)$
and a non-negative, compactly supported two-form $\omega_S \in
\Omega^2_c(S)$, another Lefschetz-Bott fibration may be constructed by
replacing $\omega_E$ with $\omega_E + \pi^* \omega_S$.  The symplectic
area of any section changes under this operation by
$$ \lan u^* [\omega_E + \pi^* \omega_S], [S] \ran  - 
\lan u^* [\omega_E ], [S] \ran   = \int_S \omega_S .$$  
The monotonicity constant therefore changes by $-2 \int_S \omega_S$.
\end{remark} 

\begin{proposition}  \label{monotone} 
Let $\pi: E \to S$ be a symplectic Lefschetz-Bott fibration with $S,E$
compact.  If the generic fiber $M$ of $E$ is monotone, $H_1(M) = 0 $
and all vanishing cycles in $E$ have codimension at least $2$ then $E$
is monotone.
\end{proposition}

\begin{proof} 
We first show that the homology classes of any two sections differ by
homology classes in the fibers.  For each critical value $s_i \in
S^\crit$ let $\rho_i: M \to E_{s_i}$ denote the map given by
symplectic parallel transport.  In particular on the vanishing cycles
this map collapses the null foliation.  We claim that for any two
sections $u_0,u_1 \in \Gamma(E)$, the push-forwards $u_{j,*}[S]$
differ by an element in the span of the homology of the fibers
$H_2(M),H_2(E_{s_i})$.  The Leray-F{\'a}ry spectral sequence for the
map $\pi:E \to S$, see F{\'a}ry \cite[Section 10]{far:val}, Bredon
\cite[Section IV.12]{br:sh} allows the computation of the homology as
follows.  Equip the base with the filtration
$$ \emptyset \subset \{ s_0,\ldots, s_n \} \subset S .$$ %
The fibers of $E$ are diffeomorphic over each associated graded piece
$\{ s_0,\ldots, s_n \}, S - \{ s_0,\ldots, s_n \} $.  Denote by
$\cH_q(M)$ the local system of homology groups with fibers
$H_q(\pi^{-1}(z)) \cong H_q(M)$ over $S - \{ s_0,\ldots, s_n \}$.  The
second page in the spectral sequence is the direct sum of relative
homologies of successive spaces in the filtration with values in the
homology of the fiber,
$$ \cC_{p,q} = H_p( S - \{ s_1,\ldots, s_n \}; \cH_q(M)) \oplus
\bigoplus_{i=1}^n H_q(E_{s_i}).$$
Since $H_1(M) = 0$, the contribution from $p = q = 1$ vanishes.  It
follows that degree two homology classes in the total space arise
either as homology classes in the generic fibers invariant under
monodromy; homology classes in the special fibers; or homology classes
of the base.  The projection $\pi_*$ is an isomorphism on the classes
of the last type.  It follows that the degree $2$ part of the kernel
of $\pi_*$ under the projection map is generated by the image of
$H_2(M)$ and $H_2(E_{s_i})$.

We now show the monotonicity relation, using the spectral sequence for
the parallel transport maps to the special fibers.  By the previous
paragraph, it suffices to check monotonicity evaluated on homology
classes of fibers.  For this we show that maps $\rho_{i,*} : H_2(M)
\to H_2(E_{s_i})$ are surjective.  Consider the long exact sequence
\begin{equation} \label{les} H_2(M) \to H_2(E_{s_i}) \to H_2(\Cone(\rho_i)) \to H_1(M) \end{equation} 
where $\Cone(\rho_i)$ is the mapping cone on $\rho_i$.  Since $\rho_i$
is a diffeomorphism away from the vanishing cycle, $\Cone(\rho_i)$
admits a deformation retraction to $\Cone(p_i)$ where $p_i: C_i \to
B_i$ is the projection.  Since $C_i = P_i \times_{SO(c+1)} S^c$, we
have 
$$\Cone(p_i) = (P_i \times_{SO(c+1)} S^{c+1})/ \sim$$ 
where $\sim$ is the equivalence relation that collapses the section
corresponding to the fixed point $(0,\ldots,0, 1)$.  For $c \ge 2$,
the spectral sequence for the fibration $P_i \times_{SO(c+1)} S^{c+1}
\to B_i$ implies that any degree two homology class arises from a
homology class in the base.  On the other hand, the base homology
classes $H_2(B_i)$ have trivial image in the cohomology
$H(\Cone(p_i))$ of the mapping cone.  Hence $H_2(\Cone(p_i)) = 0$.
The long exact sequence \eqref{les} implies that $H_2(M)$ surjects
onto $H_2(E_{s_i})$.
\end{proof}  

We now wish to allow our Lefschetz-Bott fibrations to have strip-like
ends.  The monotonicity conditions in this case will be similar, but
with additional constants depending on the limits.

\begin{definition} {\rm (Symplectic Lefschetz-Bott fibrations with strip-like ends)}  
\begin{enumerate} 
\item Let $S$ be a complex curve with boundary obtained from a compact
  curve with boundary $\ol{S}$ be removing points on the boundary
  $z_1,\ldots,z_n$.  A {\em strip-like end} for the $j$-th puncture of
  $S$ is a holomorphic map
$$\eps_j: [0,\pm \infty) \times [0,1] \to S$$ 
such that $\exp(2 \pi( ( \eps_j^{-1})_{1} + i (\eps_{j}^{-1})_{2} ))$
is a local holomorphic coordinate on the closure of the image of
$\eps_j$; the end is called {\em incoming} resp.\ {\em outgoing} if
the sign is negative resp.\ positive.  A {\em collection of strip-like
  ends} $\E $ is a set of strip-like ends, one for each $j = 1,\ldots,
n$.  We write $\E = \E_- \cup \E_+$ the union of the incoming and
outgoing ends.
\item A symplectic {\em Lefschetz-Bott fibration over a surface with
  strip-like ends} $S$ with fiber given by a symplectic manifold $M$
  is a Lefschetz-Bott fibration $E \to S$ with $S^\crit$ contained in
  the interior of $S$, together with a trivialization
$$\varphi_{S,e}: \eps_{S,e}^* E \to 
(0,\pm \infty) \times [0,1] \times
M$$
for each end $e \in \mE$, such that $ \varphi_{s,e}^* \omega_E =
\pi_M^* \omega_M $ where 
$$\pi_M: (0,\pm \infty) \times [0,1] \times M
\to M$$ 
is projection on the last factor.
\item Let $S$ be a surface with strip-like ends and $\pi: E \to S$ a
  symplectic Lefschetz-Bott fibration with fiber $M$.  A {\em
    Lagrangian boundary condition} for $E$ is a submanifold $F \subset
  \partial E$ such that
\ben
\item $\pi|F $ is a fiber bundle over $\partial S$;
\item each fiber $F_z \subset E_z, z \in \partial S$
is a Lagrangian submanifold; 
\item for each $e \in \mE$ there exist Lagrangian submanifolds
  $L^{0,e},L^{1,e} \subset M$ such that $F$ is constant sufficiently
  close to $z_e$ that is, 
$$ \varphi_{S,e}(F_{\eps_{S,e}(s,j)}) =
  L^{j,e}, \ \ \ \pm s \gg 0 ;$$ 
and
\item for each $e \in \mE$, the intersection $L^{0,e} \cap L^{1,e}$ is
  transversal.  \een
\end{enumerate} 
\end{definition}

\begin{definition} 
\begin{enumerate} 
\item 
{\rm (Monotonicity condition for Lefschetz-Bott fibrations with
  boundary)} Let $S$ be a compact surface with boundary, $(E,F)$ a
bundle with boundary condition, and $\Gamma(E,F)$ the set of smooth
sections $u: (S,\partial S) \to (E,F)$.  Each $u \in \Gamma(E,F)$
takes values in the smooth locus of $E$.  By pull-back one obtains
bundles
$$u^* T^{\on{vert}}E \to S, \quad (u | \partial S)^* T^{\on{vert}} F \to
\partial S .$$  
Taking the Maslov index of this pair gives rise to a {\em Maslov
  index} map
$$ I: \Gamma(E,F) \to \Z .$$
We also define the {\em symplectic area}
$$ A: \Gamma(E,F) \to \R, \ \ u \mapsto \int_S u^* \omega_E $$
keeping in mind that the form $\omega_E$ is only symplectic on the fibers.
A pair $(E,F)$ {\em monotone with monotonicity constant $\lambda$} if
the index depends linearly on the area; that is,
\begin{equation} \label{IA}
 \lambda I(u) = 2 A(u) + c(E,F), \ \ \forall u \in \Gamma(E,F) \end{equation}
for some constant $c(E,F)$.  
\item {\rm (Linearized operator)}
Suppose that $S$ is a surface with boundary and
strip-like ends $\mE = \mE_- \cup \mE_+$, and $(E,F)$ a bundle with
boundary condition.  For any collection %
$$ (x_e \in \cI(L^{0,e},L^{1,e}))_{e \in \mE} ,$$ 
let
$$ \Gamma(E,F;(x_e)_{e \in \mE}) = \left\{ u \in \Gamma(E) \, \left| \, \quad 
\begin{array}{c} \lim_{s \to
  \pm \infty} u(\eps_e(s,t)) = x_e , \forall e \in \E \\  u |_{\partial S}
\in \Gamma(F) \end{array} \right. \right\} $$
denote the space of sections with boundary values in $F$ and
asymptotic limits $x_e$.  For sections $u \in \Gamma(E,F)$ let
\begin{equation} \label{cr2}
 \olp u = \hh ( \d u + J(u) \circ \d u \circ j) = 0 \end{equation}
denote the Cauchy-Riemann equation associated to the pair $(j,J)$. The
linearized operator
\begin{multline} \label{Du}
D_u: \Omega^0( u^* T^{\on{vert}} E, u^* T^{\on{vert}} F) \to
\Omega^{0,1}(u^* T^{\on{vert}} E), \quad \xi \mapsto \ddt |_{t =0}
\Pi_{t\xi}^{-1} \olp \exp_u(t\xi) \end{multline}
is given by differentiating the Cauchy-Riemann operator along a path
$\exp_u(t\xi)$ of geodesic exponentials, and using parallel transport
$\Pi_{t\xi}^{-1}$ back to $u$.  The operator $D_u$ is Fredholm since
the boundary conditions at infinity are assumed transversal.
\item
{\rm (Monotonicity condition for Lefschetz-Bott fibrations with
  strip-like ends)} The pair $(E,F)$ is {\em monotone} with
monotonicity constant $\lambda \ge 0$ if for any $(x_e)_{e \in \mE}$
there exists a constant $c(E,F;(x_e)_{e \in \mE})$ such that
\begin{equation} \label{IA2}
 \lambda \Ind(D_u) = 2 A(u) + c(E,F;(x_e)_{e \in \mE}), \ \ \forall u \in
\Gamma(E,F;(x_e)_{e \in \mE}) .\end{equation}
\end{enumerate} 
\end{definition} 

\begin{proposition} \label{monotonebound}
 \label{monotoneends}  {\rm (Condition for monotonicity 
of the fiber to imply monotonicity of a fibration with strip-like
ends)} Let $E \to S$ be a symplectic Lefschetz-Bott fibration over a
 surface with boundary $S$ with connected and simply-connected fibers.
 Let $F \subset E | \partial S$ be a Lagrangian boundary condition
 with connected and simply-connected fibers.  Suppose that the generic
 fiber of $E$ is monotone and the vanishing cycles of $E$ have
 codimension at least $2$.  Then $(E,F)$ is monotone.
\end{proposition}  

\begin{proof}  First consider the case without ends.  As in 
Proposition \ref{monotone}, any two sections of $E$ differ by a
homology sphere $H_2(E_s)$ in some fiber $E_s $ of $E$, and these are
generated by classes in the generic fiber by the codimension
assumption for which monotonicity holds.  In the case with ends, as in
\cite{ww:quilts}, one may use gluing to reduce to the case that $S$
has no strip-like ends as in the statement of the Proposition.  More
precisely, fix a section $v: S \to E$ with limits $\ul{x}_\pm$.  View
$v$ as a map from the surface $S^- = (S,-j)$ with complex structure
reversed.  Then for any other section $u: S \to E$ with the same
limits, gluing along the strip-like ends produces a section of the
doubled surface $w: S \# S^- \to E \# E^-$.  Then
\begin{eqnarray*}
\Ind(D_u) &=& \Ind(D_w) - \Ind(D_v) \\ &=& I(w) + \dim(E) \chi(S \#
S^-)/2 - \Ind(D_v) \\ &=& \lambda A(w) + \dim(E) \chi(S \# S^-)/2 -
\Ind(D_v) \\ &=& \lambda A(u) + \lambda A(v) + \dim(E) \chi(S \# S^-)/2 -
\Ind(D_v) .\end{eqnarray*}
Since the last three terms are independent of $u$, this proves the
statement of the Proposition.\end{proof}

We now turn to the construction of relative invariants associated to
symplectic Lefschetz-Bott fibrations.  Let $\pi: E \to S$ be such a
fibration, equipped as in Definition \ref{fibdef} with a complex
structure $j_0$ on a neighborhood of the critical values in $S$ and
$J_0$ on a neighborhood of the critical points in $E$.  In the case
that $S$ has ends, we assume furthermore that almost complex
structures $J_e \in \J(M)$ are fixed making the moduli spaces of Floer
trajectories for that end regular.  We wish to extend these to almost
complex structures on the entire fibration and base, so that we may
define moduli spaces of pseudoholomorphic sections.

\begin{definition} \label{compatac}
\begin{enumerate} 
\item {\rm (Compatible almost complex structures)} Let $\pi: E \to S$
  be a Lefschetz-Bott fibration over a surface with strip-like ends
  $S$.  A complex structure $j$ on $S$ is {\em compatible} with $E$ if
  $j = j_0$ in an neighborhood of $S^\crit$.  An almost complex
  structure $J$ on $E$ is {\em compatible} with $\pi,j$ iff
\ben 
\item $J = J_0$ in a neighborhood of $E^\crit$;
\item $\pi$ is $(J,j)$-holomorphic in a neighborhood of $E^\crit$,
  that is, $ J \circ \d \pi = \d \pi \circ j$; and
\item $\omega_E(\cdot,J \cdot)$ is symmetric and positive
definite on $TE^v_x$, for any $x \in E$.
\item $J$ is equal to fixed almost complex structure $j \times J_e, e
  \in \cE(S)$ on the ends of $S$ in the given trivializations on a
  neighborhood $U_e \subset S$ of each end.  \een
Let $\J(E)$ denote the set of $(\pi,j)$-compatible almost complex
structures.  
To see that $\J(E)$ is non-empty, note that the symplectic form
defines a connection $\ker(D \pi)^{\omega}$ on $E \to S$ away from the
singular locus, given by the symplectic perpendicular of the vertical
part of the tangent space. First choose $j$ equal to $j_0$ near the
critical values.  Then take $J$ to equal $j$ on the horizontal
subspace, and any almost complex structure compatible with the
symplectic form on the vertical subspace.
\item {\rm (Moduli space of pseudoholomorphic sections)} Let
  $\M(E,F;(x_e)_{e \in \mE})$ denote the set of finite area sections
  $u \in \Gamma(E,F,(x_e)_{e \in \E})$ such that $u$ is
  $(j,J)$-holomorphic, with limits $(x_e)_{e \in \E}$ along the ends.
\end{enumerate}
\end{definition}  

\begin{theorem}  \label{EF} {\rm (Existence of regular almost complex structures
for symplectic Lefschetz-Bott fibrations)} Suppose that $(E,F)$ is a
  monotone symplectic Lefschetz-Bott fibration with Lagrangian
  boundary conditions over a surface $S$ with strip-like ends.  There
  exists a comeager subset $\J^\reg(E) \subset \J(E)$ such that
\ben
\item $\M(E,F;(x_e)_{e \in \mE})$ is a smooth manifold with tangent space
  at $u$ given by $\ker(D_u)$;
\item \label{zero} the zero-dimensional component 
$\M(E,F;(x_e)_{e \in \mE})_0 \subset \M(E,F;(x_e)_{e \in \mE})$
is finite; 
\item \label{one} the one-dimensional component $\M(E,F;(x_e)_{e \in
  \mE})_1 \subset \M(E,F;(x_e)_{e \in \mE})$ has a compactification
  with boundary
$$ \bigcup_{(x_e)_{e \in \mE},f,x_f'} \M(E,F;(x_e)_{e \in \mE}, x_f \mapsto x_f')_0 \times
\M(x_f,x_f')_0 $$
consisting of pairs of a section with bubbled-off trajectory; and
\item \label{orient} any relative spin structure on $(E,F)$ induces a
  set of orientations on the manifolds $\M(E,F;(x_e)_{e \in \mE})_0$
  that are coherent in the sense that they are compatible with the
  gluing maps from (c) in the sense that the inclusion of the boundary
  in (c) has the signs $(-1)^{\sum_{e<f} |x_{e}^-|}$ (for incoming
  trajectories) and $-(-1)^{\sum_{e<f} |x_{e}^+|}$ (for outgoing
  trajectories.)
\een
\end{theorem}

The proof is similar to that of \cite{se:lo} in the exact case.
Bubbling for sections can occur only in the fiber.  So sphere and disk
bubbling on the zero and one-dimensional moduli spaces is ruled out by
the monotonicity condition.  The construction of coherent orientations
is given in \cite{orient}.

\begin{remark} \label{sketch} {\rm (Sketch of construction of orientations from \cite{orient})} 
We briefly recall the construction of orientations: On any disk with
Lagrangian boundary conditions where the Lagrangian is equipped with a
relative spin structure, the linearized operator $D_u$ over the
surface $S$ may be deformed via nodal degeneration to a boundary value
problem $D_u'$ on the sphere $S_s$ and a constant boundary value
problem on the disk $S_d$. On the sphere the linearized Cauchy-Riemann
operator $D_u' | S_s$ is homotopic to a complex linear operator
$D_u''$.  The determinant line $\det(D_u')$ inherits an orientation
via the complex structure on the kernel $\ker(D_u'')$ and cokernel
$\coker(D_u'')$.  On the other hand, on the disk the boundary
condition $(u |_{\partial S_d})^* TL$ admits a canonical stable
trivialization determined by the relative spin structure.  The
determinant line on the disk admits an orientation induced from a
trivialization of $(u |_{\partial S_d})^* TL$ and an isomorphism of
the kernel $\ker(D_u' | S_d)$ with the tangent space to the Lagrangian
$T_{u(z)}L$ at any point on the boundary.  These combine to an
orientation on the determinant line $\det(D_u) \cong \det(D_u' | S_s)
\otimes \det(D_u' | S_d)$.
\end{remark}

\begin{remark} {\rm (Brane structures on spherically fibered coisotropics)}  
In order to specify the signs in the exact triangle it is necessary to
specify the brane structure on the fibered coisotropic.  If $C$ is a
spherically fibered coisotropic that reduces to a {\em spin} principal
bundle, then $TC$ admits a relative spin structure for the embedding
$C \to M \times B$ with background class $w_2(TM)$ induced by the
isomorphism
$$TC \oplus \ul{\R} \cong \pi^* TB \oplus P ( TS^c \oplus \ul{\R})
\cong \pi^* TB \oplus P(\ul{\R}^{c+1})$$
where $\ul{\R}$ denotes the trivial bundle with fiber $\R$.  
\end{remark}

For later use we recall the basic fact that maps with sufficiently
small energy must also have small diameter:

\begin{lemma}   \label{smalllength}  Let $\ul{S}$ be a holomorphic quilt with strip-like 
ends, with symplectic labels $\ul{M}$ and Lagrangian boundary and seam
condition $\ul{L}$.  For any $\delta > 0, \ell_0$, there exists $\eps
> 0$ such that if $u: \ul{S} \to \ul{M}$ is a pseudoholomorphic quilt
with energy $E(u) < \eps$, and $\gamma: [0,1] \to \ul{S}$ is a path
connecting boundary components of length less than $\ell_0$, then the
length of $u \circ \gamma$ is less than $\delta$.
\end{lemma}

\begin{proof} 
The claim is an application of the mean value inequality for
pseudoholomorphic maps \cite[Lemma 4.3.1 (i)]{ms:jh}: For any compact
almost complex manifold $M$ there exist constants $c,\eps > 0$ such
that for any $r > 0 $ and any pseudoholomorphic map $u: B_r(z) \to M$
with energy at most $\eps$,
\begin{equation} \label{mvt}
| du(z) |^2 \leq \frac{c}{r^2} \int_{B_r(z)} | du(w)|^2 \d w
\end{equation}
where $B_r(z)$ is a ball of radius $r$ around $z$.  There is a similar
version for Lagrangian boundary conditions, in which the
right-hand-side is replaced by a half-ball \cite[Lemma 4.3.1
  (i)]{ms:jh}.  In our quilted situation, we may apply the mean value
inequality in each patch.  Since $\ul{S}$ is compact on a complement
of the strip-like ends, there exists a constant $r_0 > 0 $ such that
any point in $\ul{S}$ is contained in ball of radius at least $r_0$,
or half-ball of radius $r_0$ in $\ul{S}$.  By integrating over
$\gamma(t)$ we see that the length $\ell(\gamma)$ satisfies
\begin{eqnarray*}  \ell(u \circ \gamma) &=& \int_0^1 | \d (u \circ \gamma)(t))) | \d t
\\ &\leq& \int_0^1 \frac{c}{r_0^2} \left(\int_{B_{r_0}(\gamma(t))} |
\d u(w)|^2 \d w \right)^{1/2} | \gamma'(t)| \d t \\ &\leq& c r_0^{-2}
E(u) \ell(\gamma) \leq c r_0^{-2} \eps \ell_0 . \end{eqnarray*}
This proves the claim for $\eps$ such that $ \eps < \delta r_0^2/ c
\ell_0$.
\end{proof}

\begin{definition}  {\rm (Relative invariants for Lefschetz-Bott fibrations with 
strip-like ends)} For Lefschetz-Bott fibration $E$ with boundary
  condition $F$ equipped with a relative spin structure, suppose that
  the ends $L^{e,0}, L^{e,1}$ of $F$ are transverse and $CF(L^{e,0},
  L^{e,1})$ have been defined without Hamiltonian perturbation for
  each end $e$.  Define
$$ C\Phi(E,F;\Lambda): \bigotimes_{e \in \mE_-}
  CF(L^{e,0},L^{e,1};\Lambda) \to \bigotimes_{e \in \mE_+}
  CF(L^{e,0},L^{e,1};\Lambda) $$
by 
\begin{equation} \label{CPhi}
 C\Phi(E,F;\Lambda) \left( \otimes_{e \in \mE_-} \bra{x_e} \right) =
 \sum_{u \in \M(E,F;(x_e)_{e \in \mE})_0} o(u) q^{A(u)}
\left(
 \otimes_{e \in \mE_+} \bra{x_e} \right) \end{equation}
where $o(u) = \pm 1$ are orientations constructed in \cite{orient} and
in the dimension one case we allow a $ \{ \pm 1 \}$-valued local
system on $F$ and weight the contributions by the holonomies.  By
items \eqref{one} and \eqref{orient} of Theorem \ref{EF}, the maps
$C\Phi(E,F;\Lambda)$ are cochain maps.  Passing to cohomology (and
passing to rational coefficients in the case of more than one outgoing
end) one obtains a map
\begin{equation}\label{HPhi}
 \Phi(E,F;\Lambda): \bigotimes_{e \in \mE_-}
 HF(L^{e,0},L^{e,1};\Lambda) \to \bigotimes_{e \in \mE_+}
 HF(L^{e,0},L^{e,1};\Lambda) \end{equation}
\end{definition}  

\begin{remark} \label{depend}
 {\rm (Independence from almost complex structure and fiber-wise
   symplectic form)} The cohomology-level invariants
 $\Phi(E,F;\Lambda)$ are independent of the choice of compatible
 almost complex structure $J$ on $E$, by an argument using
 parametrized moduli spaces similar to that of Theorem \ref{EF}.
 
The invariants are independent of the choice of two-form $\omega_E \in
\Omega^2(E)$ in the following sense: Given any symplectic
Lefschetz-Bott fibration $(E, \omega_E,\pi)$ we may form a new
symplectic Lefschetz-Bott fibration $(E,\omega_E + \pi^* \omega_S,
\pi)$ by adding on the pull-back of non-negative, compactly supported
two-form $\omega_S \in \Omega^2_c(S)$ on the base.  Any almost complex
structure $J_0$ compatible with $\omega_E$ will also be compatible
with $\omega_E + \pi^* \omega_S$, although not necessarily vice-versa.
As a result, for any two such almost complex structures $J_k, k \in \{
0,1 \}$ compatible with $\omega_E + \lambda_k \pi^* \omega_S$ for some
scalars $\lambda_0,\lambda_1$, the moduli spaces 
$$\M(E,F,J_0;(x_e)_{e
  \in \mE})_0 \sim \M(E,F,J_1;(x_e)_{e \in \mE})_0 $$
are cobordant.  Since the area of each pseudoholomorphic section
changes by the integral $A(S) = \int_S \omega_S$, the invariant
$\Phi(E,F;\Lambda)$ is independent of this change up to an overall
power $q^{A(S)}$.
\end{remark}

The relative invariants of fibrations with ``non-negative curvature'',
in the following sense, have particularly nice properties.  Recall the
symplectic connection \eqref{symconn} on a Lefschetz fibration $\pi:E
\to S$.  The spaces $TE^h_e$ have canonical complex structures,
induced from the complex structure $j$ on the base $S$.  Define
$$ \omega_{E,e} | T^h E = f(e) \pi^* \omega_S $$
for some function $f: E \to \R$.  

\begin{definition}  A Lefschetz-Bott fibration $E$ with
two-form $\omega_E$ has {\em non-negative curvature} if
$\omega_E(v,jv) \ge 0 $ for all $v$ in the horizontal subspace $TE^h$,
that is, $f(e) \ge 0$ for all $e \in E$.
\end{definition} 

\begin{remark}  Non-negative curvature implies that a small perturbation 
of the two-form is symplectic: Recall that the total space of any
Lefschetz-Bott fibration $\pi:E \to S$ admits a canonical isotopy
class of symplectic structures given as follows. If $\omega_S \in
\Omega^2(S)$ is a sufficiently positive two-form then $\omega_E +
\pi^* \omega_S$ is a symplectic form on $TE_e$ for any $e \in E$.  If
$E$ is compact, then $\omega_E + \pi^* \omega_S$ is symplectic on $E$
for $\omega_S$ sufficiently positive.  If $E$ is non-negative, then
$\omega_E + \pi^* \omega_S$ is symplectic for {\em any } positive form
$\omega_S \in \Omega^2(S)$.
\end{remark} 

\begin{proposition} \label{EC} {\rm (Non-negative curvature of standard Lefschetz-Bott
fibrations)} If $C \subset M$ is a spherically fibered coisotropic,
  then the standard Lefschetz-Bott fibration $E_C$ of \ref{ECexist}
  has non-negative curvature.
\end{proposition}

\begin{proof}  Let $v \in V$, the standard representation of $SO(c+1)$, and $(p,v) \in P \times
  V$.  The horizontal subspace $H_v \subset T_v V$ pairs trivially
  with $\ker(\alpha) \times T(\pi^{-1}_V(\pi_V(v)))$ under the pairing
  given by the two-form \eqref{mincoup}, where $\pi_V$ is the
  projection \eqref{piV}.  It follows that the image $[H_v]$ of $H_v$
  in $P(V)$ is the horizontal subspace at $[p,v] \in P(V) := (P \times
  V)/G$.  Let $J_V$ denote the standard complex structure on $V$, and
  $J_0$ the induced complex structure on $E$. Since $J_{V}$ is
  non-negative on $H_v$, $J_0$ is non-negative on $[H_v]$.
\end{proof}  

\begin{proposition} Let $E$ be a Lefschetz-Bott fibration with Lagrangian boundary
condition $F$ and relative spin structure.  If $E$ has non-negative
curvature, then the exponents of $q$ in the formula \eqref{CPhi}
are all non-negative.
\end{proposition} 

\begin{proof}   Since the form $\omega_E(\cdot, J \cdot)$ is non-negative
for any $J \in \J(E)$, any pseudoholomorphic section has non-negative
area.  The $q$-exponents in \eqref{CPhi} are the areas, and so are
non-negative as well.
\end{proof}

We do not give formula for the degree of the relative invariant.  See
\cite{ww:quilts} for a formula for the degree in the case without
singularities.

\subsection{Invariants for quilted Lefschetz-Bott fibrations}

The main difference between the triangle for the fibered case and the
original Seidel triangle \cite{se:lo} is the appearance of invariants
associated to quilted surfaces.  The following definitions are taken
from \cite{ww:quilts}.

\begin{definition}
\begin{enumerate} 
\item {\rm (Quilted surfaces with strip-like ends)} A {\em quilted
  surface with strip-like ends} $\ul{S}$ consists of the following
  data:
\ben 
\item a collection $\ul{S} = (S_k)_{k=1,\ldots,m}$ of surfaces with
  strip-like ends, see \cite{se:lo}, \cite{ww:quilts}, called {\em
    patches}.  Each $S_k$ carries a complex structure $j_k$ and has
  strip-like ends $(\eps_{k,e})_{e\in \E(S_k)}$.  Denote the limits of
  these ends
$$\lim_{s \to \pm
    \infty} \eps_{k,e}(s,t) =: z_{k,e} \in \partial\overline{S}_k$$
and denote the boundary components 
$$\partial S_k = (I_{k,b})_{b\in
  \B(S_k)} ;$$
\item 
a collection $\S$ of {\em seams:} pairwise disjoint $2$-element
subsets
$$ \sigma \subset \bigcup_{k=1}^m \bigcup_{b \in \B(S_k)} I_{k,b} $$
and for each $\sigma = \{ I_{k,b}, I_{k',b'} \}$ a real analytic
isomorphism 
$$\varphi_\sigma: I_{k,b} \to I_{k',b'} ;$$
where the isomorphisms $\varphi_\sigma$ should be compatible with the
strip-like ends in the sense that on each end $\varphi_\sigma$ should
be a translation;
\item
Orderings of the boundary components $\B(S_k), k = 1,\ldots, m$ of
each patch; and
\item 
Orderings of the incoming and outgoing ends of $\ul{S}$
$$\E_-(\ul{S})=(\ul{e}^-_1,\ldots,\ul{e}^-_{N_-(\ul{S})}), \quad \E_+(\ul{S})=(\ul{e}^+_1,\ldots,\ul{e}^+_{N_+(\ul{S})}) .$$
\een
\item {\rm (Quilted symplectic Lefschetz-Bott fibrations)} A {\em
  (quilted) symplectic Lefschetz-Bott fibration} $\ul{E}$ over a
  quilted surface $\ul{S}$ with strip-like ends consists of a
  collection of Lefschetz-Bott fibrations $E_k \to S_k, k = 0,\ldots,
  m$.  A {\em Lagrangian boundary/seam condition} for $E$ consists of
  a collection $\ul{F}$ of submanifolds of the boundaries and seams
$$ F_{k,b} \subset E_{k,b} | I_{k,b}, \quad F_{(k_0,b_0),(k_1,b_1)}
\subset E_{k_0}| I_{k_0,b_0} \times \varphi_\sigma^* (
E_{k_1} | I_{k_1,b_1}) $$
where $I_{k,b}$ ranges over true boundary components
resp.\ $I_{k_0,b_0}, I_{k_1,b_1}$ range over identified boundary
components, such that
\begin{enumerate}  
\item each fiber 
$$F_{(k_0,b_0),(k_1,b_1),z} \subset E_{k_0,z} \times
  (\varphi_{(k_0,b_0),(k_1,b_1)}^* E_{k_1})_z $$ 
over $z \in \ul{S}$
is a Lagrangian submanifold; and
\item over the strip-like ends the the fibers
  $F_{(k_0,b_0),(k_1,b_1),z}$ over $z \in \ul{S}$ are given by fixed
  Lagrangians $L^{k_{e,i},b_{e,i}}$ on the strip-like ends, with the property
  that the composition
$$ L^{(k_{e,0},b_{e,0})}\circ L^{(k_{e,1},b_{e,1})} \ldots \circ
L^{(k_{e,l(e)},b_{e,l(e)})} $$ 
is transversal, where $l(e)$ is the number of patches on the end $e$.
\end{enumerate} 
\end{enumerate} 
\end{definition} 

We say that a quilted Lefschetz-Bott fibration is {\em monotone} if
sections satisfy an area-index relation similar that for
pseudoholomorphic maps.  In the case without singularities (that is,
fibrations) admissibility for the Lagrangians guarantees monotonicity
in the quilted setting, see \cite[Remark 3.7]{ww:quilts}.  For
Lefschetz-Bott fibrations, admissibility together with the codimension
conditions in Lemma \ref{monotonebound} guarantee monotonicity, by the
same arguments.

\begin{definition} \label{relinv} {\rm (Relative invariants for quilted Lefschetz-Bott fibrations)} 
Theorem \ref{EF} generalizes to the quilted setting: associated to a
monotone quilted symplectic Lefschetz-Bott fibration $(\ul{E},\ul{F})$
we have (ungraded) relative invariants which (working with rational
coefficients in the case of more than one outgoing end) maps
\begin{multline}
 \bigotimes_{e \in \mE_-}
 HF(L^{(k_{e,0},b_{e,0})},L^{(k_{e,1},b_{e,1})}, \ldots ,
 L^{(k_{e,l_-(e)},b_{e,l_-(e)})};\Lambda) \\ \to \bigotimes_{e \in
   \mE_+} HF(L^{(k_{e,0},b_{e,0})}, L^{(k_{e,1},b_{e,1})},\ldots ,
 L^{(k_{e,l_+(e)},b_{e,l_+(e)})};\Lambda) .
\end{multline}
These invariants satisfy a composition relation for gluing along
strip-like ends by a gluing argument spelled out in \cite{ww:quilts}.
That is, if $\ul{E},\ul{F}$ is a quilted Lefschetz-Bott fibration
obtained from fibrations $\ul{E}_1,\ul{F}_1$ and $\ul{E}_2,\ul{F}_2$
by gluing the outgoing ends of $\ul{E}_1,\ul{F}_1$ to the incoming
ends of $\ul{E}_2,\ul{F}_2$ then
\begin{equation} \label{gluealongends} \Phi(\ul{E},\ul{F}) = \Phi(\ul{E}_1,\ul{F}_1) \circ
\Phi(\ul{E}_2,\ul{F}_2) .\end{equation}
\end{definition} 

As in the unquilted case in Remark \ref{depend}, these invariants
are independent of the choice of almost complex structure and
deformation of the two-form on the total space.

\subsection{Vanishing theorem} 

In this section we glue along a seam to obtain a vanishing theorem
analogous to \cite[Section 2.3]{se:lo} for the invariants associated
to standard fibrations associated to a fibered Dehn twist.

\begin{remark} {\rm (Gluing along a seam for quilted Lefschetz-Bott fibrations)} 
\begin{enumerate} 
\item
{\rm (Glued surface)} For $k =0,1$ let $\ul{S}_k$ be quilted surfaces
with $d_k + 1$ strip-like ends, and $z_k$ a seam point in $\ul{S}_k$.
Let $\rho > 0$ be a {\em gluing parameter}.  Let $\ul{S}^\rho$ be the
quilted surface with $d_0 + d_1 + 2$ strip-like ends formed by gluing
together quilted disks $D_0,D_1$ around $z_0,z_1$ using the map $z
\mapsto \rho/z$.  See Figure \ref{gseam}.

\begin{figure}[ht]
\includegraphics[width=5in]{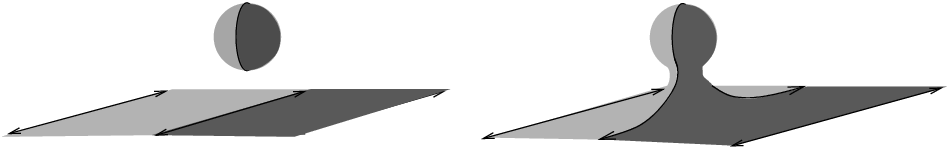}
\caption{Gluing along a seam} 
\label{gseam}
\end{figure} 

\item {\rm (Glued bundles)} Let $(\ul{E}_k,\ul{F}_k)$ be
  Lefschetz-Bott fibrations over $\ul{S}_k$, equipped with a
  trivialization of $\ul{E}_k,\ul{F}_k$ in a neighborhood of $z_k$,
  and a symplectomorphism of $(\ul{E}_{k,z_k},\ul{F}_{k,z_k})$ for $k
  = 0,1$.  The {\em seam connect sum} $\ul{E}^\rho \to \ul{S}^\rho$ is
  formed by patching $\ul{E}_0$ and $\ul{E}_1$, and similarly for the
  boundary and seam conditions $\ul{F}^\rho$.
\item {\rm (Glued complex structures)} Suppose that the following are
  given: a $(\pi_k,j_k)$-compatible almost complex structure
  $\ul{J}_k$ for $\pi_k:\ul{E}_k \to \ul{S}_k$ that is constant in a
  neighborhood of $z_k$ (with respect to the given trivialization) for
  $k \in \{ 0,1 \}$ and such that $\ul{J}_0$ agrees with $\ul{J}_1$ on
  the glued fiber.  One can patch together these almost complex
  structures to obtain a compatible almost complex structure $\ul{J}$
  for $\ul{E}\to S$.
\end{enumerate} 
\end{remark} 

We next investigate the behavior of moduli spaces of sections under
gluing.  Suppose that $(\ul{E}^\rho,\ul{F}^\rho)$ is obtained from
gluing $(\ul{E}_0,\ul{F}_0) $ and $(\ul{E}_1,\ul{F}_1)$ and all three
are monotone symplectic Lefschetz-Bott fibrations.  Let
$$ \ev_k: \ \M(\ul{E}_k,\ul{F}_k) \to
\ul{F}_{k,z_k} $$
denote the evaluation maps at the nodal points $z_k,k \in \{ 0,1 \}$.

\begin{theorem}  \label{gluing} {\rm (Behavior of moduli spaces under gluing along a seam)} 
For sufficiently small $\rho$, there exists a comeager subset
$$(\J(\ul{E}_0,\ul{F}_0) \times \J(\ul{E}_1,\ul{F}_1))^\reg \subset
\J(\ul{E}_0,\ul{F}_0)^\reg \times \J(\ul{E}_1,\ul{F}_1)^\reg$$ 
such that
\begin{enumerate} 
\item the evaluation map $\ev_0 \times \ev_1$ is transverse to the
  diagonal;
\item for any pair $(u_0,u_1)$ there exists a {\em gluing map} on a
  neighborhood $U(u_0,u_1)$ of $(u_0,u_1)$ given by
$$ \Theta^\rho: \M(\ul{E}_0,\ul{F}_0) \times_{\ev_0,\ev_1}
  \M(\ul{E}_1,\ul{F}_1) \supseteq U(u_0,u_1) \to
  \M(\ul{E}^\rho,\ul{F}^\rho) ;$$
\item as $(u_0,u_1)$ varies over points in the zero-dimensional
  component of the left-hand-side, $\Theta^\rho$ is surjective onto
  the zero-dimensional component of $\M(\ul{E}^\rho,\ul{F}^\rho)$; and
\item for any $u \in \M(\ul{E}_0,\ul{F}_0) \times_{\ev_0,\ev_1}
  \M(\ul{E}_1,\ul{F}_1)$, the sequence $\Theta^\rho(u)$ Gromov
  converges to $u$ as $\rho \to 0$, that is, converges up to sphere
  bubbling, disk bubbling, and bubbling off of Floer trajectories on
  the strip-like ends.
\end{enumerate} 
\end{theorem}

See McDuff-Salamon \cite[Chapter 10]{ms:jh} for the case of gluing at
an interior point, and Abouzaid \cite{ab:ex} for the details of gluing
along a point in the boundary.  

Next we give a formula for the dimension for the pseudoholomorphic
sections of the standard fibration studied in Propositions
\ref{ECexist} and \ref{EC}.

\begin{definition} \label{smallstandard} {\rm (Standard fibrations on small balls)} 
Let $C \subset M$ be a coisotropic submanifold of codimension $c$
spherically fibered over $B$ and $E_C \to D$ the Lefschetz-Bott
fibration over the disk $D$ of radius $1$ with generic fiber $M \times
B^-$ and monodromy $\tau_C \times 1$ from Lemma \ref{ECexist}.  Equip
$E_C$ with the Lagrangian boundary condition given by the fiber
product
$$F_C := P(T) \, | \, \partial D$$ 
where $T$ is the union of vanishing cycles in the local model $P(V)$
as in \eqref{Sig}.  As in \cite{se:lo} this boundary value problem
fits into a family of problems $E_{C,r},F_{C,r}$, the standard
fibrations of Section \ref{standard} over a disk $D_r$ of radius $r$.
Each member of the family is formed by patching together $C
\times_{SO(c+1)} V$ with $(M \times B) \ssm (i(C) \times p(C))$ with 
boundary condition 
$$F_{C,r} = \bigcup_{z \in \partial D_r } \left( P \times_{SO(c+1)}
\sqrt{z} S^c \right) $$
in the local model $ P \times_{SO(c+1)} V$.
\end{definition}  

\begin{lemma} \label{above} 
{\rm (Dimension formula for pseudoholomorphic sections of a standard
  fibration)} Let $C,E_C,F_C$ be the standard Lefschetz-Bott fibration
associated to $C$ in Lemma \ref{ECexist}, and $c$ the codimension of
$C$.  If $(E_C,F_C)$ is monotone then for $r$ sufficiently small, $J
\in \J^\reg(E_C)$ and all $u \in \M(E_{C,r},F_{C,r})$ we have
$$\dim T_u \M(E_{C,r},F_{C_r})
 \ge \dim(C) + (c-1).$$
\end{lemma}

\begin{proof} 
We explicitly describe a family of pseudoholomorphic sections as
follows.  Suppose that the almost complex structure on $E_{C,r}$ is
induced from an almost complex structure on $B$ and the standard
almost complex structure on $V$.  For each $a \in V, b \in B$ fix a
local trivialization of $P$ and define in the corresponding local
trivialization $E_{C,r}$
$$
w_{r,a,b}: D_r \to E_{C,r}, \quad z \mapsto 
r^{-1/2} az + r^{1/2}
\ol{a} .
$$
The condition for $w_{r,a,b}$ to be a section is by \eqref{piV} 
\begin{eqnarray*}
 z &=& w_{r,a,b}(z) \cdot w_{r,a,b}(z) \\
&=& (r^{-1/2} az + r^{1/2} \ol{a})
 \cdot (r^{-1/2} az + r^{1/2} \ol{a}) \\
&=& r^{-1} a \cdot a z^2 + 2 a
 \cdot \ol{a} z + r \ol{a} \cdot \ol{a} .\end{eqnarray*}
Hence 
\begin{equation} \label{relations}  2 a \cdot \ol{a} =  1, \quad a \cdot a = \ol{a} \cdot \ol{a} =
0 .\end{equation}
Given this, the section $w_{r,a,b}$ takes values in $F_{C,r}$.  We
claim that the Maslov index satisfies the following formula:
\begin{equation} \label{wind} I(w_{r,a,b}) = c-1 .\end{equation} 
Indeed, by definition the Maslov index of $w := w_{r,a,b}$ is the
index of the pair $(w^* T^{\on{vert}} E_C, (\partial w)^*
T^{\on{vert}} F_C)$.  This pair fits into the exact sequence
$$ 0 \to (w^* T^{\on{vert}} E_{C,r}, (\partial w)^* T^{\on{vert}}
F_{C,r}) \to (w^* TE_{C,r}, (\partial w)^* TF_{C,r}) \to
(TD,T(\partial D)) \to 0 .
$$
Now $ T_{w(z)} E_{C,r} = (T_b B_{C,r})^2 \oplus \C^{c+1}$ is trivial,
and the boundary condition has vertical part
\begin{equation} \label{vertpart}
 T_{w(z)}^{\on{vert}} F_{C,r} \cong \Delta_{T_b B} \oplus
 T_{w(z)}(\sqrt{z} S^c) \cong \Delta_{T_b B} \oplus \R^{c+1}/\R
 w(z) \end{equation}
where $\Delta_{T_b B}$ is the diagonal.  The horizontal part of the
boundary value problem maps isomorphically onto $(TD, T(\partial D))$
via $D\pi$.  Using \eqref{vertpart} we have
$$ I(w^* TE_{C,r}, (\partial w)^* TF_{C,r}) = I(\C^{c+1},\sqrt{z}
\R^{c+1}) .$$
Hence
\begin{eqnarray*} I(w^* T^{\on{vert}} E, (\partial w)^* T^{\on{vert}} F) &=& 
I(w^* TE, (\partial w)^* TF)- (TD,T(\partial D)) \\
&=& (c+1) - 2 = c-1 \end{eqnarray*}
which proves the claim \eqref{wind}.

The given sections are the sections of lowest index for sufficiently
small radius.  Indeed, the area $A(w_{r,a,b})$ approaches zero as $r
\to 0$, for all $a$; this fact holds even after the adjustment
\cite[(1.17)]{se:lo} since the adjustment is by the pullback of the
differential of a bounded one-form.  Choose $r$ sufficiently small so
that $A(w_{r,a,b}) \leq 1/2\lambda$.  The area-index monotonicity
relation and non-negativity of the curvature in Proposition \ref{EC}
imply that any other section $u$ has positive area.  So the index
$I(u)$ is at least the index $I(w_{r,a,b})$ of $w_{r,a,b}$.
\end{proof}

In order to define relative invariants for Lefschetz-Bott fibrations
with codimension one vanishing cycles, a stronger monotonicity
assumption must be assumed. 

\begin{definition} \label{stronglymonotone}  {\rm (Strong monotonicity)}  
In the case that $C \subset M$ is a fibered coisotropic of codimension
one we denote by $M_C$ the {\em cut space} as in Lerman \cite{le:sy2}.
The cut space $M_C$ is the space obtained by cutting $M$ along $C$ and
collapsing the resulting manifold (whose boundary is two copies of
$C$) by the circle action on the boundary.  Thus $M_C$ contains two
copies $B_\pm$ of $B$.  Denote by $[B_\pm] \in H^2(M_C)$ the dual
classes of $B_\pm$ and by $[\omega_C] \in H^2(M_C)$ the symplectic
class.  A monotone fibered coisotropic $C$ is {\em strongly monotone}
if either $\codim(C) \ge 2 $ or both $\codim(C) = 1$ and
$c_1(M_C)-[B_+]-[B_-]$ is a positive multiple of $[\omega_C]$.
\end{definition} 

\begin{proposition}  Suppose that the vanishing cycles of $E$
are strongly monotone.  Then the sections of $(E,F)$ satisfy a
monotonicity relation, and counting pseudoholomorphic sections defines
a relative invariant $\Phi(E,F;\Lambda)$ as in \eqref{HPhi}.
\end{proposition} 

\begin{proof}   The codimension two case is already covered
in Proposition \ref{monotonebound}.  In the case that the fibered
coisotropic is codimension one, there are degree two homology classes
in the special fibers that are not equivalent to homology classes in
the special fiber.  Let $\nu_C: M_C\to E_{s_0} \subset E$ denote the
normalization map, mapping the two copies $B_\pm$ of $B$ created by
cutting to the singular locus $B \subset E_{s_0}$.  Then
$$ \nu_C^* c_1(E) = c_1(M_C) - [B_+] - [B_-] .$$
Indeed, let $z$ be a local coordinate on the base near a critical
value.  The one-form $\pi^* \d z$ is non-vanishing everyone except at
the singular locus in the fiber where $\pi^* \d z$ has a simple zero.
It follows that if the vanishing cycles are strongly monotone, then
the monotonicity relation holds on fiber classes.  It follows that
monotonicity also holds on sections.  Counting pseudoholomorphic
sections defines an invariant $\Phi(E,F)$, as in the case of higher
codimension.
\end{proof}

\begin{remark} \label{codimone}  {\rm (Double cover in 
the codimension one case)} In the case that $c = 1$, $r = 1$, the
  evaluation map on the moduli space $\M(E_{C,r},F_{C,r})_0$ of
  zero-index pseudoholomorphic sections of $(E_{C,r},F_{C,r})$ at $z =
  1$ induces a double cover
\begin{equation}  \label{dc}  \ev_1: \M(E_{c,r},F_{C,r})_0 \to C, \ w_{r,a,b} \mapsto
  w_{r,a,b}(1) = a + \ol{a} \end{equation}
of the fiber of the vanishing cycle $C$.  The relations
\eqref{relations} on $a = (a_1,a_2)$ become
\begin{equation} \label{become}
 a_2 = \pm i a_1 , \quad a_1 \ol{a_1} = 1/2 .\end{equation}
\end{remark} 

\begin{corollary} \label{pinch} {\rm (Vanishing of the relative
    invariant associated to a standard fibration)} Suppose $E$ is the
  Lefschetz fibration with the Lagrangian boundary condition $F$
  obtained by a seam connect sum from a Lefschetz-Bott fibration
  $(E_{C,r},F_{C,r})$ over the disk $D_r$ corresponding to a
  spherically fibered coisotropic submanifold $C$, with an arbitrary
  quilted Lefschetz-Bott fibration $E_0 \to S_0$ with boundary
  condition $F_0$.  Suppose that all these fibrations with boundary
  conditions are monotone and equipped with relative spin structures,
  so that in particular the relative invariant $\Phi(E,F)$ is defined.
  Suppose furthermore that in the codimension one case $\codim(C) = 1$
  the coisotropic is equipped with a local system with holonomy $-1$
  around the fibers.  Then $\Phi(E,F)= 0 $.  More precisely, there
  exists a null-homotopy of the chain-level operator $C \Phi(E,F)$
  that, if $r$ is sufficiently small, has positive $q$-exponents.
\end{corollary} 

\begin{proof}   Suppose that $E,F$ are as in the statement of 
the Corollary.  Consider the family of surfaces obtained by stretching
the neck so that a standard fibration $(E_{C,r},F_{C,r})$ bubbles off.
The chain level invariants $C\Phi(E_{C,r}, F_{C,r})$ are
chain-homotopic by a chain-homotopy corresponding to isolated points
in the parametrized moduli space for the deformation.  
Since the curvature is non-negative and zero-area sections are of even
index, the $q$-exponents in the homotopy operator are positive.  The
gluing theorem \ref{gluing} gives a bijection between
pseudoholomorphic sections of $E_\rho, F_\rho$, for small $\rho$, and
$(u_0,u_1)$ in the zero-dimensional component
$(\ev_0,\ev_1)^{-1}(\Delta)$, where $\Delta$ is the diagonal in the
gluing fiber $F_{0,z_0} \cong F_{C,z_1}$.  Equality of indices gives
$$ \dim T_{u_0} \M(E_0,F_0)
 + \dim T_{u_1} \M(E_C,F_C) = \dim(C)
$$
for $k \in \{ 0,1 \}$.  In the codimension two case Corollary
\ref{pinch} now follows from Lemma \ref{above} which implies that the
moduli spaces for small $r$ and large $\rho$ are empty.  The
chain-level invariants are independent of the choice of $r$ up to
chain-homotopy.

The codimension one case depends on the description of the double
cover in \ref{codimone}.  Indeed, in this case the limiting moduli
space is not empty, but has two components corresponding to the
choices of sign in \ref{codimone}.  Each component is isomorphic to
the Lagrangian.  Furthermore, each has the induced orientation given
by deforming the boundary value problem to a trivial one and using the
orientation from the Lagrangian as in Remark \ref{sketch}.  The two
sections $u_\pm$ corresponding to a given $(a_1,a_2)$ in
\eqref{become} are related by the involution
$$\iota: \C^2 \to \C^2, \quad (z_1,z_2) \mapsto (z_1,-z_2) .$$
It follows that any deformation of $u_+^* F$ to a trivial boundary
condition induces a deformation of $u_-^* F$ and vice versa, giving 
a commutative diagram
$$ \begin{diagram} \node{ \det(D_{u_+^* E, u_+^* F})} \arrow{e}
  \arrow{s} \node{ \det( (u_+^* F)_1)} \arrow{s} \\ \node{
    \det(D_{u_-^* E, u_-^* F})} \arrow{e} \node{ \det( (u_-^* F)_1).}
  \end{diagram} $$
The evaluation maps for $u_\pm$
$$ \ker D_{u_+^* E, u_+^* F}
 \cong u_+^* F_1
, \quad \ker D_{u_-^* E, u_-^* F} \cong u_-^* F_1 $$
given by linearizing \eqref{relations} are by definition orientation
preserving.  Since the right-hand-arrow is orientation-reversing
(exactly one factor of $F$ is reversed) the left-hand-arrow is also
orientation-reversing.  Hence the involution induces an orientation
reversing involution of the moduli space $\M(E_{C,r}, F_{C,r})$ of
disks with boundary on the Lagrangian $F_{C,r}$.  Since the involution
also reverses the orientation on the Lagrangian, the induced
involution on the limiting moduli space $\M(E_\rho,F_\rho)$ is
orientation preserving.  These contributions cancel in the weighted
disk count \eqref{CPhi}.
\end{proof}  

\begin{remark} \label{indepchoice} {\rm (Independence of choices)}  
Let $S_0,E_0,F_0,E_C,F_C,D_r$ be as above.  Define an invariant
associated to the nodal surface $S_r$ by identifying a point on the
seam $S_0$ and disk $D$.  Define a nodal Lefschetz fibration $E_r$ by
identifying a fiber of $E_0$ with a fiber of $E_C$ with boundary
condition $F_r \subset E_r$.  By counting pseudoholomorphic sections
with matching condition at the node one obtains an invariant
$\Phi(E_r,F_r)$ for the nodal fibration, equal to the invariant for
the glued fibration $\Phi(E,F)$.  In particular, vanishing of
$\Phi(E_r,F_r)$ for $r$ small implies vanishing for any $r$.
%C
\end{remark} 

\subsection{Horizontal invariants}

As in Seidel \cite{se:lo}, the computation of the relative invariants
in the special cases needed for the exact triangle uses only
horizontal sections, defined as follows.

\begin{definition}   Let $\ul{S}$ be a quilted surface with
strip-like ends, $\ul{E}$ a quilted Lefschetz-Bott fibration, and
$\ul{F}$ a collection of Lagrangian seam/boundary conditions.
\begin{enumerate} 
\item {\rm (Horizontal sections)} A section ${u}: \ul{S} \to \ul{E}$
  (that is, a collection of sections $u_k: S_k \to E_k,\  k = 1,\ldots,
  m$) is {\em horizontal} if
$$\Im Du_k(s) = TE^h_{k,u_k(s)}$$
for all $s \in S_k, k = 1,\ldots, m$.  Let $ \M^h(\ul{E},\ul{F}) $
denote the space of horizontal sections.
\item {\rm (Horizontal almost complex structures)} A collection of
  compatible almost complex structures $\ul{J} \in \J(\ul{E})$ is {\em
    horizontal} if
$$J_k(e) T_e^hE_k = T^h_e E_k, \quad \forall e \in E_k
  \ssm E_k^\crit, k = 1,\ldots, m $$  
or equivalently if $\omega_{E_k}(J \cdot, \cdot)$ is symmetric for $k
= 1,\ldots, m$.  Let $\J^h(\ul{E})$ denote the set of horizontal
compatible almost complex structures.  If $J \in \J^h(\ul{E})$, then
any horizontal section is $J$-holomorphic, that is, 
$$ (J \in \J^h(\ul{E})) \implies (\M^h(\ul{E},\ul{F}) \subset
\M(\ul{E},\ul{F})).$$
\item {\rm (Clean moduli spaces)} $\M^h(\ul{E},\ul{F})$ is {\em clean}
  if $\M^h(\ul{E},\ul{F})$ is a smooth submanifold of the Banach
  manifold of sections of a sufficiently differentiable Sobolev class
  of $\ul{E}$ with tangent space
$$ T_{u} \M^h(\ul{E},\ul{F}) = \{ \xi \in \Omega^0(u^*
  T^{\on{vert}}\ul{E}, ( u | \partial \ul{S})^* T^{\on{vert}} \ul{F})
  \ | \ \nabla_u \xi = 0 \} $$
the set of horizontal sections of $u^* T^{\on{vert}}\ul{E}$ with
boundary/seams in $( u | \partial \ul{S})^* T^{\on{vert}} \ul{F}$.
\end{enumerate} 
\end{definition}

\begin{remark} {\rm (Horizontal and vertical energy)}  
The energy of a section of a Lefschetz-Bott fibration may be broken up
into horizontal and vertical parts.  We suppose that we choose
$\omega_S$ so that $\omega_E + \pi^* \omega_S$ is symplectic.  This
two-form together with the almost complex structure defines a metric
on the total space $E$.  Fix a metric on $S$, giving rise to an area
form $\d \Vol_S$, and define
$$ {\mathcal{E}}(u) = \int_S \Vert \d u \Vert^2 \d \Vol_S .$$
The energy $\mathcal{E}(u)$ splits into vertical and horizontal
contributions
$$ {\mathcal{E}}(u) = {\mathcal{E}}^v(u) + {\mathcal{E}}^h(u), \quad {\mathcal{E}}^{v,h}(u) = \int_S \Vert \d^{v,h} u
\Vert^2 \d \Vol_S 
.$$
As in Seidel \cite[Equation (2.10)]{se:lo} one has an energy-area
relation
$$ {\mathcal{E}}(u) = \int_S u^* (\omega_E + \pi^* \omega_S) + (1/2) \int_S \Vert
\olp_J u \Vert^2 \d \Vol_S .$$
This implies that for pseudoholomorphic sections one has an
energy-area relation
\begin{equation} \label{minim}
 {\mathcal{E}}^v(u)
 + \int_S f(u) \omega_S = \int_S u^* \omega_E .\end{equation}
A similar relation holds in the quilted case.
\end{remark} 

\begin{proposition} \label{horizprop} {\rm (Criterion for horizontal moduli spaces
to be a component)} Suppose that $(\ul{E},\ul{F})$ is a quilted
  symplectic Lefschetz-Bott fibration with Lagrangian boundary
  condition.  If $\ul{E}$ is non-negatively curved and equipped with a
  horizontal almost complex structure such that $\M^h(\ul{E},\ul{F})$
  is clean, of dimension $\Ind(D_u)$, and consists of sections with
  vanishing $\omega_E$-integral,
$$ \forall u \in \M^h(\ul{E},\ul{F}), \quad \int_S u^* \omega_E = 0
  $$
then $\M^h(\ul{E},\ul{F})$ is a path component of $\M(\ul{E},\ul{F})$
and consists of regular sections.
\end{proposition}

\begin{proof} 
As in \cite[Lemma 2.11]{se:lo}, suppose that $u$ is a horizontal
section with $\int_S u^* \omega_E =0 $.  Any other pseudoholomorphic
section $\ti{u}$ homotopic to $u$ satisfies by \eqref{minim} the
equation
$${\mathcal{E}}^v(\ti{u})
 + \int_S f(\ti{u})
\omega_S = 0 .$$  
By the non-negative curvature assumption $\ti{u}$ must satisfy 
$$ {\mathcal{E}}^v(\ti{u}) = 0, \quad \on{im}(\d \ti{u}) \subset T^h
E $$
as claimed.  The regularity of $\M^h(\ul{E},\ul{F})$ follows as in
Seidel \cite[Lemma 2.12]{se:lo}: Taking the second derivative of the 
index-area relation \eqref{minim} gives 
$$ \Vert \nabla_u \xi \Vert^2 + \int_S \on{Hess}(f)(u) \omega_S = \Vert 
D_u \xi \Vert^2 $$
where $on{Hess}$ is the Hessian of $f$.  Non-negativity of the
curvature implies that the Hessian is positive, and so $ \ker(D_u)
\subset \ker(\nabla_u)$.  Then $\dim \M^h(\ul{E},\ul{F}) = \Ind(D_u)$
implies $\coker(D_u) = 0$.
\end{proof} 

We give a criterion for the zero-dimensional component
$\M(\ul{E},\ul{F})_0$ to consist entirely of horizontal sections in
the monotone case; the exact case was discussed in Seidel
\cite{se:lo}.  Let $\ul{E} \to \ul{S}$ be a monotone Lefschetz-Bott
fibration with Lagrangian boundary/seam conditions $\ul{F}$.  For each
collection of intersection points $( (x_e)_{e \in \E})$ define
$c((x_e)_{e \in \E}) $ by
$$ \lambda I(u) = 2 ( A(u) - c((x_e)_{e \in \E})) $$
for any collection of sections $u$ with limits $ ( x_e )_{ e \in \mE
}$.

\begin{proposition}  \label{filter} {\rm (Criterion for the existence of only horizontal sections)} 
Suppose that $\ul{E}$ is a quilted Lefschetz-Bott fibration with
non-negative curvature and is equipped with a horizontal almost
complex structure.  If $c((x_e)_{e \in \mE}) = 0 $ then
$\M(\ul{E},\ul{F}; (x_e)_{e \in \mE} )_0$ consists of horizontal
sections.
\end{proposition}

\begin{proof}   If $u$ is a pseudoholomorphic section with 
index $0$, then the monotonicity relation \eqref{IA2} implies that $u$
has non-negative symplectic area equal to $c((x_e)_{e \in \E}) $.  If
this constant vanishes, then all such sections must have vanishing
symplectic area.  By the non-negativity of the curvature and
\eqref{minim}, any such section has vanishing vertical energy and so
is horizontal.
\end{proof} 

\section{Floer versions of the exact triangle}
\label{floerversion} 

The proofs of the exact triangles described in the introduction follow
the lines of the proof of Floer's exact triangle \cite{br:fl}, in
Seidel \cite{se:lo} and Perutz \cite{per:gys}. Namely, one first
constructs a short sequence of cochain groups that is exact {\em up to
  leading order}, and then uses a spectral sequence argument to deduce
the existence of a long exact sequence of cohomology groups.  In this
section we also describe various extensions, to the case of minimal
Maslov number two and the case of periodic Floer cohomology.

\subsection{Fibered Picard-Lefschetz formula}

In this section we prove the exact triangle on the level of vector
spaces; this is essentially equivalent to the fibered Picard-Lefschetz
formula in Theorem \ref{fpl}. 

\begin{definition} {\rm (Angle functions)}  Recall from Section \ref{twists} that a
Dehn twist in a local model $T^\dual S^c$ is defined by rotating a
vector in $T^\dual S^c$ with norm $t$ by an {\em angle function}
$$\theta \in C^\infty([0,\infty)), \quad \theta(0) = \pi, \quad
  \theta(t) = 0, t \gg 0 .$$
The angle function $\theta(t)$ is related to the choice of $\zeta(t)$
of \eqref{rt} by $\theta(t) = \zeta'(t)$.  For the rest of the paper
we assume that the function $\theta(t)$ is decreasing with $t$.  We
wish to use angle functions 
that go to $0$ sufficiently quickly.  In particular, given an angle
function $\theta$ we consider the family of angle functions defined by
\begin{equation} \label{theta} 
\theta^{\delta}(t) := \theta({\delta} t) .\end{equation}
\end{definition}

\begin{lemma} \label{bijection} {\rm (Intersection points for fibered Dehn twists)}  
Let $C \subset M$ be a compact spherically fibered coisotropic in a
symplectic manifold $M$ over a symplectic manifold $B$, and $L^0,L^1
\subset M$ Lagrangian submanifolds.  Suppose that the intersection
$$(L^0 \times C) \pitchfork (C^t \times L^1) \subset M \times B \times
M $$
is transverse while the intersection
$$ L^0 \cap L^1 \cap C = \emptyset $$
is empty.  For ${\delta} > 0 $ let $\tau_{C,{\delta}} \in
\Diff(M,\omega)$ be a fibered Dehn twist using a given local model and
the angle function $\theta^{\delta}(t)$.  There exists a constant
${\delta}_0$ such that if ${\delta} > {\delta}_0$ then there exists a
bijection
\begin{equation} \label{bij}
(i_1 \cup i_2): (L^0 \cap L^1 ) \cup \left( 
(L^0 \times C) \cap (C^t
  \times L^1) \right) \to L^0 \cap \tau_{C,{\delta}}^{-1} L^1
  .\end{equation}
\end{lemma} 

\begin{proof}  The intersection points
of the untwisted Lagrangians naturally include in the intersection
points of the twisted Lagrangians for ${\delta}$ sufficiently large:
The transversality assumptions imply that the intersection $L^0 \cap
L^1 \cap C$ is empty.  By taking ${\delta}_0$ sufficiently large we may
suppose that $\tau_{C,{\delta}}$ is supported in a neighborhood $U_C$
disjoint from $L^0 \cap L^1$.  Thus there is an inclusion
$$ i_1: L^0 \cap L^1 \to L^0 \cap \tau_{C,{\delta}}^{-1} L^1 .$$
We wish to identify the remaining intersection points with $(L^0
\times C) \cap (C^t \times L^1) $.

We first recall the unfibered case in Seidel \cite[Lemma 3.2]{se:lo}.
Suppose that the base $B$ is a point.  Choose a local model for the
Lagrangian $C$ given by a symplectomorphism of a neighborhood of the
zero section in $T^\dual S^c$ with a neighborhood of $C$ in $M$.  The
local model $\phi: T^\dual S^c \supset U \to M$ may be chosen so that
the Lagrangians $L^0,L^1$ are fibers of $\pi: T^\dual S^c \to S^c$:
$$ L^k \cap \phi(U) = \phi( U \cap \pi^{-1}(v_k)), \quad v_k \in S^c.
$$
What this amounts to is choosing the initial diffeomorphism with this
property and then choosing the Moser isotopy to preserve the
Lagrangians.  The Dehn twist $\tau_C^{-1}$ acts at $[v] \in P(T^\dual
S^c)$ by normalized geodesic flow by time $\theta^{\delta}(|v|)$.
There exists a unique $v \in L^1$ of norm less than $\pi$ such that
its time $\pi - |v|$ normalized geodesic flow lies in $L^0$.  The
unique point $w \in \R_{> 0} v$ with $ \theta^{\delta}(|w|) = \pi -
|v|$ gives the desired intersection point.

We reduce to the case of trivial base by the use of suitable local
models as follows.  Let
$$l_0 \in L^0 \cap C, l_1 \in L^1 \cap C, \quad \pi(l_0) = \pi(L_1) =: b$$
be points with the same projection $b \in B$.  Consider the action of
scalar multiplication on the vector bundle $P(\pi_{T^\dual S^c}):
P(T^\dual S^c) \to P(S^c)$.  This action induces an action of
sufficiently small scalars on $M$ via the coisotropic embedding.  By
\eqref{theta} $\tau_{C,{\delta}} = {\delta} \tau_C {\delta}^{-1} $
wherever the right-hand-side is defined.  For $k \in \{ 0,1 \}$, the
submanifolds $({\delta} L^k \cap U_C)_{{\delta} > 0 }$ fit together with
$P(\pi_{T^\dual S^c})^{-1} (L^k \cap C)$ in the limit ${\delta} \to
\infty$ to a smooth family at ${\delta} = \infty$.  Indeed, if $L^k$ is
given locally by $\{ f(z,x,y) = 0 \}$ in local coordinates $z$ on $B$
and cotangent coordinates $(x,y)$ on $T^\dual S^c$ then
\begin{equation} \label{cutsout}
{\delta} L^k = \{ f(z,x,{\delta}^{-1} y) = 0 \} .\end{equation}
By assumption $L^k$ is transversal to the zero section in $T^\dual
S^c$.  This implies that the equation \eqref{cutsout} cuts out a
smooth family at ${\delta} = \infty$.  By the case $B$ trivial, the
intersection at ${\delta} = \infty$ is given by
$$ P(\pi_{T^\dual S^c})^{-1} (L^0 \cap C) \cap \tau_{C}^{-1}
P(\pi_{T^\dual S^c})^{-1} (L^1 \cap C) $$
and is transverse.  By the implicit function theorem, the set of
intersection points $ L^0 \cap \tau_{C,{\delta}}^{-1} L^1$ forms a
smooth manifold parametrized by ${\delta} \gg 0 $ and intersection
points $(L^0 \times C) \cap (C^t \times L^1)$ as desired.

Alternatively, one can argue as follows via the limit of the twisted
Lagrangian.  After restricting to connected component in an open
neighborhood of $p^{-1}(b)$, the projection $L^1 \cap C \to B$ is an
embedding.  For any point $l \in L^1$ with $\tau_C(l) \neq l$, the
images of the derivatives of $\tau_{C,{\delta}}^{-1}$ converge to the
tangent space of $p^{-1}(C \cap L^1)$: If the projection of
$\tau_C^{-1}(l)$ to $C$ is $c$ in the local model then
$$ \on{Im}(D \tau_{C,{\delta}}^{-1}({\delta}^{-1} l)) =
\on{Im}({\delta}^{-1} D \tau_C^{-1}(l) {\delta}) 
\to D_c p^{-1}(T_{p(l)}
(C \cap L^1)), \quad {\delta} \to \infty $$
since ${\delta}^{-1} D \tau_C^{-1}(l) {\delta}$ converges to the
composition of $D \tau_C^{-1}(l)$ with projection to the tangent space
to $C$.  It follows that the sequence of submanifolds
$\tau_{C,{\delta}}^{-1}(L^1) - L^1$ converges, as an embedded
submanifold, to $p^{-1}(p(L^1 \cap C)) - L^1$.  Again by the implicit
function theorem, the intersections $L^0 \cap
\tau_{C,{\delta}}^{-1}(L^1)$ correspond to intersection points of $L^0
\cap p^{-1}(p(L^1 \cap C)) \cong (L^0 \times C) \cap (C^t \times
L^1)$.
\end{proof} 

\subsection{Lagrangian Floer version}
\label{exact}

We now prove the exact triangle Theorem \ref{main}.  Since $\tau_C$ is
a symplectomorphism, it suffices to prove the theorem with $L^1$
replaced by $\tau_C^{-1} L^1$.  That is, we construct a long exact
sequence
$$ \ldots HF(L^0,C^t,C,\tau_C^{-1}L^1) \to HF(L^0,\tau_C^{-1}L^1) \to
HF(L^0,L^1) \ldots .$$
More generally in the quilted case, it suffices to prove the theorem
with $\ul{L}^1$ replaced by $\ul{L}^1 \# \on{graph}(\tau_C)$.

\subsubsection{Definition of the maps}

Let $M$ be a monotone symplectic background and $C \subset M$ a
spherically fibered coisotropic submanifold of codimension $c \ge 2$.
Let $L^0,L^1 \subset M$ be admissible Lagrangian branes, and suppose
that $C$ is equipped with an admissible brane structure as a
Lagrangian submanifold of $M^- \times B$.  These conditions imply that
all Lefschetz-Bott fibrations discussed below are monotone as in Lemma
\ref{monotoneends}.  We may assume, after Hamiltonian perturbation,
that $C,L^0,L^1$ all intersect transversally.

\begin{definition} 
\begin{enumerate} 
\item {\rm (Chaps map)} The first map in the exact sequence is defined
  as the relative invariant associated to a ``quilted pair of pants'',
  or more accurately, ``quilted chaps'' in American dialect.  Let
  $\ul{S}_1$ denote the quilted surface shown in Figure \ref{pants}:
\begin{figure}[ht]
\begin{picture}(0,0)%
\includegraphics{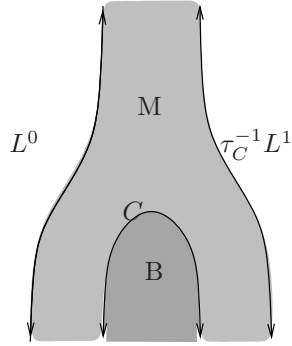}%
\end{picture}%
\setlength{\unitlength}{4144sp}%
\begingroup\makeatletter\ifx\SetFigFont\undefined%
\gdef\SetFigFont#1#2#3#4#5{%
  \reset@font\fontsize{#1}{#2pt}%
  \fontfamily{#3}\fontseries{#4}\fontshape{#5}%
  \selectfont}%
\fi\endgroup%
\begin{picture}(1613,2067)(2326,-2113)
\put(3601,-961){\makebox(0,0)[lb]{\smash{{\SetFigFont{10}{12.0}{\rmdefault}{\mddefault}{\updefault}{$\tau_C^{-1} L^1$}%
}}}}
\put(3016,-1366){\makebox(0,0)[lb]{\smash{{\SetFigFont{10}{12.0}{\rmdefault}{\mddefault}{\updefault}{$C$}%
}}}}
\put(3106,-736){\makebox(0,0)[lb]{\smash{{\SetFigFont{10}{12.0}{\rmdefault}{\mddefault}{\updefault}{M}%
}}}}
\put(3151,-1726){\makebox(0,0)[lb]{\smash{{\SetFigFont{10}{12.0}{\rmdefault}{\mddefault}{\updefault}{B}%
}}}}
\put(2341,-961){\makebox(0,0)[lb]{\smash{{\SetFigFont{10}{12.0}{\rmdefault}{\mddefault}{\updefault}{$L^0$}%
}}}}
\end{picture}%

\caption{Quilted surface $\ul{S}_1$ defining $HF(L^0,C^t,C,\tau_C^{-1}
  L^1;\Lambda)\to HF(L^0,\tau_C^{-1} L^1;\Lambda)$}
\label{pants}
\end{figure}
Let $(S_B,S_M)$ denote the patches of $\ul{S}_1$, and $\ul{E} = (S_M
\times M,S_B \times B)$, where $B$ is the base of the fibration $p: C
\to B$.  We identify $C$ with its image in $M \times B$.  Let $\ul{F}$
denote the Lagrangian seam/boundary condition for $\ul{E}$ given by
$L^0,L^1,C$ and consider the relative invariant
\begin{equation} \label{phi1} \Phi_1: HF(L^0 ,C, C^t, \tau_C^{-1} L^1;\Lambda)[\dim(B)] \to
HF(L^0, \tau_C^{-1} L^1;\Lambda) .\end{equation} 
This invariant was defined in Definition \ref{relinv} by counting
points in the zero-dimensional component of the moduli space $\M_1$ of
pseudoholomorphic quilts on $\ul{S}_1$.
\item {\rm (Lefschetz-Bott map)} The second map in the exact sequence
  is a relative invariant associated to a Lefschetz-Bott fibration
  with monodromy given by the Dehn twist.  Namely let $E_{C,r} \to
  D_r$ denote a standard Lefschetz-Bott fibration with monodromy
  $\tau_C$ from Lemma \ref{ECexist} and Definition
  \ref{smallstandard}.  Gluing in $E_{C,r}$ with the trivial fibration
  over a strip (using the identity as transition map to the left of
  the disk, and $\tau_C$ as transition map to the right) as in Seidel
  \cite[p. 7]{se:lo} gives a Lefschetz-Bott fibration $({E}_2,{F}_2)$
  over the infinite strip shown in Figure \ref{tauC}.
\begin{figure}[ht]
\begin{picture}(0,0)%
\includegraphics{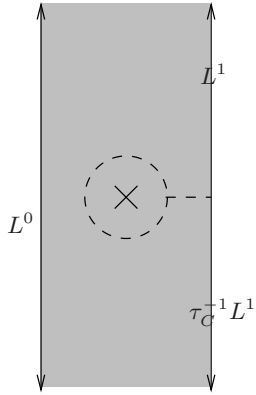}%
\end{picture}%
\setlength{\unitlength}{4144sp}%
\begingroup\makeatletter\ifx\SetFigFont\undefined%
\gdef\SetFigFont#1#2#3#4#5{%
  \reset@font\fontsize{#1}{#2pt}%
  \fontfamily{#3}\fontseries{#4}\fontshape{#5}%
  \selectfont}%
\fi\endgroup%
\begin{picture}(1271,2398)(840,-2672)
\put(855,-1686){\makebox(0,0)[lb]{\smash{{\SetFigFont{9}{10.8}{\rmdefault}{\mddefault}{\updefault}{$L^0$}%
}}}}
\put(2010,-796){\makebox(0,0)[lb]{\smash{{\SetFigFont{9}{10.8}{\rmdefault}{\mddefault}{\updefault}{$ L^1$}%
}}}}
\put(1945,-2203){\makebox(0,0)[lb]{\smash{{\SetFigFont{9}{10.8}{\rmdefault}{\mddefault}{\updefault}{$\tau_C^{-1}
L^1$}%
}}}}
\end{picture}%
\caption{Lefschetz-Bott fibration $E_2 \to S_2$ defining the map
  $HF(L^0,\tau_C^{-1} L^1;\Lambda) \to HF(L^0, L^1;\Lambda)$}
\label{tauC}
\end{figure}
Let 
\begin{equation} \label{phi2}
 \Phi_2: \ HF(L^0, \tau_C^{-1} L^1;\Lambda) \to
 HF(L^0,L^1;\Lambda) \end{equation} 
denote the associated relative invariant.  Relative invariants were
defined in \eqref{HPhi} by counting points in the zero-dimensional
component of the moduli space $\M_2$ of pseudoholomorphic sections of
$E_2 \to S_2$ with boundary in $F_2$.  (It follows from Theorem
\ref{zeroindex} below that $\Phi_2$ has degree zero.)
\end{enumerate} 
\end{definition} 

The first step in the proof of the exact sequence is to show that the
composition of the chaps and Lefschetz-Bott maps vanishes:

\begin{lemma} \label{vanishes} {\rm (Exactness at the middle term)}  
The composition $\Phi = \Phi_2 \circ \Phi_1$ (the relative invariant
associated to picture on the left in Figure \ref{deform}) vanishes;
more precisely there exists a null-homotopy of the chain level maps $C
\Phi_2 \circ C \Phi_1$ whose terms have positive $q$-exponent for $r$
sufficiently small.
\end{lemma}

\begin{proof}   The composition of the two relative invariants
is the relative invariant associated to a Lefschetz-Bott fibration
over the glued surface $\ul{S} = \ul{S}_1 {\sharp} \ul{S}_2$ by
\eqref{gluealongends}.  Consider the deformation $\ul{S}_t$ of
$\ul{S}$ obtained by moving the critical value of the Lefschetz-Bott
fibration towards the boundary marked $C$ and pinching off a disk in
$M \times B$ with boundary values in $C$.  This process is shown in
the right-most two pictures in Figure \ref{deform}.  The bundles
$\ul{E}$ and Lagrangian boundary/seam conditions $\ul{F}$ naturally
extend to families $\ul{E}_t,\ul{F}_t$ that are obtained from gluing
for $t \gg 0$.  It follows from Corollary \ref{pinch} that the
relative invariant $C\Phi$ is null-homotopic and, for $r$ sufficiently
small, has positive $q$-exponents.
\end{proof}
\begin{figure}[h]

\begin{picture}(0,0)%
\includegraphics{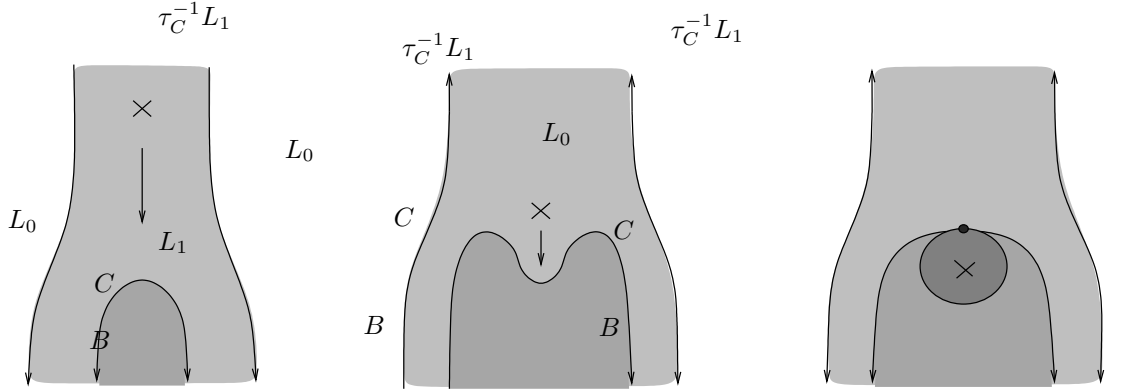}%
\end{picture}%
\setlength{\unitlength}{4144sp}%
\begingroup\makeatletter\ifx\SetFigFont\undefined%
\gdef\SetFigFont#1#2#3#4#5{%
  \reset@font\fontsize{#1}{#2pt}%
  \fontfamily{#3}\fontseries{#4}\fontshape{#5}%
  \selectfont}%
\fi\endgroup%
\begin{picture}(4559,2413)(628,-2687)
\put(1577,-489){\makebox(0,0)[lb]{\smash{{\SetFigFont{10}{7.2}{\rmdefault}{\mddefault}{\updefault}{\color[rgb]{0,0,0}$\tau_C^{-1} L_1$}%
}}}}
\put(676,-1706){\makebox(0,0)[lb]{\smash{{\SetFigFont{10}{7.2}{\rmdefault}{\mddefault}{\updefault}{\color[rgb]{0,0,0}$L_0$}%
}}}}
\put(1577,-1834){\makebox(0,0)[lb]{\smash{{\SetFigFont{10}{7.2}{\rmdefault}{\mddefault}{\updefault}{\color[rgb]{0,0,0}$L_1$}%
}}}}
\put(1164,-2427){\makebox(0,0)[lb]{\smash{{\SetFigFont{10}{7.2}{\rmdefault}{\mddefault}{\updefault}{\color[rgb]{0,0,0}$B$}%
}}}}
\put(1194,-2089){\makebox(0,0)[lb]{\smash{{\SetFigFont{10}{7.2}{\rmdefault}{\mddefault}{\updefault}{\color[rgb]{0,0,0}$C$}%
}}}}
\put(2804,-2340){\makebox(0,0)[lb]{\smash{{\SetFigFont{10}{7.2}{\rmdefault}{\mddefault}{\updefault}{\color[rgb]{0,0,0}$B$}%
}}}}
\put(2987,-1692){\makebox(0,0)[lb]{\smash{{\SetFigFont{10}{7.2}{\rmdefault}{\mddefault}{\updefault}{\color[rgb]{0,0,0}$C$}%
}}}}
\put(3036,-676){\makebox(0,0)[lb]{\smash{{\SetFigFont{10}{7.2}{\rmdefault}{\mddefault}{\updefault}{\color[rgb]{0,0,0}$\tau_C^{-1}L_1$}%
}}}}
\put(2332,-1285){\makebox(0,0)[lb]{\smash{{\SetFigFont{10}{7.2}{\rmdefault}{\mddefault}{\updefault}{\color[rgb]{0,0,0}$L_0$}%
}}}}
\put(4644,-566){\makebox(0,0)[lb]{\smash{{\SetFigFont{10}{7.2}{\rmdefault}{\mddefault}{\updefault}{\color[rgb]{0,0,0}$\tau_C^{-1}L_1$}%
}}}}
\put(3869,-1185){\makebox(0,0)[lb]{\smash{{\SetFigFont{10}{7.2}{\rmdefault}{\mddefault}{\updefault}{\color[rgb]{0,0,0}$L_0$}%
}}}}
\put(4298,-1770){\makebox(0,0)[lb]{\smash{{\SetFigFont{10}{7.2}{\rmdefault}{\mddefault}{\updefault}{\color[rgb]{0,0,0}$C$}%
}}}}
\put(4210,-2355){\makebox(0,0)[lb]{\smash{{\SetFigFont{10}{7.2}{\rmdefault}{\mddefault}{\updefault}{\color[rgb]{0,0,0}$B$}%
}}}}
\end{picture}%

\caption{Pinching off a disk at the seam}
\label{deform} 
\end{figure}

\subsubsection{Exactness to leading order}
\label{leading}

The proof that the maps $\Phi_1,\Phi_2$ of \eqref{phi1}, \eqref{phi2}
fit into a long exact sequence follows a standard argument, familiar
from Floer's exact triangle \cite{br:fl}.  In this argument one first
proves that the ``leading order terms'' in the cochain-level map
define a short exact sequence and then applies a spectral sequence to
deduce the triangle.  Recall that $ L^0 \cap \tau_C^{-1} L^1$ is the
disjoint union of the images of $i_1$ and $i_2$ of the map in Proposition
\ref{bijection}.
 
\begin{theorem} \label{leadingthm} \label{leadingone}    \label{leadingtwo}  \label{zeroindex} 
{\rm (Exactness of the short sequence to leading order)} Let
$C,L^0,L^1$ be as in Theorem \ref{main}.  There exists $\eps > 0$ such
that for any fibered Dehn twist $\tau_C$ defined using \eqref{theta}
for $\lambda$ sufficiently large, the following properties hold:
\begin{enumerate} 
\item \label{smalltri} {\rm (Small triangles as leading order
  contributions to $C\Phi_1$)} If $x \in ((L^0 \times C) \cap (C^t
  \times \tau_C^{-1} L^1))$ then $C\Phi_1(\bra{x})$ is equal to $q^\nu
  \bra{i_1(x)}$ for some $\nu < \eps/2 $, plus terms of the form
  $q^\mu \bra{z}$ with $\mu > \eps$ and $z \in L^0 \cap \tau_C^{-1}
  L^1$.
\item \label{horiz} {\rm (Horizontal sections as leading order
  contributions to $C\Phi_2$)} 
\begin{enumerate} 
\item For any $x \in L^0 \cap L^1$, if $y=i_2(x)$ then $C\Phi_2(
  \bra{y})$ is equal to $\bra{x}$ plus terms of the form $q^\nu
  \bra{z}$ with $z \in \cI(L^0, L^1)$ and $\nu > \eps$.

\item If $y \neq i_2(x)$ for any $x$, then $C\Phi_2( \bra{y})$ is a
  sum of terms of the form $q^\nu \bra{z}$ for $z \in \cI(L^0, L^1)$
  and $\nu > \eps$.
\end{enumerate}
\end{enumerate} 
Furthermore, $C\Phi_2$ has zero degree.
\end{theorem}

\begin{proof}
\eqref{smalltri} We aim to reduce to the exact, unfibered case
considered by Seidel \cite{se:lo}.

\vskip .1in
\noindent {\em Step 1: The antipodal case.}  The simplest case is that
in which the points on the spherical fiber are antipodes.  That is,
let $x = (m,b,m')$ be as in the statement of the Theorem, so that
\begin{equation} \label{pis}
m \in L^0 \cap C, \ m' \in L^1 \cap C, \ \pi(m) = \pi(m') \in B
.\end{equation}
Suppose that $m,m'$ are antipodes: 
$$ b = \pi(m) = \pi(m'), \quad L^0 \cap \tau_C^{-1}L^1 = \{ m \} .$$
Let 
$$u: \ul{S}_1 \to \ul{M}, \quad u|S_M \equiv m, 
\quad u| S_B \equiv b $$
denote the constant section with value $m$ on the part mapping to $M$
and value $b$ on the part mapping to $B$.  This section satisfies the
boundary conditions $L^0, \tau_C^{-1} L^1,C$, by \eqref{pis}.  

We claim that these constant sections give the only contributions to
the relative invariant in the antipodal case.  Let $J \in
\J^h(\ul{E})$ be any horizontal almost complex structure, for which
$u$ is a pseudoholomorphic section with zero area:
$$ A(u) = \int_{\ul{S}} u^* \ul{\omega} = 
A(u|S_M) + A(u|S_B) = 0 .$$
The map $u$ is $J$-regular.  Indeed, by non-negativity of the
curvature and Proposition \ref{horizprop}, it suffices to show that
the index of the linearized operator $D_{u,J}$ for $u$ the constant
section is zero.  Since the boundary conditions are constant, we may
decompose the boundary conditions into vertical and horizontal parts,
$$T_m M \cong T_b B \oplus T_v T^\dual S^c, \quad T_{m} C 
\cong T_b B \oplus T_v S^c $$
for some vector $v \in S^c$ representing $m$ in the trivialization of
the fiber at $b$.  The linearized operator $D_u$ breaks into the sum
of linearized operators $D_u^h, D_u^v$ for the horizontal and vertical
pieces.  The linearized operator $D_u^h$ for the horizontal piece $T_b
B$ may be identified with a linearized operator on a strip with
constant, transverse boundary conditions $T_b (L_0 \circ C)$ and $T_b
(L_1 \circ C)$.  The kernel and cokernel of $D_u^h$ are trivial, hence
the horizontal piece $D_u^h$ has index zero.  The linearized operator
for the vertical piece $D_u^v$ is equivalent to that for the unfibered
case covered in \cite[Proof of Lemma 3.3]{se:lo}.  Since the kernel
and cokernel in this case is also trivial, the cokernel of the
linearized operator for $u$ is trivial.  Hence $J$ is regular for
horizontal $u$.  This implies that we may use $J$ to compute the
coefficient of $\bra{i_1(x)}$ in $C \Phi_1(\bra{x})$, and the fact
that the required sections are constant in this case proves the claim.

\vskip .1in
\noindent {\em Step 2: Deformation to the antipodal case.}  We reduce
to the antipodal case of the previous paragraphs by a deformation
argument.  Suppose that $(m,b,m')$ gives a generalized intersection
point of $(L^0,C^t,C,L^1)$.  Choose a family of identifications $C_b
\cong S^c$ depending on $t \in [0,1]$ such that for $t = 1$ the points
$m,m'$ are antipodes in $C_b$.  We extend the pull-back metrics on
$C_b$ to a family of fiber-wise metrics on $C$.  The family of metrics
defines a reduction of the structure group depending on $t$, that is,
a family of principal $SO(c+1)$ bundles $P_t$ together with a family
of diffeomorphisms
$$P_t \times_{SO(c+1)} S^c \to C .$$
Let $\tau_C^t$ denote the resulting family of fibered Dehn twists, and
consider the family of boundary conditions given by $\tau_C^t L^0,
L^1$.  The intersection points $\cI(\tau_C^t L^0, C^t,C, L^1)$ fit
into smooth families depending on $t$ as in the proof of Proposition
\ref{bijection}.  Consider the parametrized moduli space
\begin{equation} 
\label{parM} \widetilde{\M}(i_1(x),x) = \bigcup_{\rho \in [0,1]}
\{ \rho \} \times \M^\rho(i_1(x),x) \end{equation}
of pseudoholomorphic curves for this deformation.  Standard arguments
show that there is a parametrized regular family of almost complex
structures for the deformation, in the sense that the moduli space
\eqref{parM} is a smooth, finite dimensional manifold with boundary.
The moduli space is also compact, as long as ${\delta}$ is sufficiently
large.  In this case bubbling off trajectories in the deformation
\eqref{parM} is impossible by Remark \ref{impossible} below.  On the
other hand, bubbling off holomorphic disks and spheres is impossible
because of the monotonicity conditions.  It follows from compactness
that the count of holomorphic quilts is invariant under the
deformation \eqref{parM}.

We show that the sections contributing to the coefficient of
$\bra{i_1(x)}$ in $C\Phi_1(\bra{x})$ are of small area.  Suppose $u:
\ul{S}_1 \to \ul{M}$ is a quilt contained in a neighborhood of $C_b$
with Lagrangian boundary conditions $L^0, \tau_C^{-1} L^1, C$ and
limits $x,i_1(x)$.  Near $C_b$ all Lagrangians are exact and path
connected.  The area computation reduces to the computation of action
differences in \cite{se:lo}.  In particular, for these maps, for
${\delta}$ sufficiently large, all such $u$ of index $0$ (parametrized
index $1$) connecting $i_1(x)$ with $x$ have ${\mathcal{E}}(u)$ at
most $\eps/2$:
$$ u \in \M(x,i_1(x)) \implies \mathcal{E}(u) \leq \eps/2 .$$
This argument shows that the area of small index zero quilts is at
most $\eps/2$ for ${\delta}$ sufficiently large.  By monotonicity, any
holomorphic quilt of index zero with boundary/seam conditions $L^0,
\tau_C^{-1} L^1, C$ has the same area as one of the index zero quilts
above.

It follows in particular that sphere and disk bubbling does not occur
in these moduli spaces.  So the component of the moduli space
$\widetilde{\M}_1( i_1(x),x)$ of formal dimension one is compact.
Counting boundary components gives
\begin{equation} \label{twosums} \sum_{u \in \M^{\rho = 0}(i_1(x),x )_0} o(u) = \sum_{ u \in
    \M^{\rho = 1}(i_1(x),x)_0} o(u) \end{equation}
where $o(u) = \pm 1$ are the orientations.  The second sum in
\eqref{twosums} is equal to $1$ if the almost complex structure is
horizontal.  For the unique contribution then comes from the
horizontal section with value $i_1(x) = (m,b,m')$.  This completes the
deformation argument.

\vskip .1in
\noindent {\em Step 3: The higher order terms.} We next show that the
pseudoholomorphic sections with limits other than those corresponding
to $i_1,i_2$ have higher energy, using the mean value inequality.
Suppose that $u = (u_0,u_1)$ is a quilt from $\ul{S}_1$ connecting $x
= (m,b,m')$ with $y$ with energy at most $\eps$ with $u_1$ resp. $u_0$
mapping to $M$ resp. $B$.  For any neighborhood $U$ of $C_b$, there
exists $\eps> 0$ sufficiently small so that the image of $u_1$ is
contained in $U$:
$$ \mathcal{E}(u) < \eps \implies u_1(S_M) \subset U .$$  
Indeed let $(s_0,t_0)$ denote the coordinate of the top-most point in
the seam.  For $s > s_0 + 1$, integrating the mean value inequality
\eqref{mvt} as in Lemma \ref{smalllength} over a path $\gamma_s(t) =
(s,t)$ shows that the distance between $L^0$ and $C$ and between
$\tau_C^{-1} L^1$ and $C$ are at most $c \eps$ for some constant $c$.
In fact these estimates are independent of the choice of $\tau_C$.
Indeed, the image $\tau_C^{-1}(L^1)$ is independent of $\tau_C$ in a
neighborhood of $L^0 \cap L^1$.  On the other hand, $\tau_C^{-1}(L^1)$
converges to $\pi^{-1} \pi(L^1 \cap B)$ as $\zeta(0) \to 0$.  As in
Lemma \ref{approxint}, let $V_\eps$ be the set of points in $X$ within
distance $c\eps$ of both $L^0$ and $\tau_C^{-1} L^1$.  Since
$L^0,\tau_C^{-1} L^1$ are compact, for $\eps$ sufficiently small,
$V_\eps$ is a collection of disjoint open neighborhoods of the points
in $L^0 \cap \tau_C^{-1} L^1$:
$$ V_{\eps} = \bigcup_{z \in L^0 \cap \tau_C^{-1} L^1} V_\eps(z) .$$
It follows that if $u$ has sufficiently small energy then every point
with $s > s_0 + 1 $ maps to $V_\eps$, hence $V_\eps(x)$.  Similarly,
for $s < s_0 - 1$, one obtains by ``folding'' a strip with values in
$M \times B \times M$ with boundary values $L^0 \times C$ and $C^t
\times \tau_C^{-1} L^1$.  Integrating the mean value inequality
\eqref{mvt} over paths to the boundary $\gamma_s(\tau) = (s,t + \tau)$
or $\gamma_s(t) = (s,t- \tau)$ shows that
$$d(\pi_1 u(s,t),L^0) < c \eps, \quad d(\pi_2 u(s,t), C) < c \eps,
\quad d(\pi_3 u(s,t), \tau_C^{-1} L^1) < c \eps .$$ 
In particular note that the intersection $ (L^0 \circ C) \cap
(\tau_C^{-1} L^1 \circ C)$ is transverse and independent of $\tau_C$.
Hence the projections $\pi_1(u(s,t))$ and $\pi_3(u(s,t))$ are
contained a small neighborhood $W$ of $\pi^{-1}( (C \circ L^0) \cap (C
\circ \tau_C^{-1} L^1))$ in $M$.  It follows that $u$ takes values in
$W$ for all $s < s_0 - 1$.  Applying Lemma \ref{smalllength} again,
this time for a path of the form $\gamma(t) = [s_0 + 2t, k]$ for $k
\in \{ 0,1\}$, shows the same holds for points in the boundary of the
intermediate region.  Now $i_1(x)$ is the unique intersection point of
$L^0 \cap \tau_C^{-1}L^1$ connected to $L^0 \cap C_b$ by path from $m$
in $L^0 \cap W$, and a path from $m'$ in $\tau_C^{-1} L^1 \cap W$, as
in the proof of Proposition \ref{bijection}.  Thus $y = i_1(x)$.

\eqref{horiz} The second part of the Lemma is somewhat easier, since
the leading order terms have order exactly zero arising from the
horizontal sections.

\vskip .1in
\noindent {\em Step 1: We show that the degree zero terms arise from
  horizontal sections.}  Let $u$ be the horizontal section of $E_2$ on
the infinite strip with value $x$. Then $\Ind(D_{u,J}) = 0$, since the
boundary conditions are constant.  Hence $u$ is regular for horizontal
$J$, and the count for $y = i_1(x)$ follows by Proposition
\ref{filter}.  The map $\Phi_2$ has degree $0$, since the horizontal
sections have zero index.  Because of the non-negative curvature of
the standard Lefschetz-Bott fibration in \ref{EC}, any non-horizontal
section has positive area.  Thus the $q$-exponent-zero terms arise
only from horizontal sections.

\vskip .1in
\noindent {\em Step 2: The non-horizontal sections satisfy a uniform
  energy gap condition as in the Lemma.}  For an almost complex
structure pulled back from one for $(L^0,C^t,C,L^1)$ under $\Id_M
\times \Id_B \times \tau_C^{-1}$; this is an immediate consequence of
Gromov compactness and preservation of area and index under the
pull-back.  The following alternative argument using the mean value
inequality holds for any almost complex structure used to define the
relative invariant: Suppose $u$ is a section of energy less than
$\eps$. By Lemma \ref{smalllength}, for $\eps$ small $u$ takes values
in the set $U$ of points in $M$ within distance $c\eps$ of both $L^0$
and $L^1$.  By Lemma \ref{approxint}, $U_\eps$ is a collection of
disjoint open neighborhoods of the points in $L^0 \cap L^1$:
$$ U_{\eps}  = \bigcup_{z \in L^0 \cap L^1}  U_\eps(z) .$$
Since $u(s,t)$ lies in $U_\eps(x)$ for $s$ sufficiently large, this
implies that 
$$u(s,t) \in U_\eps(x), \quad s \ge 1 .$$  
For the points in $[-1,1] \times \{ 0, 1 \}$ integration applied to
the path from $(s,0)$ to $(1,0)$ to $(-1,1)$ shows that (again for
$\eps$ sufficiently small)
$$u(s,0) \in U_\eps(x), \quad s \in [-1,1] .$$  
Then another application of integration shows that 
$$u(s,t) \in U_\eps(x), \quad s \leq -1 $$
as well.  Thus $u$ has the same limit $x$ at $s = -\infty$ as $s =
\infty$.  On the intermediate region $s \in [-1,1], t \in [0,1]$ on a
neighborhood of $x$ the section corresponds to a map to $M$.  An
argument using Lemmas \ref{approxint} and \ref{smalllength} shows that
$u$ takes values in $U_\eps(x)$ everywhere.  In this neighborhood
$L^0,L^1$ may be assumed to be exact.  Thus case $u$ has zero area
and, by the vanishing of the curvature in this region, must be
horizontal.
\end{proof}

\begin{remark} \label{impossible} {\rm (Energy gap for Floer trajectories)} 
Let $C,L^0,L^1$ be as in the previous two Lemmas.  We claim that there
exists ${\delta}$ sufficiently large such that any non-constant Floer
trajectory for $(L^0,C^t,C,\tau_C^{-1}L^1)$ or $(L^0,\tau_C^{-1}L^1)$
has symplectic area at least $\eps$.  We consider only the case of
trajectories $u: \R \times [0,1] \to M$ for $(L^0,\tau_C^{-1}L^1)$;
the case of trajectories for $(L^0,C^t,C,\tau_C^{-1}L^1)$ is similar.
Choose an open neighborhood $U$ of $m$ disjoint from $L^1$.  Each
component of $\tau_C^{-1}L^1 \cap U $ converges to $p^{-1}(p(C \cap
L^1))\cap U$ as ${\delta} \to \infty$ as smooth submanifolds.  The
discussion above, again using Lemma \ref{smalllength}, shows that if
$u$ has sufficiently small area $A(u)$ then the image $u(\R \times
[0,1])$ is contained in a small neighborhood of either some
intersection point $m \in L^0 \cap L^1$ or an intersection point in $m
\in L^0 \cap \tau_C^{-1}L^1$.  Since each component of $\tau_C^{-1}L^1
\cap U $ contains at most one intersection point with $m$ and the
image of the $\R \times \{ 1 \}$ under $u$ is connected, this implies
that any index one trajectory $u$ connects an intersection point $m$
to itself and is homotopic to the trivial trajectory.  Hence $u$ has
vanishing area.
\end{remark} 

\subsubsection{Isomorphism with the mapping cone}

Every mapping cone of cochain complexes gives rise to an exact
triangle.  To construct an exact triangle it suffices to prove an
isomorphism of a third complex with a mapping cone.  So to prove
Theorem \ref{main} in the unquilted case (of simple Lagrangians
$L^0,L^1$) it suffices to show the following:

\begin{theorem} {\rm (Isomorphism with the mapping cone)}   Let $L^0,L^1,\eps,\tau_C$ as in 
Theorem \ref{main}.  Then
\label{cone} the map $C\Phi_{2}$ induces
an isomorphism of $CF(L^0,\tau_C^{-1} L^1)$ with the mapping
cone on $C\Phi_{1}$.
\end{theorem} 

Before we give the proof of the theorem we recall a bit of homological
algebra, explained for example in \cite{gm:ho}.

\begin{remark} \label{qi}
\begin{enumerate}
\item {\rm (Mapping cone)} If $C_j= (C_j,\partial_j), j = 0,1$ are
  cochain complexes and $f:C_0 \to C_1$ is a cochain map then the
  mapping cone on $f$ is the complex
$$ \Cone(f) := C_0[1] \oplus C_1, \ \ \partial(c_0,c_1)
= (-\partial_0 c_0,\partial_1 c_1 + f(c_0)) .$$
\item \label{qitwo} {\rm (Quasiisomorphisms from mapping cones)} A
  cochain map from $\Cone(f)$ to a complex $C_2$ is equivalent to pair
  $(k,h)$ consisting of a cochain map $k:C_1 \to C_2$ together with a
  cochain homotopy 
$$h: C_0 \to C_2, \quad k \circ f = h \partial_0 + \partial_2 h .$$
Such a map induces a quasi-isomorphism if and only if
$\Cone(k[1]\oplus h)$ is acyclic.
\end{enumerate} 
\end{remark} 

We will need the following criterion for a cochain map $(k,h)$ as in
Remark \ref{qi} \eqref{qitwo} to induce a quasi-isomorphism, similar
to that in Seidel \cite{se:lo} and Perutz \cite[Lemma 5.4]{per:gys}.
By an {\em $\R$-graded ${\Lambda}$-cochain complex} we mean a
$\Lambda$-linear cochain complex equipped with an $\R$-grading so that
multiplication by $q^\alpha$ shifts the grading by $\alpha$.

\begin{lemma}  \label{doublecone} 
{\rm (Double mapping cone lemma)} Suppose that $C_0,C_1$ and $C_2$ are
free, finitely-generated $\R$-graded ${\Lambda}$-cochain complexes
with differentials $\delta_0,\delta_1,\delta_2$.  Suppose that
$$ C_0 \stackrel{f}{\rightarrow} C_1 \stackrel{k}{\rightarrow} C_2 $$ 
is a sequence of cochain maps (not necessarily exact at the middle
term) and $h: C_0 \to C_2$ a nullhomotopy of $k \circ f$. Assume
\begin{enumerate} 
\item The differentials $\delta_0, \delta_1, \delta_2$ each have
  positive order while $h$ has non-negative order.
\item We have $f = f_0 + f_1$, $k = k_0 +k_1$, where $f_0 , k_0$ have
  order zero while $f_1, k_1$ have positive order.
\item The leading order terms $f_0,k_0$ give a short exact sequence of
  abelian groups (not necessarily cochain complexes)
$$ 0 \to C_0 \stackrel{f_0}{\rightarrow} C_1
  \stackrel{k_0}{\rightarrow} C_2 \to 0 $$
\end{enumerate} 
Then the induced map $(h, k) : \Cone(f) \to C_2$ is a
quasiisomorphism.
\end{lemma}

\begin{proof}  The proof is similar to that in Perutz \cite[Lemma 5.4]{per:gys}.
  The leading order differential in $C := \Cone(\Cone(f),C_2)$ is
  acyclic by standard homological algebra: Given $(c_0,c_1,c_2)$ with
  leading order coboundary $( 0, f_0(c_0), k_0(c_1) + h_0(c_0)) = 0$,
  we have $c_0 =0 $ since $f_0$ is injective, and so $k_0(c_1) = 0$.
  Hence $c_1 = f_0(b_0)$ for some $b_0 \in C_0$.  Now $h_0(b_0) = -
  b_2$ for some $b_2 \in C_2$ and $c_2 + b_2= k_0(b_1)$ for some $b_1
  \in C_1$.  So the coboundary of $(b_0,b_1,b_2)$ is $(0,c_1,c_2+ b_2
  - b_2) = (c_0,c_1,c_2)$ as desired.  Since all maps $f,h,k$ have
  finitely many terms, there exists an $\eps > 0$ such that, if
  $C^{\ge n} \subset C$ is the sub-complex of terms with order in $[n
    \eps,\infty)$, the first page of the associated spectral sequence
    has vanishing cohomology. It follows that $C$ is itself acyclic.
\end{proof}

\begin{proof}[Proof of Theorem \ref{cone}] 
We apply Lemma \ref{doublecone} to the Floer complexes
\begin{equation} \label{c0c1c2} \begin{array}{l} 
C_0 = CF(L^0 , C^t, C, \tau_C^{-1}    L^1;{\Lambda}) \\ 
C_1 = CF(L^0,\tau_C^{-1}
    L^1;{\Lambda}) \\ C_2 = CF(L^0, L^1,{\Lambda}) \end{array}
  .\end{equation}
Let $k$ denote the cochain level map $C \Phi_{1,\Lambda}$ defined by
the quilted surface in Figure \ref{pants}.  For any $x \in \cI(L^0
,C^t,C,\tau_C^{-1} L^1)$, $z \in \cI(L^0, L^1)$ consider the
parametrized moduli space $\widetilde{\M}(x,z)$ for the one-parameter
family of deformations $(\ul{E}_\rho,\ul{F}_\rho)$ (with $\rho$
representing the length of the neck) connecting the pair obtained by
gluing $(\ul{E}_1,\ul{F}_1)$ and $({E}_2,{F}_2)$ along a strip-like
end to the nodal surface equipped with bundles $(\ul{E},\ul{F}),
(\ul{E}_{C,r},\ul{F}_{C,r})$, as shown in Figure \ref{deform}.  An
element of $\widetilde{\M}(x,z)$ consists of a quilted surface
$\ul{S}^\rho$ in the family, together with a pseudoholomorphic quilt
$u^\rho$ with the given boundary and seam conditions.

As in the proof of Corollary \ref{pinch}, for $r$ sufficiently small
and $\rho$ sufficiently large the parametrized moduli space is empty.
The boundary of $\widetilde{\M}(x,z)_1$ is therefore
\begin{multline} \label{consists}
 \partial \widetilde{\M}(x,z)_1 = \bigcup_y \left( \M_0(x,y)_0 \times
 \M_1(y,z)_0 \right) \cup \bigcup_y \left( \widetilde{\M}(x,y)_0 \times \M
 (y,z)_0 \right) \\ \cup \bigcup_y \left( \widetilde{\M}(y,z)_0 \times
 \M(x,y)_0 \right);
 \end{multline}
where 
\begin{itemize}
\item the first union consists of pairs of pseudoholomorphic sections
  of $\ul{E}_1$ and $\ul{E}_2$, and 
\item the second two unions correspond to bubbling off Floer
  trajectories $[u] \in \M(y,z)_0$ in $M$ or a Floer trajectory $ [u]
  \in \M(x,y)_0$ in $M \times B \times M$.  
\end{itemize}
Define a map
$$ h: CF(L^0 ,C^t,C,\tau_C^{-1} L^1;{\Lambda}) \to
CF(L^0,L^1;{\Lambda})
$$
by 
$$ h( \bra{x} ) = \sum_{[u] \in \widetilde{\M}(x,y)_0} o(u) q^{A(u)} \bra{y} $$
where $o(u) = \pm 1$ are the orientations constructed in
\cite{orient}.  The description of the boundary components of
$\widetilde{\M}(x,z)_1$ in \eqref{consists} gives the relation $
\partial h + h \partial = k \circ f $.  By Remark \ref{qi} the pair
$(k,h)$ define a morphism $ \Cone(f) \to C_2 .$ We claim that the
mapping cone
\begin{multline} \label{conekh}
 \Cone(k,h) = C_0[2] \oplus C_1[1] \oplus C_2, \\ \partial(c_0,c_1,c_2) =
(\partial_0 c_0, \partial_1 c_1 + f(c_0), \partial_2 c_2 + k(c_1) +
h(c_0))
\end{multline}
is acyclic.  Theorem \ref{leadingthm} shows that for Dehn twists
satisfying the conditions in the Theorem, the differential $\partial$
splits into the sum of an operator $\partial_0$ whose $q$-exponents
lie in $[0,\eps/2)$ and a term $\partial - \partial_0$ whose
  $q$-exponents lie in $(\eps,\infty)$.  Because of the bijection
  \eqref{bij} there exists an $\R$-grading on $C_0,C_1,C_2$ so that
  the contributions to $\partial$ separates into an acyclic leading
  order part $\partial_0$ with $\R$-degree zero and a remaining part
  with $\R$-degree at least $\eps/2$.  By Lemma \ref{vanishes}, there
  exists a null-homotopy of $k \circ f$ to the identity.  By Remark
  \ref{impossible}, for $\delta$ sufficiently small the conditions of
  Lemma \ref{doublecone} hold.  The result for ${\Lambda}$
  coefficients follows.  Note that here in contrast to the case in
  Perutz \cite{per:gys} the differentials have finitely many terms, so
  the formal completion is not necessary.  The result for $q = 1$
  follows by specialization as in Remark \ref{spec}.
\end{proof}

The quilted version of Theorem \ref{main}, where $\ul{L}^0,\ul{L}^1$
are generalized Lagrangian branes, is proved similarly, but replacing
the boundary labelled $L^0,L^1$ with collections of strips
corresponding to the symplectic manifolds appearing in the generalized
Lagrangian branes $\ul{L}^0,\ul{L}^1$.  The details are left to the
reader.

\begin{remark} \label{codimonemain} Under strong monotonicity assumptions as in Definition \ref{stronglymonotone}, 
the statement of Theorem \ref{main} also holds in the case of
codimension one coisotropics.  This case is proved similarly but now
using the cancellation discussed in Remark \ref{codimone}.
\end{remark} 

\subsection{Minimal Maslov two case} 
\label{maslovtwo} 

In general, Lagrangian Floer cohomology is defined only the case that
certain holomorphic disk counts vanish.  In the case that one of the
Lagrangians has minimal Maslov number two, the relevant disk count is
that of Maslov index two holomorphic disks.  First we recall some
basics of the derived category of matrix factorizations from
\cite{fieldb}.

\begin{definition} \label{fact}
\begin{enumerate} 
\item {\rm (Category of matrix factorizations)} For any $w \in \Z$,
  let $\Fact(w)$ denote the category of factorizations of $w\Id$.
\begin{enumerate}
\item 
The objects of $\Fact(w)$ consist of pairs $(C,\partial)$, where%
$$C
= C = C^0 \oplus C^1$$
is a $\Z_2$-graded free abelian group and $\partial$ is a group
homomorphism squaring to a multiple of the identity:
$$\partial : C^\bullet \to C^{\bullet + 1} , \quad \partial^2 = w\Id
.$$
\item For objects $(C,\partial),(C',\partial')$, the space of
  morphisms 
$$\Hom_{\Fact}((C,\partial),(C',\partial'))$$ 
is the space of grading preserving maps intertwining the ``differentials''
$$f: C^\bullet \to (C')^\bullet, \quad f \partial = \partial' f .$$
\end{enumerate}
\item {\rm (Cohomology)} For any matrix factorization $(C,\partial)$
  let $H((C,\partial)\otimes_\Z \Z_w)$ denote the cohomology of the
  differential $\partial \otimes_\Z {\rm Id} : C \otimes_\Z \Z_w \to
  C\otimes_\Z \Z_w $ obtained from $\partial$ by tensoring with
  $\Z_w$.  Any morphism in $\Fact(w)$ defines a homomorphism of the
  corresponding cohomology groups. The {\em cohomology with
    coefficients} functor has target the category $\Ab$ of
  $\Z_2$-graded abelian groups,
$$
 \Fact(w) \to \Ab, \ \ (C,\partial) \mapsto H( (C,\partial) \otimes_\Z
\Z_w) .$$
\end{enumerate} 
\end{definition} 

We recall the definition of the Maslov index two disk count studied in
Oh \cite{oh:fl1}.  Let $M$ be a symplectic background and $L \subset
M$ be a compact Lagrangian submanifold with minimal Maslov number
equal to $2$.  For any $\ell \in L$, consider the moduli space
$\M^1_2(L,J,\ell)$ of $J$-holomorphic disks $u:(D,\partial D) \to
(M,L)$ with Maslov number $2$, mapping a point $1 \in \partial D$ to
$\ell$, modulo automorphisms preserving $1$.  By results of Kwon-Oh
\cite{ohkw:st} and Lazzarini \cite{la:ex}, for $J$ in a comeager
subset $\J^{\reg}(M,L) \subset \J(M)$ the moduli space
$\M^1_2(L,J,\ell)$ is a finite set.  Suppose $L$ is equipped with a
relative spin structure.  By \cite{fooo}, see also \cite{orient}, this
structure induces an orientation on the moduli space
$\M^1_2(L,J,\ell)$.  Let
$$o: \M^1_2(L,J,\ell) \to \{ \pm 1 \} $$ 
denote the map comparing the given orientation to the canonical
orientation of a point.

\begin{definition} {\rm (Disk invariant of a Lagrangian)}  
Let $M$ be a monotone symplectic background and $L \subset M$ a
compact monotone Lagrangian brane with minimal Maslov number equal to
$2$.  The {\em disk invariant} of $L$ is the sum
$$ w(L) = \sum_{[u] \in \M^1_2(L,J,\ell)} o(u) q^{A(u)} \in \Lambda
.$$ 
% C
The element $w(L)$ is independent of $J \in \J^\reg(M,L)$ and
$\ell \in L$.
\end{definition}  

Denote by $\J_t(M,L^0,L^1) \subset \J_t(M)$ the subset of
$t$-dependent almost complex structures whose restriction to a fixed
small neighborhood of $ t = 0$ resp. $t = 1$ lies in $\J^\reg(M,L^0)$
resp. $\J^\reg(M,L^1)$.

\begin{proposition}  {\rm (Floer cohomology)} Let $L^0,L^1$ compact Lagrangian branes in $M$.   There exists a comeager subset
  $\J^\reg_t(M,\omega,L^0,L^1) \subset \J_t(M,\omega,L^0,L^1) $ such
  that
\begin{enumerate} 
\item $\partial^2 = (w(L^0) - w(L^1))\Id$.
\item $(CF(L^0,L^1;\Lambda), \partial)$ is independent of the choice
  of $J,H$ up to cochain homotopy.
\end{enumerate}  
\end{proposition}  

If $L^k \subset M, k= 0,1$ are monotone Lagrangian branes with the
same disk invariant, then the exact triangle in Theorem \ref{main}
holds, with the same proof.  More generally, the disk invariant for a
generalized Lagrangian brane $\ul{L} = (L_1,\ldots, L_k)$ is the sum
of the disk invariants for the components $L_1, \ldots, L_k$.
Furthermore, The disk invariant for a correspondences $L_{01} \subset
M_0^- \times M_1$ is the {\em opposite} of the disk invariant for its
transpose $L_{01}^t \subset M_1^- \times M_0$, because the change in
complex structure reverses orientations on Maslov index two disks
\cite{orient}.  With these conventions, the quilted version of Theorem
\ref{main} also holds provided that the disk invariants of the
generalized correspondences $\ul{L}^0,\ul{L}^1$ are equal.

\subsection{Periodic Floer version}

One can also formulate a version of the exact triangle for
symplectomorphisms, that is, in periodic Floer theory.  In this
formulation, the exact triangle relates the symplectic Floer
cohomology of the Dehn twist with the Lagrangian Floer cohomology of
the vanishing cycle and the identity:

\begin{theorem}\label{period}  {\rm (Exact triangle in periodic quilted Floer theory)}  
Let $M,C,B,\tau_C$ be as in Theorem \ref{main}.  There exists a long
exact sequence of Floer cohomology groups
$$ \tri{\hskip -.2in$HF(\on{id})$}{
$HF(C^t,C)[\dim(B)]$.}{$HF(\tau_C)$} $$
\end{theorem} 

The ideas are similar to the case of Lagrangian Floer cohomology in
Theorem \ref{main} and we only sketch the proof.  The map $HF(\on{id})
\to HF(\tau_C)$ is the relative invariant corresponding to a
Lefschetz-Bott fibration over the cylinder in Figure \ref{quiltcyl3}
below.  In the Figure the outer and inner boundary represent
cylindrical ends, where the former has monodromy $\tau_C$.
\begin{figure}[h]
\begin{picture}(0,0)%
\includegraphics{quiltcyl3.pstex}%
\end{picture}%
\setlength{\unitlength}{3947sp}%
\begingroup\makeatletter\ifx\SetFigFont\undefined%
\gdef\SetFigFont#1#2#3#4#5{%
  \reset@font\fontsize{#1}{#2pt}%
  \fontfamily{#3}\fontseries{#4}\fontshape{#5}%
  \selectfont}%
\fi\endgroup%
\begin{picture}(2877,2698)(286,-3651)
\put(301,-2002){\makebox(0,0)[lb]{$\tau_C$}%
}
\put(1390,-2402){\makebox(0,0)[lb]{$\on{id}$}%
}
\end{picture}%

\caption{Lefschetz-Bott fibration defining
  $HF(\on{id}) \to HF(\tau_C)$}
\label{quiltcyl3}
\end{figure}
The map $HF(C^t,C) \to HF(\on{id})$ is the relative invariant
associated to the quilted cylinder in Figure \ref{quiltcyl4}.  In the
Figure, the outer boundary represents a cylindrical end while the
inner boundary represents a quilted cylindrical end with seams
labelled $C^t,C$.
\begin{figure}[h]
\begin{picture}(0,0)%
\includegraphics{quiltcyl4.pstex}%
\end{picture}%
\setlength{\unitlength}{3947sp}%
\begingroup\makeatletter\ifx\SetFigFont\undefined%
\gdef\SetFigFont#1#2#3#4#5{%
  \reset@font\fontsize{#1}{#2pt}%
  \fontfamily{#3}\fontseries{#4}\fontshape{#5}%
  \selectfont}%
\fi\endgroup%
\begin{picture}(2873,2698)(286,-3651)
\put(301,-2302){\makebox(0,0)[lb]{$$}%
}
\put(1392,-2302){\makebox(0,0)[lb]{$C^t{\sharp} C$}%
}
\end{picture}%
\caption{Quilted surface defining $HF(C^t,C)[\dim(B)] \to HF(\on{id})$}
\label{quiltcyl4}
\end{figure} 
The proof is similar to that of Theorem \ref{main}. Namely, we have identifications
\begin{equation} \label{ident} HF(\on{id}) \cong HF(L^0,L^1), 
\quad HF(\tau_C^{-1} ) = HF(L^0, (\tau_C^{-1} \times 1)L^1) \end{equation}
where $L^0 = L^1 = \Delta$ is the diagonal in $M^2$. There is a
natural bijection
$$ \cI(L^0,(\tau_C^{-1} \times 1)L^1) \to \cI(L^0,L^1) \cup \cI(L^0,
C^t, C, L^1) ;$$ 
this amounts to repeating the argument of Proposition \ref{bijection}.  The
same filtration arguments as before are used to construct an exact
triangle of {\em quilted} Floer cohomology groups
$$ \tri{\hskip -.7in$HF(L^0,\tau_C^{-1}L^1)$}{\hskip -.5in
  $HF(L^0,C^t,C,L^1)[\dim(B)]$.}{$HF(L^0,L^1)$} .$$
Now functoriality of quilted Floer cohomology groups under the symplectomorphism $\tau_C$
gives an exact triangle 
\vskip .1in
$$ \tri{\hskip -.5in$HF(L^0,L^1)$}{
 \hskip -.5in $HF(L^0,C^t,C,L^1)[\dim(B)]$.}{$HF(L^0,\tau_C L^1)$} $$
Using the identifications in \eqref{ident} this proves Theorem
\ref{period}.  With integer coefficients, there is a complication
caused by the fact that the diagonal correspondence $\Delta$ does not
have a canonical relative spin structure.  However, the proof goes
through by treating the diagonal as the ``empty correspondence'', see
Section \ref{closedfukaya} below.

\section{Applications to surgery exact triangles}
\label{field} 

In this section we apply the exact triangle to obtain versions of the
Floer \cite{don:floer}, Khovanov \cite{kh:ca}, and Khovanov-Rozansky
\cite{kr:ma} exact triangles.  We have already established in Section
\ref{flat} that Dehn twists of surfaces induce fibered Dehn twists of
moduli spaces.

\subsection{Exact triangle for three-bordisms} 

In \cite{field}, we introduced a category-valued-field theory
associated to certain connected, decorated surfaces and bordisms.
We use freely the notation and terminology from \cite{field}:

\begin{definition} 
\begin{enumerate} 
\item {\rm (Decorated surface)} A decorated surface is a compact
  connected oriented surface $X$ equipped with a line bundle with
  connection $D \to X$.  The {\em degree} of $D$ is the integer $d =
  (c_1(D),X)$.
\item {\rm (Decorated bordisms)} A {\em bordism} from $X_-$ to $X_+$
  is a compact, oriented $3$-manifold $Y$ with boundary equipped with
  an orientation-preserving diffeomorphism
$$\phi: \partial Y \to \ol{X}_-
  \sqcup X_+ .$$  
Here $\ol{X}_-$ denotes the manifold $X_-$ equipped with the opposite
orientation.  We often omit $\phi$ from the notation and assume that
$X_-,X_+$ are boundary components of $Y$.  A {\em decorated bordism}
of degree $d$ is a compact connected oriented bordism $Y$ between
decorated surfaces $(X_\pm,D_\pm)$ equipped with a line-bundle-with
connection $D \to Y$ such that $D | X_\pm = D_\pm$.  Given a decorated
bordism $(Y,D)$ between $(X_\pm,D_\pm)$ one can obtain another
decorated bordism by tensoring $D$ with any line bundle that is
trivial on the boundary of $Y$.
\item {\rm (Moduli spaces for decorated surfaces)} Suppose that $X$ is
  a decorated surface with line bundle $D$.  Let $M(X,D)$ denote the
  moduli space of constant curvature rank $2$ bundles on $X$ with
  fixed determinant.  The space $M(X,D)$ is a compact, monotone
  symplectic manifold with monotonicity constant $1/4$ with a unique
  Maslov cover of any even order.  In \cite{field} we describe how to
  equip $M(X,D)$ with the background classes so that $M(X,D)$ is
  equipped with the structure of a symplectic background.
\item {\rm (Correspondences for decorated bordisms)} Suppose that
  $Y$ is a decorated surface with line bundle $D$ between decorated
  surfaces $(X_-,D_-)$ and $(X_+,D_+)$.  Let 
$$L(Y,D) \subset M(X_-,D_-)^- \times M(X_+,D_+)$$ 
denote the set of isomorphism classes of constant central curvature
bundles with fixed determinant that extend over the interior of $Y$.
If $Y$ is an elementary bordism (that is, admits a Morse function
with a single critical point) then $L(Y,D)$ is a smooth Lagrangian
correspondence.  In \cite{field} we describe how to equip $L(Y,D)$
with relative spin structures and gradings, so that $L(Y,D)$ has the
structure of an admissible Lagrangian brane.
\end{enumerate}
\end{definition} 

For the purposes of the exact triangle, we will need an alternative
description as flat bundles with fixed holonomy around an additional
marking.

\begin{definition} 
\begin{enumerate}
\item {\rm (Marked surfaces)} For any integer $n \ge 0$, an $n$ {\em
  marked surface} is a compact oriented surface $X$ equipped with a
  tuple $\ul{x} = (x_1,\ldots, x_n)$ of distinct points on $X$ and a
  labelling $\ul{\mu} = (\mu_1,\ldots,\mu_n)$.  In this section we
  take labels all equal to $1/2$ corresponding to the central
  element $-I$ of $\Alc$.
\item {\rm (Moduli spaces for marked surfaces)} Denote by
  $M(X,\ul{\mu})$ the moduli space of flat bundles on $X - \ul{x}$ with
  holonomy $-I$ around each marking $x_i$; we consider more general
  holonomies in the next subsection.  The moduli space $M(X,\ul{\mu})$
  may be identified with the moduli space $M(X,D)$ where $D$ is a line
  bundle of degree $n$, by a construction described in Atiyah-Bott
  \cite{at:mo} given by twisting by a fixed central connection.
\item {\rm (Marked bordisms)} A {\em marked bordism} from
  $(X_-,\ul{x}_-)$ to $(X_+,\ul{x}_+)$ is a compact, oriented
  bordism $Y$ equipped with a tangle (compact oriented one-dimension
  submanifold transverse to the boundary) $K \subset Y$ such that $K
  \cap X_\pm = \ul{x}_\pm$.  
\item {\rm (Moduli spaces for marked bordisms)} Denote by 
$$L(Y,K) \subset M(X_-,\ul{\mu}_-)^- \times M(X_+,\ul{\mu}_+)$$ 
the moduli space of bundles that extend over $Y - K$ with holonomy
$-I$ around $K$.  The identification $M(X_\pm,\ul{\mu}_\pm) \to
M(X_\pm,D_\pm)$ induces a homeomorphism $M(Y,K) \to M(Y,D)$ where $D
\to Y$ is a line bundle whose first Chern class is dual to the tangle
$K$.  Thus, in particular, the addition to $K$ of a circle component
$K'$ corresponds to twisting the determinant line bundle $D$ by a line
bundle whose first Chern class is dual to the homology class of $K'$.
\end{enumerate} 
\end{definition} 

The following proposition relates the moduli spaces of bundles with
fixed holonomy around an embedded circle with the Lagrangian
correspondences associated to elementary bordisms.

\begin{proposition} {\rm (Correspondences for elementary bordisms)} 
\label{tangledef} 
\begin{enumerate}
\item 
Let $Y$ be a bordism from $X_-$ to $X_+$ containing a single
critical point of index $1$ and a trivial tangle $K$ (that is, a union
of intervals connecting $\ul{x}_-$ to $\ul{x}_+$) and $C \subset X_+$
is the attaching cycle.  The Lagrangian $L(Y,K)$ is diffeomorphic via
the projection to $M(X_+,\ul{\mu}_+)$ to the subset of connections on
$X_+ - \ul{x}_+$ with holonomy along $C$ equal to $I$.
\item Let $Y$ be a decorated bordism from $X_-$ to $X_+$ containing
  a single critical point of index $1$, $C \subset X_+$ the attaching
  cycle, $K_0$ a trivial bordism connecting $\ul{x}_-$ to $\ul{x}_+$
  and $K_1 \subset Y$ the unstable manifold of the critical point.
  The Lagrangian $L(Y,K_0 \cup K_1)$ is diffeomorphic to the subset of
  flat bundles on $X_+ - \ul{x}_+$ with holonomy along $C$ equal to
  $-I$.
\end{enumerate}  
\end{proposition} 

\begin{proof} By Seifert-van Kampen, $\pi_1(Y - K)$ is the quotient 
of $\pi_1(X_+ - K)$ by the ideal generated by the element $[C]$
obtained from $C$ by joining by a path to the base point.  Hence in
the first case, $L(Y,K)$ is diffeomorphic to the submanifold of
$M(X,\ul{\mu}_+)$ obtained by setting the holonomy along $C$ equal to
the identity.  For the second case, the gradient flow the Morse
function defines a deformation retract of $Y - K_0 - K_1$ to $X_+ -
\ul{x}_+$.  Homotopy invariance implies that $\pi_1(Y - K_0 - K_1 )$
is isomorphic to $\pi_1(X_+ - \ul{x}_+)$.  Since $C$ is a loop around
$K_1$, the holonomy around $K_1$ is equal to the holonomy along $C$,
hence the claim.
\end{proof}

In order to obtain smooth Lagrangian correspondences, we break the
given bordism into elementary bordisms.

\begin{definition}
\begin{enumerate} 
\item {\rm (Cerf decompositions)} A {\em Cerf decomposition} of a
  bordism $Y$ is a decomposition of $Y$ into elementary bordisms
  $Y_1,\ldots,Y_K$, that is, bordisms admitting a Morse function
  with at most one critical point.  Associated to any Cerf
  decomposition
$$ Y = Y_1 \cup_{X_1} \ldots \cup_{X_{k-1}} Y_k $$
and a decoration on $Y$ is a {\em generalized Lagrangian
  correspondence}
$$\ul{L}(Y) = (L(Y_1),\ldots, L(Y_k)) .$$
\item 
  The generalized correspondence $\ul{L}(Y)$ may be equipped with a
  relative spin structure via its structure as a fibration over the
  moduli space of the incoming or outgoing surface.  Thus $\ul{L}(Y)$
  gives rise to a functor of generalized Fukaya categories
$$ \Phi(Y) : \Fuk^{\sharp}(M(X_-)) \to \Fuk^{\sharp}(M(X_+)) .$$
This functor is independent of the choice of Cerf decomposition
\cite{field}.
\item 
Given admissible Lagrangian branes $L^\pm \subset M(X_\pm)$, define
\begin{eqnarray*} HF(Y;L^-,L^+) &:=& H(\Hom_{\Fuk^{\sharp}(M(X_+))}(\Phi(Y)L^-,L^+)) \\ 
&:=& HF(L^-, \ul{L}(Y), L^+) \end{eqnarray*}
where the second equality is by definition of the Fukaya category.
Since we have ignored absolute gradings
% HERE
and the minimal Chern number is two 
this is a relatively $\Z_4$-graded group depending only on the
equivalence class of the decorated bordism $Y$.
\end{enumerate} 
\end{definition} 

We prove the following surgery exact triangle for the invariants
$HF(Y;L^-,L^+)$.

\begin{definition} 
\begin{enumerate} 
\item {\rm (Knots)} A {\em knot} in a bordism $Y$ is an embedded,
  connected $1$-manifold $K \subset Y$ disjoint from the boundary.
\item {\rm (Knot framings)} A {\em framing} of a knot $K \subset Y$ is
  a non-vanishing section of its normal bundle 
up to homotopy.  Given a framed knot, the other framings are obtained
by twisting by representations of $\pi_1(K) \cong \Z$ into $SO(2)$,
and so are indexed by $\Z$.
\item {\rm (Knot surgeries)} For $\lambda \in \Z$ the {\em
  $\lambda$-surgery} $Y_{\lambda,K}$ of $Y$ along $K$ is obtained by
  removing a tubular neighborhood of $K$ and gluing in a solid torus
  $D^2 \times S^1$ so that the {\em meridian} $\partial D^2 \times \{
  \pt \}$ is glued along the curve given by the framing of the knot
  corresponding to $\lambda$.  Denote by $K_\lambda$ the knot in
  $Y_{\lambda,K}$ corresponding to a longitude in $\partial D^2 \times
  S^1$. Thus $K_\lambda$ intersects the meridian transversally once.
\end{enumerate}
\end{definition} 

\begin{remark} \label{elem}
{\rm (Knot surgeries in terms of decompositions into elementary
  bordisms)} The three-manifolds $Y_{0,K},Y_{-1,K}$ obtained by a
$0$ resp. $-1$-knot surgery have decompositions into elementary
bordisms given as follows.  Suppose that $Y$ is decomposed into
elementary bordisms $Y_1,\ldots, Y_l$, so that $K$ is contained in
the boundary $(\partial Y_i)_+ = (\partial Y_{i+1})_-$ and the framing
is the direction normal to the boundary.  Gluing in $D^2 \times S^1$
produces two new critical points.  The first critical point has stable
manifold with unit sphere equivalent to $K$ and the second has
unstable manifold with unit sphere equivalent to $K$.  Thus
\begin{enumerate} 
\item The zero-surgery $Y_{0,K}$ has a decomposition into elementary
  bordisms with two additional pieces, $Y_\cup, Y_\cap$ inserted
  between $Y_i$ and $Y_{i+1}$.  The knot $K_0 \subset Y_{0,K}$ is
  divided into the two additional pieces $K_0 \cap Y_\cup$ and $K_0
  \cap Y_\cap$.  The correspondence $L(Y_\cup, Y_\cup \cap K_0) \circ
  L( Y_\cap ,Y_\cap \cap K_0)$ is the correspondence associated to the
  moduli space of bundles $L(Y_\cup) \circ L(Y_\cap)$ on the decorated
  surface $Y_\cup \cup Y_\cap$ with the {\em shifted line bundle} as
  in Proposition \ref{tangledef}.
\item The $-1$ surgery $Y_{-1,K}$ has decomposition into simple
  bordisms $Y_1,\ldots, Y_l$ but where the identification $(\partial
  Y_i)_+ \to (\partial Y_{i+1})_-$ is the Dehn twist along $K$.
\end{enumerate} 
\end{remark} 

\begin{lemma} \label{compat} {\rm (Existence of Cerf decompositions compatible with a knot)}  
For any framed knot $K \subset Y$, there exists a Cerf decomposition
$Y= Y_1 \cup \ldots \cup Y_l$ so that $K$ is contained in the boundary
$(\partial Y_i)_+ = (\partial Y_{i+1})_-$ for some $i = 1,\ldots, l$
and the framing is the direction normal to the boundary.
\end{lemma} 

\begin{proof}   Choose a Morse function $f: Y \to \R$ such that
$f$ is constant on $K$ and the framing is given by the gradient of $f$
  at $K$.  The level set $f^{-1}(\lambda)$ containing $K$ can be made
  connected by adding $1$-handles in $Y$, so that $f^{-1}(\lambda)$
  becomes a connected surface containing $K$ separating the boundary
  components of $Y$.  By taking a self-indexing Morse function on
  $f^{-1}((\pm \infty,\lambda])$ one obtains that $f$ is Morse and has
connected fibers.  \end{proof}

\begin{theorem} \label{floertri}  {\rm (Exact triangle for knot surgery)}  
 Let $Y$ be a decorated bordism from $X_-$ to $X_+$, let $K \subset
 Y$ be a framed knot contained in the interior of $Y$, and let
 $Y_{-1,K},Y_{0,K}$ denote the $-1$ and $0$-surgeries on $K$.  Let
 $L^-,L^+$ be admissible Lagrangian branes in $M(X_\pm)$.  There is a
 long exact sequence of (relatively graded) Floer cohomology groups
$$ \ldots \to HF(Y_{0,K};L^-,L^+) \to HF(Y;L^-,L^+) \to
 HF(Y_{-1,K};L^-,L^+) \to \ldots $$
where the determinant bundle on $Y_{0,K}$ has been shifted by the dual
class of the knot $K_0 \subset Y_{0,K}$, or equivalently, $Y_{0,K}$ is
considered as a marked bordism with knot $K_0$.
\end{theorem}  

\begin{proof} 
  By combining Remark \ref{elem}, Lemma \ref{compat}, Theorem
  \ref{nomarkings}, the statement becomes a special case of Theorem
  \ref{main}: Let $Y = Y_- \cup_X Y_+$ be a decomposition into
  compression bodies with $K \subset X$ a non-separating knot and
  $\tau$ a corresponding Dehn twist.  We have an exact sequence
\begin{multline}  \ldots \to HF(L(Y_-),
L(Y^\cup, Y^\cup \cap K_0), L(Y^\cap, Y^\cap \cap K_0),
L(Y_+)) \to HF(L(Y_-),L(Y_+)) \\ \to
HF(L(Y_-),\graph((\tau^{-1})^*),L(Y_+)) \to \ldots .\end{multline}
The addition of the knot $K_0$ is equivalent to a shift in the
determinant line bundle, as explained in Remark \ref{elem}; we thank
Guillem Cazassus for pointing out the missing shift in an earlier
version of the paper.
\end{proof}  

\begin{remark} {\rm (Dehn twists for separating curves)}   
We do not discuss here the exact triangle for a Dehn twist around a
separating curve, because the moduli space $M$ does not satisfy the
``strong monotonicity'' condition of Definition
\ref{stronglymonotone}.  Instead, one needs to establish positivity
properties of the form $c_1(M_C) - [B_-] - [B_+]$ with respect to the
{\em canonical} complex structure on the moduli space of parabolic
bundles $M_C$.  This seems like to hold but would take us outside the
framework of this paper.
\end{remark} 

\subsection{Exact triangles for tangles}

We obtain Floer-theoretic invariants of tangles constructed in
\cite{fieldb} exact triangles that are the same as those obtained by
Khovanov \cite{kh:ta,kh:ca} and Khovanov-Rozansky \cite{kr:ma}.  We
assume freely the terminology from \cite{fieldb}, in particular the
terminology for moduli spaces of bundles with fixed holonomy for
marked surfaces.  In this subsection we take $G = SU(2)$ and $\Alc
\cong [0,1/2]$ via the identification \eqref{interval}.

\begin{definition} {\rm (Functors for bordisms with tangles via flat bundles with fixed holonomy)} 
\begin{enumerate} 
\item {\rm (Correspondences for elementary tangles)} Let $X_\pm$ be a
  compact, oriented surface with odd numbers of markings $\ul{x}_\pm$
  admissible labels $\ul{\mu}_\pm$ all equal to $1/4 \in \Alc$.  Let
  $K \subset Y: =X \times [-1,1]$ be a tangle, that is, a bordism
  between marked surfaces $(X_-,\ul{\mu}_-)$ and $(X_+,\ul{\mu}_+)$.
  Let $M(X_\pm,\ul{\mu}_\pm)$ denote the moduli space of $SU(2)$
  bundles in Definition \ref{labels}.  Let $L(K)$ denote the subset of
  $M(X_-,\ul{\mu}_-)^- \times M(X_+,\ul{\mu}_+)$ of bundles that
  extend over the interior of the bordism.  Assuming that $K$ is
  elementary (admits a Morse function with at most one critical point
  on $K$, and none on $Y$) the Lagrangian $L(K)$ has an admissible
  brane structure \cite{fieldb}.
\item {\rm (Functors for tangles)} More generally, let $K \subset Y =
  X \times [-1,1]$ be an arbitrary tangle, and $K = K_1 \cup \ldots
  \cup K_r$ a decomposition into elementary tangles.  Associated to
  each elementary tangle $K_i$ is a Lagrangian correspondence $L_i$
and so a generalized Lagrangian correspondence associated to $K$
$$ \ul{L}(K) = (L_1,\ldots, L_r ) .$$
Composing the functors $\Phi(L_i)$ for these correspondences gives a
functor
$$ \Phi(K): \Fuk^{\sharp}(M(X_-,\ul{\mu}_-)) \to
\Fuk^{\sharp}(M(X_+,\ul{\mu}_+)) ;$$
In \cite{fieldb} we proved that $\Phi(K)$ is independent, up to
isomorphism, of the choice of decomposition $K_1 \cup \ldots \cup
K_r$.
\item {\rm (Group valued invariants)} Given objects 
$$L^\pm \in
  \Obj(\Fuk^{\sharp}(M(X_\pm,\ul{\mu}_\pm))) ,$$ 
such that the sum of disk invariants for $L_-, L(K), L_+$ is zero, the
cohomology
$$ HF(K;L^-,L^+) =
H(\Hom_{\Fuk^{\sharp}(M(X_+,\ul{\mu}_+))}(\Phi(K)L^-, L^+)) $$
is a ($\Z_2$-relatively graded) invariant of $K$.  
\end{enumerate} 
\end{definition} 

We prove the following surgery exact triangle for these invariants.

\begin{definition}   \label{cupcap} 
Given a tangle $K \subset Y$, a separating embedded surface $\Sigma
\subset Y$, and a disk $D \subset \Sigma$ meeting $K$ in two points.
let $K^{\cupcap}, K^{\twist}$ be the tangles obtained by modifying $K$
by a half-twist, respectively adding a cup and cap as in Figure
\ref{crossingchange}.
\end{definition} 

\begin{figure}[h]
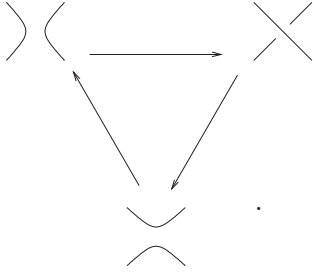

\vskip .2in
$$ \tri{\hskip -.2in \three}{\one}{\two} .$$
\vskip .2in
\caption{Exact triangle for a crossing change, $SU(2)$ case}
\label{crossingchange}
\end{figure}

\begin{theorem} \label{tantri} {\rm (Exact triangle for changing a crossing)}  
Let $K,K^{\cupcap},K^{\twist}$ be tangles in a cylindrical bordism
$Y = X \times [-1,1]$ between $X_\pm = X \times \{ \pm 1 \}$ as in
Definition \ref{cupcap} and $L^\pm$ admissible Lagrangian branes in
$M(X_\pm,\ul{\mu}_\pm)$ such that the sum of the disk invariants for
$\ul{L}(K), L^-,L^+$ is zero. The cohomology groups for
$K^{\twist},K,K^{\cupcap}$ are related by an exact triangle
\vskip 0in\begin{equation} \label{tri2} \tri{\hskip
  -1in$HF(K;L^-,L^+)$}{$HF(K^{\cupcap};L^-,L^+)$}{$HF(K^{\twist};L^-,L^+).$}
 \hskip 1in \end{equation}
\end{theorem} 

\begin{proof}   The crossing change depicted in Figure \ref{crossingchange} 
has the following effect on the decomposition of elementary tangles.
Suppose that $K = K_1 \cup \ldots \cup K_r$ is such a decomposition.
Denote by
$$ L_i \subset M(X_i,\ul{\mu}_i)^- \times M(X_{i+1},\ul{\mu}_{i+1}), \ i =
1,\ldots, r $$
the corresponding Lagrangian correspondences.  The tangle $K^{\twist}$
is obtained by inserting a half-twist of the $j$ and $j+1$-st markings
after some elementary tangle $K_i$.  We may assume that $K_i$ is a
cylindrical bordism (that is, admits a Morse function with no
critical points) so that the cylindrical Cerf decomposition of
$K^{\twist}$ is obtained from that of $K$ by replacing $K_i$ with a
half-twist.  Similarly for $K^{\cupcap}$, let $K_j^{\cap},K_j^{\cup}$
denote the tangles corresponding to the cap and cup of the $j$-th and
$j+1$-st strands.  Then
$$ K^{\cupcap} = K_1 \cup \ldots \cup K_{i-1} \cup K_j^{\cap} \cup
K_j^{\cup} \cup K_{i+1} \cup \ldots \cup K_r $$
is a cylindrical Cerf decomposition of $K^{\cupcap}$.  The Lagrangian
correspondence for $K_j^\cup$ is that associated to the coisotropic
submanifold
$$ C_j^\cup = \{ g_{j} g_{j+1} = 1 \} \subset M(X_{i},\ul{\mu}_{i}) $$
where $g_i$ is the holonomy around the $j$-th marking.  Similarly the
correspondence for $K_j^\cap$ is $ C_j^\cap = (C_j^\cup)^t $.  These
correspondences are simply-connected, hence automatically
monotone. Let $\tau_{C_j} \in \Diff(M(X_{i+1},\ul{\mu}_{i+1}))$ denote
the corresponding fibered Dehn twist.  We have
\begin{eqnarray*}
 CF(K;L^-,L^+) &=& CF(L^-,L_1,\ldots,L_r,L^+) \\
 CF(K^{\cupcap};L^-,L^+) &=&
CF(L^-,L_1,\ldots,L_i,C_j^\cap,C_j^\cup,L_{i+1},\ldots,L_r,L^+) \\
 CF(K^{\twist};L^-,L^+) &=& CF(L^-,L_1,\ldots,L_{i-1},(\tau_{C_j}
 \times 1)L_i,L_{i+1},\ldots,L_r, L^+ ) .\end{eqnarray*}
Theorem \ref{tantri} now follows from Theorems \ref{main} and
\ref{halftwist}.\end{proof}

More generally, in higher rank invariants we obtain an exact triangle
for the {\em Khovanov-Rozansky modification}.

\begin{definition}  \label{krmod} Let $G= SU(r)$ for some integer $r \ge 2$. 
\begin{enumerate}
\item {\rm (Admissible labels)} An {\em admissible label} is a
  projection of the barycenter of $\Alc$ onto some face.  In the
  absence of reducibles, given a marked surface $(X,\ul{x})$ with
  admissible labels $\ul{\mu}$ the moduli space $M(X,\ul{\mu})$ of
  bundles with fixed holonomy around the markings $\ul{x}$ in the
  conjugacy classes associated to $\ul{\mu}$ is a smooth, compact,
  monotone symplectic manifold.
\item {\rm (Correspondences for admissible graphs)} Let 
$$K \subset Y := X \times [-1,1]$$ 
be a trivalent graph with admissible labels, where each trivalent
vertex is of the form described in Definition \ref{kr}.  This means
that $K$ is made up of a finite number of {\em vertices} that are
points in $Y$, and {\em edges} that are embedded, compact one-manifold
with boundary.  The endpoints of the edges are either vertices or
points on the boundary of $Y$.  Trivalence means that each vertex is
required to lie in the boundary of exactly three edges.  Each edge is
labelled either by $\nu_k^1$ or $\nu_k^2$, so that at any vertex the
labels are $\nu_k^1,\nu_k^1,\nu_k^2$.  See Figure \ref{elemg} where
the squiggly edge represents a vertex labelled $\nu_k^2$.

\begin{figure}[h]
\includegraphics[height=2in]{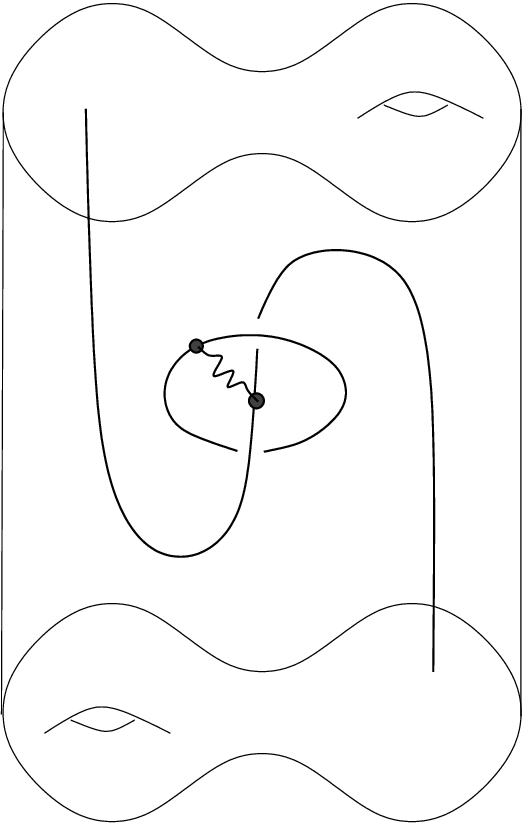}
\caption{Trivalent graph in a bordism}
\label{elemg} 
\end{figure}

\item {\rm (Decomposition into elementary graphs)} Let 
$$ K = K_1 \cup
  \ldots \cup K_e$$ 
be a decomposition into {\em elementary graphs} admitting cylindrical
Morse functions with at most one critical point or vertex.  Each
elementary graph $K_i$ defines a smooth Lagrangian correspondence with
admissible brane structure $L_i$ consisting of bundles that extend
over the interior, see \cite{fieldb}, and so a generalized Lagrangian
correspondence $\ul{L} = (L_1,\ldots, L_e)$.  The functor
$$\Phi(K) : \Fuk^{\sharp}( M(X_-,\ul{\mu}_-)) \to \Fuk^{\sharp}(
M(X_+,\ul{\mu}_+)) $$
obtained by composing the functors $\Phi(L(K_i))$ is independent of
the choice of decomposition into elementary graphs \cite{fieldb}, up
to \ainfty homotopy.
\item {\rm (Khovanov-Rozansky modification of a graph)} Suppose $K
  \subset X \times [-1,1]$ is a trivalent graph with admissible
  labels.  Let $(K_1,\ldots,K_e)$ be a cylindrical Cerf decomposition
  of $K$ and $ K \cap (X \times \{ b_i \}) $ a slice such that two
  points have the same label $\nu_k^1$ from \ref{kr}.  We obtain a new
  trivalent graphs $K^{\twist},K^{\kr}$ by inserting a half-twist,
  respectively inserting two new vertices as shown in Figure
  \ref{gentri}.  Here the intermediate edge represented by a squiggle
  is labelled $\nu_k^2$ from \ref{kr}.
\end{enumerate} 
\end{definition}

\begin{figure}[h]
\includegraphics{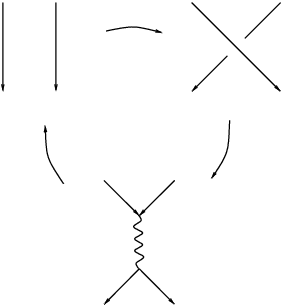}%
\caption{Exact triangle for a crossing change, $G=SU(r)$ case}
\label{gentri}
\end{figure}

 By Theorems \ref{main} and \ref{halftwistr},

\begin{theorem}  {\rm (Exact triangle for a Khovanov-Rozansky modification)}  
Suppose that $G = SU(r)$, $(X,\ul{\mu})$ is a marked surface with
$\mu_i = \mu_j = \nu_k^1$, with $\nu_k^1 =\hh(\omega_k +
\omega_{k+1})$.  Let $\nu_k^2 = \hh(\omega_{k+2} + \omega_{k})$ and
suppose that $K,K^{\twist}, K^{\kr}$ are as in Definition \ref{krmod},
and $L^\pm$ are Lagrangian branes such that the sum of the disk
invariants for $\ul{L}(K), L^-,L^+$ is zero.  There exists an exact
triangle
\vskip .1in
\begin{equation} \label{tri2d}
 \tri{\hskip
 -1in$HF(K;L^-,L^+)$}{$HF(K^{\kr};L^-,L^+).$}{$HF(K^{\twist};L^-,L^+)$}
 \end{equation}

\end{theorem} 

Note that this generalizes the $SU(2)$ exact triangle since if $k = 2$
then $\omega_0 = \omega_2 = 0$.  The exact triangle of
Khovanov-Rozansky \cite{kr:ma} has a similar form.  The theories for
the other standard markings will not in general have surgery exact
triangles of this form, since the corresponding symplectomorphisms are
not, in general, Dehn twists. It would be interesting to understand
whether there is a replacement for the surgery exact triangle in these
more general cases.

\section{Fukaya versions of the exact triangle} 
\label{fukayaversion} 

 In \cite{Ainfty} the authors constructed \ainfty functors for
 Lagrangian correspondences between Fukaya categories.  The gluing
 results necessary for the construction of the \ainfty functors for
 Lagrangian correspondences are proved in \cite{mau:gluing}.  Applied
 to the Lagrangian correspondences arising from moduli spaces of flat
 bundles one obtains a (partial) \ainfty-category-valued field theory.
 We now explain the Fukaya-categorical versions of the exact triangle
 for fibered Dehn twist.

\subsection{Open Fukaya-categorical version}
\label{fukaya}

Let $M$ be a symplectic background.  We take $\Fuk(M)$ to be the
\ainfty category whose objects are admissible monotone branes with
torsion fundamental group and minimal Maslov number at least three.
That is, we disallow the minimal Maslov number two case.  Morphisms
are Floer cochain groups and composition maps count (perturbed)
holomorphic polygons with boundary on the given Lagrangians.  In
\cite{Ainfty} we defined a similar category $\GFuk(M)$ whose objects
are generalized Lagrangian submanifolds.  Let $D^\flat \GFuk(M)$
denote its bounded derived category, as defined by Kontsevich, see
\cite{se:bo}.  Given any Lagrangian correspondence $L_{01} \subset
M_0^- \times M_1$ with admissible brane structure, \cite{Ainfty}
constructs an \ainfty functor
$$ \Phi(L_{01}) : \GFuk(M_0) \to \GFuk(M_1) .$$
The functor $\Phi(L_{01})$ is defined on the level of objects by
concatenating $L_{01}$ to the sequence, and on the level of morphisms
by counting holomorphic quilts.  In particular, given a spherically
fibered coisotropic $M \supset C \to B$ of codimension $c \ge 2$ with
admissible brane structure, one obtains \ainfty functors
$$ \Phi(C): \GFuk(M) \to \GFuk(B), \quad \Phi(C^t): \GFuk(B) \to
\GFuk(M) .$$
\begin{theorem}  \label{fukthm} {\rm (Exact triangle in the 
derived Fukaya category)} Let $(M,\omega)$ be a symplectic background,
  $L \subset M$ an admissible Lagrangian brane, and $\iota: C \to M$ a
  coisotropic submanifold of codimension at least $2$ equipped with
  admissible brane structure spherically fibered over a symplectic
  manifold $B$.  Let $\tau_C : M \to M$ denote a fibered Dehn twist
  along $C$.  There is an exact triangle in $D^\flat \GFuk(M)$ of the
  form
$$ \tri{\hskip .1in $L$}{\hskip -.0in $\Phi(C)\Phi(C^t)
    L$.}{$\tau_C L$}
$$
\end{theorem}

\noindent The morphisms in the triangle Theorem \ref{fukthm} are Floer
cochains defined as relative invariants associated to quilted surfaces
that are variations of the surfaces involved in the proof of Theorem
\ref{main}.  The first map
$$f = C\Phi_S \in CF(L,C^t,C,L)[\dim(B)] = \Hom( \Phi(C)
\Phi(C^t)L,L) $$
is obtained by counting the elements of the moduli spaces $\M_1(x), x
\in \cI(L,C^t,C,L)$ associated to the quilted surface $S$ shown in
Figure \ref{first} with Lagrangian boundary and seam conditions $L,C$:
$$ f = \sum_{\substack{x \in \cI(L,C^t,C,L) \\ u \in \M_1(x)_0}} o(u) \bra{x} .$$
\begin{figure}[ht]
\begin{picture}(0,0)%
\includegraphics{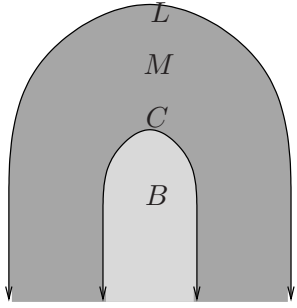}%
\end{picture}%
\setlength{\unitlength}{4144sp}%
\begingroup\makeatletter\ifx\SetFigFont\undefined%
\gdef\SetFigFont#1#2#3#4#5{%
  \reset@font\fontsize{#1}{#2pt}%
  \fontfamily{#3}\fontseries{#4}\fontshape{#5}%
  \selectfont}%
\fi\endgroup%
\begin{picture}(1749,1805)(1601,-1591)
\put(2438,-216){\makebox(0,0)[lb]{{{{$M$}%
}}}}
\put(2479, 95){\makebox(0,0)[lb]{{{{$L$}%
}}}}
\put(2448,-998){\makebox(0,0)[lb]{{{{$B$}%
}}}}
\put(2452,-531){\makebox(0,0)[lb]{{{{$C$}%
}}}}
\end{picture}%
\caption{Quilted surface defining the morphism from $\Phi(C) \Phi(
  C^t) L$ to $L$}
\label{first}
\end{figure}

The second map in the exact triangle is the relative invariant for the
standard Lefschetz-Bott fibration with a single end and Lagrangian
boundary condition $L$ in Figure \ref{second}. That is, if $\M_2(y)_0$
denotes the zero-dimensional component of the moduli space of
pseudoholomorphic sections then the map is
\begin{equation} \label{kdef} k = \sum_{\substack{y \in \cI(\tau_C^{-1}L,L) \\ u \in \M_2(y)_0}} o(u) \bra{y}
.\end{equation}
\begin{figure}
\begin{picture}(0,0)%
\includegraphics{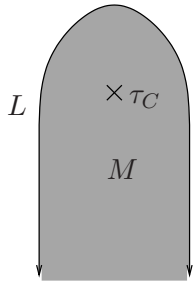}%
\end{picture}%
\setlength{\unitlength}{4144sp}%
\begingroup\makeatletter\ifx\SetFigFont\undefined%
\gdef\SetFigFont#1#2#3#4#5{%
  \reset@font\fontsize{#1}{#2pt}%
  \fontfamily{#3}\fontseries{#4}\fontshape{#5}%
  \selectfont}%
\fi\endgroup%
\begin{picture}(1032,1674)(1471,-1873)
\put(1981,-1254){\makebox(0,0)[lb]{{$M$}%
}}
\put(1386,-871){\makebox(0,0)[lb]{{$L$}%
}}
\put(2116,-826){\makebox(0,0)[lb]{$\tau_C$}%
}
\end{picture}%

\caption{Lefschetz-Bott fibration defining the morphism from
  $\tau_C^{-1} L$ to $L$}
\label{second}
\end{figure}

\begin{proposition} \label{fukver} 
The morphism $k \in \Hom(L, \tau_C L)$ of \eqref{kdef} induces an
isomorphism of the mapping cone $\Cone(f: \Phi(C) \Phi(C^t) L \to L)$
with $\tau_C L$.
\end{proposition} 

\noindent The proof of Proposition \ref{fukver} depends on the
following lemma, whose proof is left as an exercise (c.f. \cite[Lemma
  2.6]{se:ho}).

\begin{lemma} \label{xyz}  {\rm (Sufficient condition for a morphism from a mapping
cone to be an isomorphism)} Let $\cC$ be a c-unital $A_\infty$
  category.  Let $X,Y,Z$ be objects of $\cC$ and
$$ f \in \Hom^0(X,Y) .$$
Any pair 
$$ k \in \Hom^0(Y,Z), \ \  h \in \Hom^{-1}(X,Z) $$
satisfying 
\begin{equation} \label{chain}
 \mu^1(k) = 0, \ \ \ \mu^1(f) = 0, \ \ \ 
\mu^1(h) + \mu^2(f,k) = 0 
\end{equation}
defines a morphism $\Cone(f) \to Z$ in $D^\flat(\cC)$.  This is an
isomorphism if for all $W \in \Ob(\cC)$, the complex
$$ \Hom(W,X)[2] \oplus \Hom(W,Y)[1] \oplus \Hom(W,Z) $$
with differential
\begin{equation} \label{differ}
 (a,b,c) \mapsto (\mu^1(a), \mu^1(b) + \mu^2(a,f),
\mu^1(c) + \mu^2(b,k) + \mu^2(a,h) + \mu^3(a,f,k)) 
\end{equation}
is acyclic.
\end{lemma}

\begin{proof}[Proof of Proposition \ref{fukver}] 
 Consider the relative cochain-level invariant associated to the surface
 in the center in Figure \ref{hpic}.
\begin{figure}[h]
\begin{picture}(0,0)%
\includegraphics{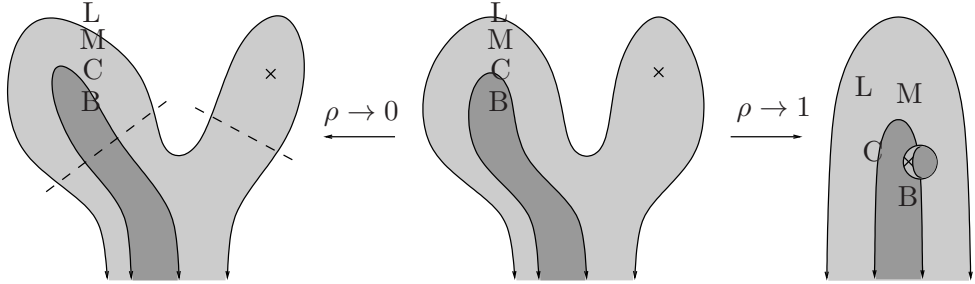}%
\end{picture}%
\setlength{\unitlength}{4144sp}%
\begingroup\makeatletter\ifx\SetFigFontNFSS\undefined%
\gdef\SetFigFontNFSS#1#2#3#4#5{%
  \reset@font\fontsize{#1}{#2pt}%
  \fontfamily{#3}\fontseries{#4}\fontshape{#5}%
  \selectfont}%
\fi\endgroup%
\begin{picture}(5795,1614)(698,-1482)
\put(3597,-264){\makebox(0,0)[lb]{{{C}%
}}}
\put(3597, 75){\makebox(0,0)[lb]{{{L}%
}}}
\put(3578,-87){\makebox(0,0)[lb]{{{M}%
}}}
\put(3588,-455){\makebox(0,0)[lb]{{{B}%
}}}
\put(6027,-409){\makebox(0,0)[lb]{{{M}%
}}}
\put(6037,-1025){\makebox(0,0)[lb]{{{B}%
}}}
\put(5779,-362){\makebox(0,0)[lb]{{{L}%
}}}
\put(5825,-751){\makebox(0,0)[lb]{{{C}%
}}}
\put(1159,-264){\makebox(0,0)[lb]{{{C}%
}}}
\put(1159, 75){\makebox(0,0)[lb]{{{L}%
}}}
\put(1140,-87){\makebox(0,0)[lb]{{{M}%
}}}
\put(1149,-455){\makebox(0,0)[lb]{{{B}%
}}}
\put(2604,-543){\makebox(0,0)[lb]{{{$\rho \to 0$}%
}}}
\put(5071,-520){\makebox(0,0)[lb]{{{$\rho \to 1$}%
}}}
\end{picture}%
\caption{Null-homotopy for the composition $\Phi(C) \Phi(C^t) L$ to
  $\tau_C L$}
\label{hpic}
\end{figure}
We consider a family of deformations of this surface depending on a
parameter $\rho$ as follows.  As $\rho \to 1$, deform the glued
surface so that a disk with values in $E_{C,r}$ bubbles off.  By the
proof of Proposition \ref{pinch}, the relative invariant for the
picture on the right corresponding to $\rho = 1$ is zero on the
cochain level for $r$ sufficiently large.  As $\rho \to 0$, we pinch
off a pair of pants as in the left side of Figure \ref{hpic}.  Let
$$\widetilde{\M}(x) = \bigcup_{\rho \in [0,1]} \{ \rho \} \times
\widetilde{\M}^\rho(x) $$
denote the parametrized moduli space for this deformation consisting
of pairs of a parameter $\rho$ and a holomorphic quilt for the quilt
corresponding to the parameter.  Standard transversality arguments
show that $\widetilde{\M}(x)$ is smooth for generic choices of perturbation
data.  The monotonicity conditions imply the absence of sphere and
disk bubbling, hence compactness of the moduli space.  The boundary of
the moduli space admits a natural identification
$$ \partial \widetilde{\M}(x)_1 \cong \bigcup_{y,z} (\M_1(y)_0 \times \M_2(z)_0
\times {\M}(x,y,z)_0) \cup \bigcup_{y} (\widetilde{\M}(y)_0 \times
       {\M}(x,y)_0).
$$
Here ${\M}(x,y),{\M}(x,y,z) $ are the moduli spaces for the $2$
resp.\ $3$ marked disk, counted by the compositions $\mu^1,\mu^2$.
The first part consists of the $\rho = 0$ boundary of $\widetilde{\M}(x)_1$,
and corresponds to $\mu^2(f,k)$. The other boundary components at
$\rho \in (0,1)$ are formed by splitting off Floer trajectories $v \in
{\M}(x,y)$ for $L,C^t,C,L$.  Define
$$ h \in CF(L,C^t,C,L), \ \ h = \sum_{(u,\rho) \in \widetilde{\M}(y)_0}
o(u) q^{A(u)} \bra{y} .$$
Then by counting the ends of the one-dimensional component of the
moduli space we obtain
$$ 0 = \sum_{\substack{x \in \cI(L,C^t,C,L) \\ (u,\rho) \in \partial
  \widetilde{\M}(x)_1}} o(u) q^{A(u)} \bra{x} = \mu^2(f,k) + \mu^1(h) $$
as claimed.

Now let $L^1$ be another object in $D^\flat \Fuk^{\sharp}(M)$, for
simplicity unquilted.  Acyclicity of the differential \eqref{differ}
is shown as follows.  It suffices to prove acyclicity with $L$
replaced with $\tau_C^{-1}L$.  The terms of lowest order in $q$ are
$\mu^2(f,a)$ and $\mu^2(k,b)$.  As in Section \ref{exact}, the leading
term of $\mu^2(f,a)$ corresponds to the canonical injection $ \cI(L^1,
C^t,C,\tau_C^{-1} L) \to \cI(L^1,\tau_C^{-1} L) $.  On the other hand,
the leading term of $\mu^2(k,b)$ corresponds to the canonical
injection $ \cI(L^1,L) \to \cI(L^1,\tau_C^{-1} L) .$ As before, the
lowest order terms in complex are acyclic, after a small shift in the
$\R$-degrees of the generators.  Filtering the complex by energy shows
that entire complex is acyclic. An application of Lemma \ref{xyz}
completes the proof of Proposition \ref{fukver}.
\end{proof} 

Theorem \ref{fukthm} follows by taking the long exact sequence
associated to the mapping cone in Proposition \ref{fukver}.

\subsection{Periodic Fukaya-categorical version} 
\label{closedfukaya}

In this section we discuss a version of the triangle taking values in
$D^\flat \GFuk(M,M)$, the bounded derived Fukaya category of
generalized Lagrangian correspondences from $M$ to $M$.  Recall from
\cite{Ainfty} that the empty correspondence $\emptyset$ from $M$ to
$M$ considered as a sequence of length zero induces the identity
$$\Phi(\emptyset) = \Id: \GFuk(M) \to \GFuk(M) $$
on the Fukaya category; indeed, labelling each seam by the empty set
has the effect of ``removing the seam''. 

\begin{theorem}  {\rm (Exact triangle in the derived Fukaya category 
of correspondences)} Let $(M,\omega)$ be a symplectic background, and
  $i: C \to M$ a coisotropic submanifold of codimension at least two,
  whose null foliation $p: C \to B$ is spherically fibrating over a
  manifold $B$, equipped with an admissible brane structure.  Let
  $\tau_C : M \to M$ denote a fibered Dehn twist along $C$.  There is
  an exact triangle in $D^\flat \GFuk(M,M)$ of the form
\vskip .05in
$$ \tri{$\emptyset$} {$C^t {\sharp} C[\dim(B)].$}{\hskip -.0in
  $\on{\graph}(\tau_C) $}
$$
\end{theorem}

\begin{proof}[Sketch of proof]
The proof is the similar to that of Theorem \ref{fukver}, replacing
the strip-like ends with cylindrical ends.
\begin{figure}[ht]
\includegraphics{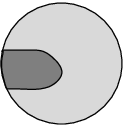}%
\caption{Quilt defining the morphism from $(C^t,C)[\dim(B)]$ to
  $\emptyset$}
\label{quiltcap6}
\end{figure}
The morphism from $C \sharp C^t$ to $\emptyset$ is obtained from the
quilted cap $\ul{S} =(S_M, S_B)$ in Figure \ref{quiltcap6}, where
\begin{itemize} 
\item the outer circle represents a quilted cylindrical end with seams
  $C^t,C$, 
\item the lightly shaded patch $S_M$ maps to $M$ and
\item the darkly shaded patch $S_B$ maps to $B$.
\end{itemize}
\begin{figure}[ht]
\includegraphics{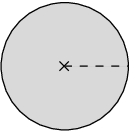}%
\caption{Quilt defining the morphisms from $\emptyset$ to
$\graph(\tau_C)$}
\label{quiltcap7}
\end{figure}
The morphism from $\emptyset$ to $\graph(\tau_C)$ is defined by the
Lefschetz-Bott fibration over the cap shown in Figure \ref{quiltcap7},
where the outer circle represents a cylindrical end with monodromy
around the end given by $\tau_C$.  The composition of the two maps is
defined by the surface shown in Figure \ref{quiltcap8}.  By deforming
the singularity on the surface onto a disk with boundary condition in
$C$, one obtains a null-homotopy of the composition.  This
null-homotopy defines, as in Lemma \ref{xyz}, an isomorphism of the
mapping cone $\Cone(C^t {\sharp} C[\dim(B)] \to \emptyset)$ with
$\tau_C$.
\begin{figure}[h]
\includegraphics{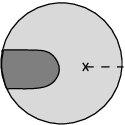}
\caption{Composition of quilts}
\label{quiltcap8}
\end{figure}
Applying the
functor
$$\GFuk(M,M) \to \Fun(\GFuk(M),\GFuk(M))$$
constructed in \cite{Ainfty} this leads to an exact triangle in
$\Fun(D^\flat \GFuk(M),D^\flat \GFuk(M))$
$$ \tri{$\on{id}$}
{\hskip -.3in $D^\flat \Phi(C) \circ D^\flat \Phi(C^t)$.}
{\hskip -.1in $D^\flat \Phi(\tau_C) $} 
 $$
That is, there exists an isomorphism in $\Fun(D^\flat \GFuk(M),
D^\flat \GFuk(M))$, 
$$ D^\flat \Phi(\tau_C) \to \Cone( D^\flat \Phi(C) D^\flat \Phi(C^t)
\to \on{id}) .$$
Applying this exact triangle to any object $L$ of $D \GFuk(M)$ leads
to the exact triangle given in Theorem \ref{fukthm}. 
\end{proof}  

\begin{remark} {\rm (\ainfty results for minimal Maslov two) } 
Similar results hold in the case case of minimal Maslov number two for
the \ainfty categories $\GFuk(M,w)$ whose objects $L$ have disk
invariant $w = w(\ul{L})$, counting the number of Maslov index disks
passing through a generic point in $\ul{L}$. See \cite[Section
  4.4]{Ainfty} for more on the disk invariant and Fukaya category
$\GFuk(M,w)$.
\end{remark} 

%\bibliographystyle{plain}
%\bibliography{../Bib/ref}
\def\cprime{$'$} \def\cprime{$'$} \def\cprime{$'$} \def\cprime{$'$}
\def\cprime{$'$} \def\cprime{$'$}
\def\polhk#1{\setbox0=\hbox{#1}{\ooalign{\hidewidth
      \lower1.5ex\hbox{`}\hidewidth\crcr\unhbox0}}} \def\cprime{$'$}
\def\cprime{$'$}

\end{document}